\documentclass[preprint,11pt]{elsarticle}

\usepackage{microtype}
\usepackage{graphicx}
\usepackage{subfigure}
\usepackage{booktabs} 
\usepackage{caption}

\usepackage{float}

\usepackage{amsmath}
\usepackage{amssymb}
\usepackage{mathtools}
\usepackage{amsthm}

\usepackage[utf8]{inputenc} 
\usepackage[T1]{fontenc}    
\usepackage{hyperref}       
\usepackage[capitalize,noabbrev]{cleveref}
\usepackage{url}            
\usepackage{booktabs}       
\usepackage{diagbox}
\usepackage{amsfonts}       
\usepackage{nicefrac}       
\usepackage{microtype}      
\usepackage{xcolor}         
\usepackage{graphicx}
\usepackage{wrapfig}
\usepackage{mathtools}
\usepackage{amsthm}
\usepackage{enumitem}
\usepackage{natbib}
\usepackage{algorithm}
\usepackage{multirow}
\usepackage{xcolor}
\definecolor{myred}{HTML}{800606}
\definecolor{myblue}{HTML}{131D85}
\usepackage{color, colortbl}
\usepackage{pifont}

\theoremstyle{plain}
\newtheorem{theorem}{Theorem}
\newtheorem{proposition}{Proposition}

\usepackage{lineno}
\usepackage{xcolor}

\newcommand{\cD}{\mathcal{D}}
\newcommand{\cE}{\mathcal{E}}

\newcommand{\cL}{\mathcal{L}}

\newcommand{\cS}{\mathcal{S}}

\newcommand{\cU}{\mathcal{U}}

\newcommand{\bN}{\mathbb{N}}

\newcommand{\bR}{\mathbb{R}}

\newcommand{\bV}{\mathbb{V}}

\newcommand{\dV}[1][u]{\cU\left(\Omega;\bR^{d_v}\right)}
\newcommand{\dVt}[1][u]{\cU\left(\left[0,\infty\right)\times\Omega;\bR^{d_v}\right)}

\newcommand{\bx}{\mathbf{x}}
\newcommand{\by}{\mathbf{y}}
\newcommand{\bv}{\mathbf{v}}
\newcommand{\bb}{\mathbf{b}}
\newcommand{\bn}{\mathbf{n}}
\newcommand{\sgn}{\text{sgn}}
\DeclareMathOperator*{\argmin}{\arg\min}

\begin{document}

\begin{frontmatter}


\title{\textit{ReSDF}: Redistancing Implicit Surfaces using Neural Networks}



\author[label1]{Yesom Park}
\ead{yeisom@snu.ac.kr}
\author[label1]{Chang hoon Song}
\ead{goldbach2@snu.ac.kr}
\author[label2]{Jooyoung Hahn}
\ead{jooyoung.hahn@stuba.sk}
\author[label1]{Myungjoo Kang\corref{cor1}}
\ead{mkang@snu.ac.kr}

\cortext[cor1]{corresponding author}
\address[label1]{Department of Mathematical Sciences, Seoul National University, Seoul, Republic of Korea}
\address[label2]{Department of Mathematics and Descriptive Geometry,  Faculty of Civil Engineering, Slovak University of Technology in Bratislava, Slovakia}

\begin{abstract}
This paper proposes a deep-learning-based method for recovering a signed distance function (SDF) of a given hypersurface represented by an implicit level set function. Using the flexibility of constructing a neural network, we use an augmented network by defining an auxiliary output to represent the gradient of the SDF. There are three advantages of the augmented network; (i) the target interface is accurately captured, (ii) the gradient has a unit norm, and (iii) two outputs are approximated by a single network. Moreover, unlike a conventional loss term which uses a residual of the eikonal equation, a novel training objective consisting of three loss terms is designed. The first loss function enforces a pointwise matching between two outputs of the augmented network. The second loss function leveraged by a geometric characteristic of the SDF imposes the shortest path obtained by the gradient. The third loss function regularizes a singularity of the SDF caused by discontinuities of the gradient. Numerical results across a wide range of complex and irregular interfaces in two and three-dimensional domains confirm the effectiveness and accuracy of the proposed method. We also compare the results of the proposed method with physics-informed neural networks approaches and the fast marching method.
\end{abstract}


\end{frontmatter}


\section{Introduction}
The signed distance function (SDF) to a hypersurface $\Gamma\subset \bR^n$, which is the distance to $\Gamma$ in the outer region and the negative of the distance to $\Gamma$ in the inner region, has been crucial in various fields, ranging from computational fluid dynamics \cite{olsson2005conservative, gibou2007level, roget2013wall} to image segmentation \cite{tsai2001model, alvino2007efficient, li2010distance}, 3D shape reconstruction from scattered point data \cite{taubin2012smooth, gropp2020implicit, sitzmann2020metasdf}, architectural geometry \cite{pottmann2010geodesic, novello2022exploring}, and robotic navigation \cite{kimmel1998multivalued, lee2016structured}. After being devised by Osher and Sethian \cite{osher1988fronts}, the level set method, which represents $\Gamma$ as the zero level set of a continuous function $\phi:\bR^n\rightarrow \bR$, provides a numerical and theoretical paradigm of evolving hypersurfaces. To facilitate geometric features such as the normal vector and mean curvature of the interface and to reduce numerical instability, the level set function is preferable to be neither too flat nor steep near its zero contours. Reinitializing it as the SDF has been a common numerical treatment in modelling the motion of dynamic interfaces \cite{sussman1999efficient, strain1999semi, kang2000boundary,  cho2022solving} and shape optimization \cite{osher2001level, allaire2002level, wang2007extended, van2013level}.

There have been several numerical methods to re-distance the given level set function $\phi$. One of the most prominent efforts is built on the fact that the SDF is a solution of partial differential equations (PDEs) \cite{peng1999pde, min2010reinitializing}. Fast marching methods (FMMs) \cite{sethian1996fast,kimmel1998computing,sethian2000fast,hassouna2007multistencils,yang2017highly} and fast sweeping methods \cite{zhao2005fast, qian2007fast, li2008second} recover the SDF as a viscosity solution to the eikonal equation. Sussman et al. \cite{sussman1994level} reformulated the eikonal equation by a pseudo-time-dependent nonlinear hyperbolic PDE. This approach is known to be more suitable than solving directly the eikonal equation in the case of evolving interfaces. Since the SDF \cite{sussman1994level} is obtained by the stationary solution, it is time-consuming and requires a large number of iterations depending on the CFL restriction. A more serious issue is that the zero level set fails to be maintained. Lee et al. \cite{lee2017revisiting} propose a fast method by using the Hopf-Lax formula of the Hamilton-Jacobi equation. In \cite{barles1993front}, it is known that the solution of the Hamilton-Jacobi equation in \cite{lee2017revisiting} is not exactly the SDF. Another approach \cite{belyaev2015variational} employs Varadhan's distance function \cite{varadhan1967behavior}. By the Hopf-Cole formula, it can be transformed into a regularized eikonal equation with an artificial viscosity term.

Inspired by the tremendous success of deep learning in diverse machine learning tasks, such as image classification \cite{simonyan2014very, nath2014survey, he2016deep, choi2019attention, dosovitskiy2020image} and density estimation \cite{kingma2013auto, goodfellow2020generative, song2020score, chen2018neural}, the use of neural networks for solving PDEs has begun to attract significant attention in recent years. Pioneering studies \cite{lee1990neural, lagaris1998artificial} incorporate physical principles into neural networks by directly constructing a loss function as a residual of PDEs and errors of boundary or initial conditions. Building upon these earlier works, Raissi et al. \cite{raissi2019physics} have revisited them by using modern computational tools and provided a framework named physics-informed neural networks (PINNs). The advantage of PINNs is that they can be readily transformed into various problems \cite{jin2021nsfnets, hu2022discontinuity, patel2022thermodynamically} and can treat the PDE in a fully mesh-free and time-continuous manner.
However, they suffer from a challenging optimization landscape \cite{wang2021understanding, krishnapriyan2021characterizing}, as it is difficult for them to learn many multi-scale or complex PDE systems \cite{fuks2020limitations, wang2022and}. Another line of work involves neural operators \cite{lu2019deeponet,li2020fourier,kovachki2021neural}, which unveil physical systems from data by learning implicit solution operator that maps boundary or initial conditions to solutions. They have shown promise in learning complex PDEs \cite{wen2022u, lu2022comprehensive}, but the requirement for large amounts of available data limits application to various of practical problems. In addition, hybrid methods \cite{fang2021high, qiu2021cell, lienen2022learning, karlbauer2022composing} that combine deep learning with well-grounded numerical methods have been studied.

\begin{figure*}[t]
    \centering
    \includegraphics[page=1,width=0.9\textwidth]{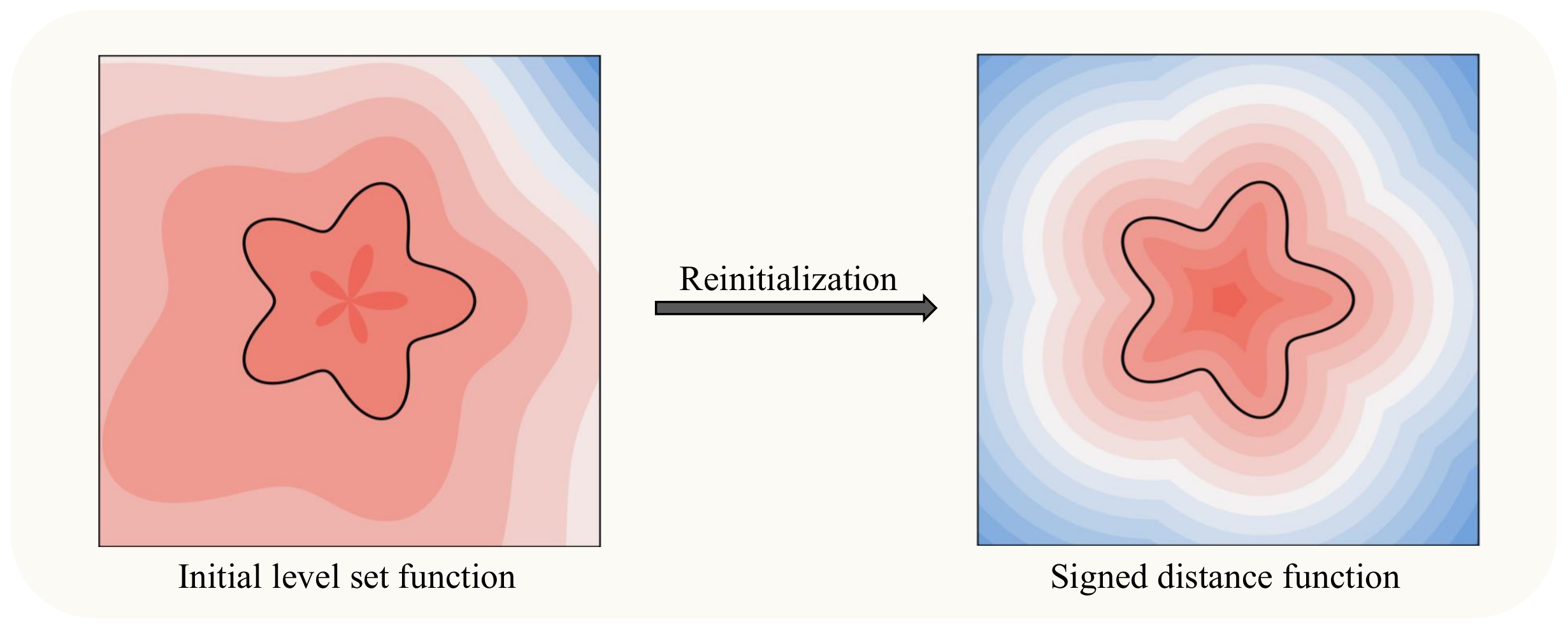}
    \caption{Concept of the level set reinitialization. The interface $\Gamma$ is presented by black solid lines. The left is an iso-contour plot of the given level set function $\phi$ and the right is the iso-contours of the SDF after the reinitialization of $\phi$.}
    \label{fig:concept}
\end{figure*}

In this paper, we propose a novel method, named \textit{ReSDF}, for reconstructing a SDF from a given implicit level set function $\phi$ whose zero level set is an interface $\Gamma$; see Figure \ref{fig:concept}. By exploiting flexibility of network design and optimization objectives, we compute the SDF through two major ideas. Firstly, inspired by the variable-splitting scheme \cite{wang2008new, goldstein2009split}, we introduce an augmented neural network that parametrizes the gradient of the SDF as an auxiliary variable while keeping the number of network parameters. The network is designed so that the approximating SDF accurately distinguishes the interface from the interior and the outer regions and the gradient has a unit norm. Secondly, we propose a training objective consisting of three loss terms. The first loss function comes from the splitting method and it enforces matching vectors between the gradient of the estimated SDF and the auxiliary gradient. The second loss function imposes a geometric fact that the gradient of the SDF defines the shortest path to the interface. The third loss function is devised to alleviate the nonuniqueness of the gradient caused by multiple shortest paths at a singular point of the SDF. We also provide a theoretical validation of the proposed objectives for reinitializing a level set function. Numerical results confirm that proposed loss functions in conjunction with the designed neural network significantly improve the accuracy compared to the existing PINN approach. The \textit{ReSDF} also accurately approximates the SDF for the cases of complex and irregular interfaces without tuning sensitive hyper-parameters. The capability of the model is also tested to estimate the distance function on three-dimensional space without increment in the number of parameters of the network.

The rest of the paper is organized as follows. In Section 2, we present the formulation of the problem and a brief discussion about prior works. In Section 3, we introduce the proposed ReSDF by describing the augmented network and the objective function in detail.
Numerical experiments are presented in Section 4 to demonstrate the effectiveness and accuracy of ReSDF, followed by some concluding remarks given in Section 5.


\section{Previous Works} \label{sec:pre_work}

Let $\Omega\subset \bR^n$ be a domain and $\Gamma\subset \Omega$ be a compact hypersurface implicitly represented by a zero level set of a continuous level set function $\phi:\Omega\rightarrow \bR$. The hypersurface $\Gamma$ divides $\Omega$ into two disjoint open subsets: the outer region $\Omega^+=\left\{\bx\in\Omega \mid \phi\left(\bx\right)>0\right\}$ and the inner region $\Omega^-=\left\{\bx\in\Omega \mid \phi\left(\bx\right)<0\right\}$ satisfying $\Omega \setminus \Gamma =\Omega^+ \sqcup \Omega^-$. From the given function $\phi$, the goal is to find the signed distance function (SDF) $u:\Omega\rightarrow \bR$ that satisfies 
\begin{equation}
u\left(\bx\right)=\begin{cases}
d\left(\bx,\Gamma\right) & \text{in }\Omega^{+}\\
0 & \text{on }\Gamma\\
-d\left(\bx,\Gamma\right) & \text{in }\Omega^{-},
\end{cases}
\end{equation}
where $d\left(\bx,\Gamma\right)=\underset{\by\in\Gamma}{\min}\parallel \bx-\by\parallel$ denotes the standard Euclidean distance function to $\Gamma$. It is a unique viscosity solution to the eikonal equation \cite{crandall1984two}
\begin{align}
\parallel \nabla u \parallel  &= 1, \label{eq:eikonal} \\
\sgn\left(u\right)  &=\sgn\left(\phi\right),  \label{eq:sign_cond}
\end{align}
where $\phi$ is the given level set function and $\sgn$ is the signum function, which takes either $1$, $0$, or $-1$ for points in $\Omega^+$, $\Gamma$, or $\Omega^-$, respectively.

The variational approaches to approximate the distance function~\eqref{eq:eikonal} in \cite{li2005level, xin2012global, alblas2023going} are to minimize the energy functional related to the eikonal equation:
\begin{equation} \label{eq:energy_eikonal}
    \cE\left(u\right)=\int_\Omega \left(\left\Vert \nabla u\right\Vert - 1\right)^2
\end{equation}
in conjunction with the Dirichlet boundary condition $u\left(\Gamma\right)=0$ as a constraint relaxed by a penalty term:
\begin{equation}\label{eq:eik_variational}
\underset{u}{\min} \ \cE\left(u\right) + \int_\Gamma \left\vert u \right\vert.
\end{equation}
Several efforts have also been made to improve the ill-posedness or convergence of the variational problem \eqref{eq:energy_eikonal}, including penality methods introducing external energy functionals for shifting the zero level set toward the target interface \cite{li2010distance} and for preserving the shape of the free interface \cite{basting2013minimization}, and effective splitting schemes \cite{belyaev2015variational, belyaev2020admm} using alternating direction method of multipliers \cite{boyd2011distributed}.

As neural PDE surrogates have proliferated as an impactful area of research, several efforts have been made to reconstruct the SDF using neural networks. Prior studies mostly resort to the steady eikonal equation to find the SDF. Lichtenstein et al. \cite{lichtenstein2019deep} propose a hybrid method that integrates neural networks into the FMM \cite{sethian1996fast}. They replace the local numerical solver in the FMM with a neural network trained from data. The accuracy of the numerical solution is improved by leveraging the expressive power of neural networks. Since the approach requires a wealth of training data with true distance values, it might be difficult to be used in application where lots of true distances are hard to obtain. The performance relies on the amount of available data and would be unfavourable outside of the data on which the network is trained. Another approach directly uses a neural network to parametrize the SDF. Gropp et al. \cite{gropp2020implicit} adopt the PINN approach to learn the SDF from an unorganized cloud of points. They directly use a neural network to parametrize the SDF of the given point cloud by minimizing the eikonal-embedded loss function.

Recent works \cite{fayolle2021signed, bin2021pinneik} leverage a framework of PINN \cite{raissi2019physics}. Similar to the variational approaches \eqref{eq:energy_eikonal}, they convert the problem of solving the eikonal equation \eqref{eq:eikonal} into an optimization problem in which the loss function embeds the knowledge of \eqref{eq:eikonal}:
\begin{equation}\label{eq:pinn_eikonal}
\cL_{Eik}\left(\theta\right)=\frac{1}{\mid \cD \mid}\sum_{\bx\in\cD}\Bigl( \| \nabla u_\theta\left(\bx\right) \| - 1 \Bigl)^2 + \lambda \cL_{R}\left(\theta\right),
\end{equation}
where $u_\theta$ is a neural network parametrized by $\theta$, which is an approximator of the solution to \eqref{eq:eikonal}. The soft penalty term $\cL_{R}$ enforces an additional constraint on the solution such as weight normalization or boundary conditions. The expectation is taken with respect to a collection $\cD$ of scattered collocation points usually chosen by uniform random sampling. The objective of the residual of the eikonal equation characterizes the deviation of $u_\theta$ from the SDF. The trained network $u_\theta\left(\bx\right)$ serves as a suitable approximation of the solution. Fayolle \cite{fayolle2021signed} also suggests an alternative PINN-based approach that relies on $p$-Poisson distances \cite{manfredilimits}. As discussed in the numerical section, we compare the proposed method with the results of the PINN-based approach for irregular and complex interfaces and check the robustness of using various initial level set functions $\phi$. 

\section{Proposed Method} \label{sec:method}
In this section, we propose a learning-based approach to recover the SDF (\textit{ReSDF}) of a given hypersurface implicitly represented by a level set function. In order to increase the expressiveness of the network we use an augmented network that parameterizes the gradient of SDF as an auxiliary output while keeping the number of parameters. Moreoever, novel objectives are designed to exploit a global property and alleviate a singularity of the SDF by harnessing the geometric properties of the SDF.

 \subsection{Augmented Network Representation} \label{sec:net}
 
 \begin{figure*}
    \centering
    \includegraphics[page=1,width=0.97\textwidth]{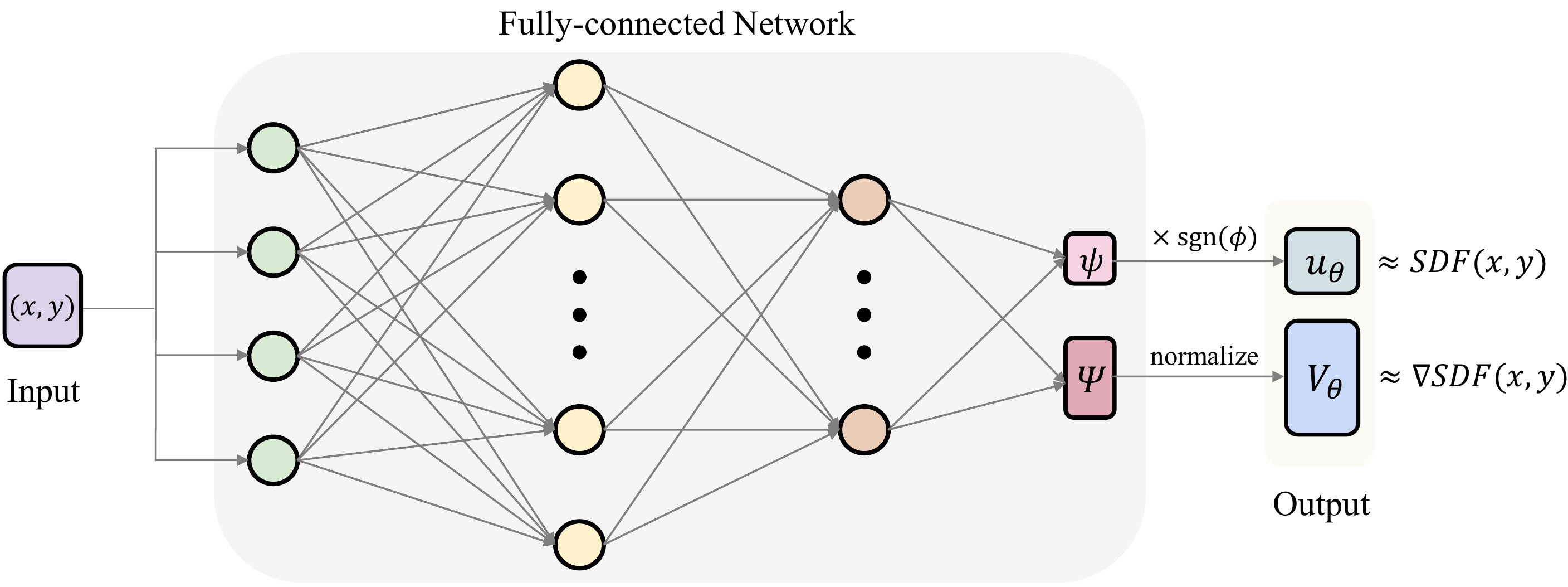}
    \caption{The \textit{ReSDF} model architecture is visualized. The input $\left(x,y\right)$ passes through the shared shallow fully-connected network on the shaded region. The \textit{ReSDF} parametrizes $u_\theta=\sgn\left(\phi\right)\psi$ as an ansatz of the SDF together with the auxiliary output $V_\theta=\Psi / \mid \Psi\mid$ which approximates the gradient of the SDF.}
    \label{fig:network}
\end{figure*}

In this section, inspired by variable-splitting methods \cite{wang2008new, goldstein2009split} in optimization, an augmented network structure is considered to separately parametrize the gradient of the SDF as an auxiliary variable. The augmented network has two outputs that satisfy the following characteristics: (i) the primary output $u_\theta$ approximates the SDF and it automatically vanishes on the interface as well as satisfies the sign condition \eqref{eq:sign_cond} and (ii) the gradient field represented by the auxiliary output $V_\theta$ has a unit norm; see Figure \ref{fig:network}.

We parameterize a single network $N_\theta :\bR^n\rightarrow\bR \times \bR^{n}$ to output $\psi_\theta\left(\bx\right)\in \bR$ together with an auxiliary value $\Psi_{\theta}\left(\bx\right)\in\bR^n$
\begin{equation} \label{eq:AugNet}
N_\theta\left(\bx\right)=\Bigl( \psi_\theta\left(\bx\right), \Psi_{\theta}\left(\bx\right) \Bigl).
\end{equation}
because it is more efficient to approximate the SDF and its gradient using a single common network rather than learning with two individual networks. The first scalar function $\psi_\theta$ and the other vector-valued component $\Psi_\theta$ are employed to parametrize $u_\theta$ and $V_\theta$, respectively, through a  designed architecture in Figure \ref{fig:network}. The neural perceptron $\tilde{\Psi}_\theta$ is defined using a multi-layer fully-connected neural network:
\[
\tilde{\Psi}_\theta\left(\bx\right)= W\left(f_L\circ \cdots \circ f_0\left(\bx\right)\right) + \bb,\ \bx\in\bR^n,
\]
where $L\in\bN$ is a given depth, $W\in\bR^{n+1\times d_L}$ is a weight of the output layer, $\bb\in \bR^{n+1}$ is a final bias vector and the perceptron (also known as the hidden layer) $f_{\ell}:\bR^{d_{\ell-1}} \rightarrow \bR^{d_{\ell}}$ is defined by 
\[
f_{\ell}\left(\by\right)=\sigma\left(W_\ell \by + \bb_\ell\right),\ \by\in\bR^{d_{\ell-1}},\ \text{for all } \ell=0,\dots,L,
\]
for $W_\ell\in \bR^{d_{\ell}\times d_{\ell-1}}$, $\bb_\ell\in\bR^{d_\ell}$, and a non-linear activation function $\sigma$. The dimensions $d_\ell$ of the hidden layers are also called by the width of the network. As the input to the network is a point $\bx$ in $\Omega$, the input dimension of the input layer $f_0$ is $d_0=n$.
The output layer recovers the $(n+1)$-dimensional output values using the matrix product between the output of the final hidden layer $f_L$ and $W$ in addition to a bias vector $\bb$. A shorthand notation $\theta$ is used for all parameters in the weights $\left\{W, W_0,\cdots, W_L\right\}$ and biases $\left\{\bb, \bb_0,\cdots ,\bb_L\right\}$.
Given the current parameter configuration, the parameters $\theta$ are successively adapted by minimizing an assigned loss function explained in Section \ref{sec:objective}.

\paragraph{\textbf{Representation of SDF}}
We aim to adjust the network output $\psi_\theta$ so that the approximated SDF satisfies the sign condition \eqref{eq:sign_cond}. Thanks to the level set representation, the sign of the SDF is easily obtained. Similar to \cite{fayolle2021signed}, we treat the condition on the sign of the SDF as a hard constraint by parameterizing the primary output as
\begin{equation}\label{eq:net_u}
u\left(\bx\right)\approx u_{\theta}\left(\bx\right)\coloneqq \cS\left(\phi\left(\bx\right)\right)\left|\psi_{\theta}\left(\bx\right)\right|,
\end{equation}
where the quantity $\cS\left(\phi\left(\bx\right)\right)$ is a smoothed sign function of $\phi\left(\bx\right)$:
\[
\cS\left(y\right) = \gamma\tanh\left(\beta y\right),\ y\in\bR
\]
with a scaling factor $\gamma = 1.4\approx 1/\tanh 1$ and a smoothing parameter $\beta = 70$ are fixed in all examples. Clearly, the designed artificial neural network as an ansatz of the SDF automatically brings the sign condition, $\sgn\left(u_\theta\right)=\sgn\left(\phi\right)$. In particular, $u_\theta$ vanishes on the target interface $\Gamma$. Therefore, the zero-level set of the approximated SDF can accurately preserve the interface without smearing out.

\paragraph{\textbf{Approximation of the gradient field}}
In contrast to the existing approaches, which treat the unit norm property of the gradient as a soft constraint, we design the augmented output
\[
\nabla u\left(\bx\right) \approx V_\theta\left(\bx\right) : \bR^n\rightarrow S^{n-1}
\]
to automatically satisfy the unit norm constraint.
To this end, we define a neural network that lies in the unit sphere $S^{n-1}$ by normalizing the auxiliary output $\Psi$:
\begin{equation}\label{eq:netV}
V_\theta\left(\bx\right) \coloneqq \frac{\Psi_\theta\left(\bx\right)}{\mid\Psi_\theta\left(\bx\right)\mid}.
\end{equation}
In other words, $\Psi_\theta$ learns only the direction of the gradient of SDF and it is automatically adjusted to the unit length by normalization.

To sum up, the SDF and its gradient are estimated from the output of the augmented network \eqref{eq:AugNet} as 
\begin{align}
u\left(\bx\right) & \approx u_{\theta}\left(\bx\right)\coloneqq\cS\left(\phi\left(\bx\right)\right)\left|\psi_{\theta}\left(\bx\right)\right|, \\
\nabla u\left(\bx\right) & \approx V_\theta\left(\bx\right) \coloneqq \frac{\Psi_\theta\left(\bx\right)}{\mid\Psi_\theta\left(\bx\right)\mid}.
\end{align}
The resultant $u_\theta$ vanishes on $\Gamma$ and $V_\theta$ has unit length. Moreover, it is worth emphasizing that $u_\theta$ and $V_\theta$ share the same weights and biases from the input layer till the last hidden layer. This enables a reduction in the number of parameters and faster training than using two separate networks. 

\subsection{Design of loss functions}\label{sec:objective}
A formulation of the loss function is another key ingredient to achieve effective learning using two outputs of the augmented networks \eqref{eq:AugNet}. In this section, we explain the following objective function consisting of three loss terms:
\begin{equation}\label{eq:total}
\cL(u, V) = 
\int_\Omega \left\Vert \nabla u\left(\bx\right) -V\left(\bx\right)\right\Vert^2 
+ \int_\Omega \left\vert \phi\left( \bx-u\left(\bx\right)V\left(\bx\right) \right)\right\vert ^2
+ \int_\Omega \left\Vert V\left(\bx\right)- V\left( \bx-\eta u\left(\bx\right)V\left(\bx\right) \right)\right\Vert ^2,
\end{equation}
where $\eta > 0$ is a fixed constant. The main characteristics are summarized:
\begin{itemize}
 \item  The first term enforces the vector matching between the gradient of the primary output and the auxiliary output.
 \item The second term imposes a global property using the fact that the SDF and its gradient determine the shortest path to the interface.
 \item The third term alleviates a singularity where the SDF is not differentiable.
\end{itemize}
In \ref{appen:analysis}, it is proved that the SDF is a minimizer of the functional $\cL (u, V)$. 
The Euler-Lagrange equation for the functional $\cL (u, V)$ is nonlinear and it is not straightforward to find a minimizer by conventional methods. It highlights one of the main advantages of the deep learning approach to minimize such a complex loss function.

\paragraph{\textbf{Gradient matching objective}} The gradient matching (GM) objective directly enforces to reduce a gap between the gradient of the primary output $u_\theta$ and the unit-norm auxiliary output $V_\theta$ in \eqref{eq:AugNet}:
\begin{equation}\label{eq:loss1}
   \cL_{\text{GM}}\left(\theta\right)=\frac{1}{\mid \cD \mid}\sum_{\bx\in\cD} \| \nabla u_\theta\left(\bx\right) - V_\theta\left(\bx\right)\|^2.
\end{equation}
The gradient of $u_\theta$ in \eqref{eq:loss1} by auto-differentiation library (\texttt{autograd}) \cite{paszke2017automatic} calculates the exact derivatives of the networks. In the view of the variable-splitting method, the variational problem \eqref{eq:eik_variational} is reformulated by
\[
\underset{u,\ V}{\min}\int_\Gamma \left\vert u \right\vert + \int_\Omega \left(\left\Vert V \right\Vert - 1\right)^2,\ \text{subject to } \nabla u=V.
\]
By regularizing the constraint as a penalty term, an unconstrained problem is presented:
\begin{equation}\label{eq:HQS}
\underset{u,\ V}{\min}\int_\Gamma \left\vert u \right\vert + \int_\Omega \left(\left\Vert V \right\Vert - 1\right)^2 + \lambda\int_\Omega \left\Vert \nabla u - V\right\Vert^2,
\end{equation}
where $\lambda>0$ is a regularization parameter. Since $u_\theta$ automatically vanishes on $\Gamma$ in the proposed network structure in Section \ref{sec:net}, the first loss term in \eqref{eq:HQS} is no longer needed. The network structure in which the $V_\theta$ is designed to have a unit norm allows us to exclude the second term in \eqref{eq:HQS}. This leads us to consider $\cL_{\text{GM}}$ \eqref{eq:loss1}. Since the existing PINN methods are used to learn the eikonal equation with a single network output, the network must learn the direction of the gradient while matching its norm to be one. In the matching loss function \eqref{eq:loss1}, $\Psi_\theta$ learns the direction of the gradient of the SDF and $\nabla u_\theta$ is matched pointwise to the normalized vector $V_\theta$ of $\Psi_\theta$.

\paragraph{\textbf{Shortest path objective}}
The shortest path (SP) objective explains how the gradient of SDF makes the shortest path to the interface $\Gamma$:
\begin{equation}\label{eq:loss2}
   \cL_{\text{SP}}\left(\theta\right)=\frac{1}{\mid \cD \mid}\sum_{\bx\in\cD}\left\vert \phi\left( \bx-u_\theta\left(\bx\right)V_\theta\left(\bx\right) \right)\right\vert ^2.
\end{equation}
The loss function $\cL_{\text{SP}}$ specifies a relations between the interface $\Gamma$ and the points away from $\Gamma$ by using the gradient of the SDF. The following proposition delivers a justification of using \eqref{eq:loss2} in the view of geometric property of the SDF.
\begin{proposition}\label{prop:SDF} \cite{dugundji1966topology,federer2014geometric}
Let $\Gamma\subset \Omega$ be a compact hypersurface of a domain $\Omega\subset \bR^n$. Suppose $u:\Omega\rightarrow\bR$ is the signed distance function to $\Gamma$. Then, $u$ is differentiable except on a set of zero measure. Moreover, for any point $\bx\in\Omega$ where $u\left(\bx\right)$ is differentiable, it satisfies
\begin{equation}\label{eq:condition1}
\bx_\Gamma = \bx-u\left(\bx\right)\nabla u\left(\bx\right)\in\Gamma,
\end{equation}
where $\bx_\Gamma = \argmin_{\mathbf{y} \in \Gamma} d(\bx, \Gamma)$.
\end{proposition}
The property \eqref{eq:condition1} describes the shortest path of a point $\bx$ to the target interface. The shortest path explains that a point $\bx\in\Omega$ came from a point on the interface $\bx_\Gamma = \argmin_{\mathbf{y} \in \Gamma} d(\bx, \Gamma)$, that is, the closest point to $\bx$ on $\Gamma$. Note that $\cL_{\text{GM}}$ \eqref{eq:loss1} matches a local relation between the gradient of $u_\theta$ and $V_\theta$ pointwisely. On the other hand, the shortest path loss function $\cL_{\text{SP}}$ dictates non-local relation between points and directly enforces $u_\theta$ to satisfy the geometric property along with $V_\theta$. With a help of a given function $\phi$ whose zero level set represents $\Gamma$, we have
\begin{equation} \label{eq:back_to_interface}
    \bx-u\left(\bx\right)\nabla u\left(\bx\right)\in\Gamma \iff \phi\left( \bx-u\left(\bx\right)\nabla u\left(\bx\right) \right) = 0,
\end{equation}
then $\left\vert\phi\left( \bx-u_\theta\left(\bx\right)\nabla u_\theta\left(\bx\right) \right)\right\vert$ in the loss function \eqref{eq:loss2} is a reasonable choice and train it to approach zero. Note that this loss function requires the derivative of $u_\theta$ with respect to the spatial variable $\bx$ and the computed gradient is multiplied by $u_\theta$ again. It is a challenging optimization because the gradient calculation is deep and the loss landscape is complex. However, as we have $V_\theta$ as an approximation of the gradient of the SDF, we replace $\nabla u_\theta$ with $V_\theta$ in \eqref{eq:back_to_interface}. By replacing $\nabla u_\theta$ with $V_\theta$, the chain rule for computing the gradients of the objective \eqref{eq:loss2} is simpler than using $\nabla u_\theta$.

\paragraph{\textbf{Regularizing singularity objective}}
The regularizing singularity (RS) objective resolves a singularity of the SDF:
\begin{equation}\label{eq:loss3}
\cL_{\text{RS}}\left(\theta\right)=\frac{1}{\mid \cD \mid}\sum_{\bx\in\cD}\left\Vert V_\theta\left(\bx\right)- V_{\theta}\left( \bx-\eta u_\theta\left(\bx\right)V_\theta\left(\bx\right) \right)\right\Vert ^2,
\end{equation}
where $\eta > 0$. At a singular point $\bx \in \Omega$, since there are multiple closest points $\bx_\Gamma$ \eqref{eq:condition1} on the interface, the gradient of the SDF at $\bx$ is not unique. In \eqref{eq:loss3}, we detour the singularity of the SDF by defining $V_\theta\left(\bx\right)$ at a singular point $\bx$ as a reliable gradient value near the interface $V_\theta\left(\bx - t\nabla u_\theta\left(\bx\right)\right)$ for some $t$. The idea is supported by the following proposition.
\begin{proposition}\label{prop:HJB}
Let $\Gamma\subset \Omega$ be a compact hypersurface of a domain $\Omega\subset \bR^n$ and  $u:\Omega\rightarrow\bR$ the signed distance function to $\Gamma$. For any point $\bx\in \Omega$ where $u$ is differentiable, the gradient of $u$ is constant along the ray $s\left(t\right)$ emanating from $\bx$ in direction $\nabla u\left(\bx\right)$:
\begin{align*}
    s\left(t\right)=\bx-t \sgn \left(u\left(\bx\right)\right)\nabla u\left(\bx\right),\ \forall\ t\in\left[0,\left\vert u\left(\bx\right)\right\vert\right).
\end{align*}
\end{proposition}
The proof is provided in \ref{appen:prop2}. In practice, as in $\cL_{\text{SP}}$, we replace $\nabla u_\theta$ with $V_\theta$ and we set $\eta = 0.99$ so that the point gets close to the target interface having reliable gradient value.
The loss function $\cL_{\text{RS}}$ enables accurate and stable training of $V_\theta$ and provides an appropriate approximation of SDF for various irregular interfaces.

To sum up, the proposed \textit{ReSDF} optimizes the augmented network through the objective
\begin{equation}\label{eq:total_loss}   
\cL_{\text{total}}\left(\theta\right) =\cL_{\text{GM}}\left(\theta\right) + \cL_{\text{SP}}\left(\theta\right) + \cL_{\text{RS}}\left(\theta\right).
\end{equation}
It can be extended to the weighted sum of each loss term where the positive regularization parameters balance the role of each. We may impose a larger weight on the larger component to accelerate its convergence. A more provable alternative to hand-tuned weights is the use of adaptive regularization parameter methods \cite{leung1999adaptive, wang2021understanding, xiang2022self}. However, in contrast to most PINN methods, which are sensitive to the regularization parameter, 
we experimentally confirm that the objective \eqref{eq:total_loss} is transferable across a wide variety of interfaces $\Gamma$ and level set functions $\phi$.

\subsection{Numerical Considerations}
There are a few numerical considerations to clarify the implementation of the proposed method. In order to minimize biased training of the neural network proportional to the values of $\phi$, we consider the following normalization of $\phi$
\[
\phi_{\text{normalized}}\left(\bx\right) =\frac{\phi\left(\bx\right)}{\sqrt{\max_\Omega\left(\phi\right)\max_\Omega\left(-\phi\right)}}, 
\]
where the maximum is computed over points in $\Omega$. The values of $\phi$ over the entire computational domain $\Omega$ are balanced in the objective $\cL_{\text{SP}}$ and the normalization alleviates the imbalance of $\phi$ and enables the model to recover the SDF well from the strongly distorted level set function $\phi$. To remove a singularity caused by the absolute value function in \eqref{eq:net_u}, we use a smooth approximation. 
For a point $x_0$ at which $\cS'\left(x_0\right)=0.6 \max_x \cS '\left(x\right) $, we define a quadratic approximation
\[
ABS_\infty \left(x\right)\coloneqq \begin{cases}
    \alpha x^2, & \text{if } \left\vert x\right\vert\leq x_0\\
    \left\vert x\right\vert + q, & \text{if } \left\vert x\right\vert> x_0,
\end{cases}
\]
where $\alpha$ is chosen such that $\left(\cS\cdot ABS_\infty\right)'\left(x_0\right)=1$, and $q$ is taken so that $ABS_\infty$ is continuous at $x_0$.
Moreover, we adopt the leaky rectified linear unit (LeakyReLU) activation function:
\[
\text{LeakyReLU}\left(x\right)=\begin{cases}x & \text{if } x\geq 0, \\ 0.01 x & \text{otherwise}, \end{cases}
\]
which is possibly a good representation for a non-smooth function. 

The second term of the loss functions $\cL_{\text{RS}}$ has a complex form in which the network is composited within the network. It causes practical difficulties in optimization when the loss function has a stacked network. In the implementation, we optimize $\cL_{\text{RS}}$ as follows: 
\[ \cL_{\text{RS}}\left(\theta\right)=\frac{1}{\mid \cD \mid}\sum_{\bx\in\cD}\left\Vert V_\theta\left(\bx\right)- V_{\hat{\theta}}\left( \bx-\eta u_\theta\left(\bx\right)V_\theta\left(\bx\right) \right)\right\Vert ^2,
\]
where we do not update the network $V_{\hat{\theta}}$ by disabling the gradient calculation of $V_{\hat{\theta}}$. Then, the computational graph is no longer associated with the parameters of $V_{\hat{\theta}}$ and it is regarded as a fixed function in the loss function.

The SDF is Lipschitz continuous with the Lipschitz constant $1$. In the multilayer perceptron structure approximating $u_\theta$, the activation function and weights $W_\ell$ determine the Lipschitz condition of the approximated SDF. Since we are using LeakyReLU whose Lipschitz conatant is 1 as an activation function, we can obtain the desired condition by adjusting the weights. The commonly used method to enforce the Lipschitz condition is \textit{weight clipping} \cite{arjovsky2017wasserstein}. It clamps the weights $W_\ell$ to a bounded box $W_\ell \in \left[-M, M\right]^{d_{\ell}\times d_{\ell-1}}$ after each gradient update. In the implementation, the clipping parameter $M=0.1$ is used.


\section{Numerical Results}
In this section, we present numerical results of \textit{ReSDF} on several problems to validate the effectiveness and accuracy for re-distancing the given level set functions. For quantitative measurements, the accuracy of the trained model is measured by the difference with the exact solution $u$ and exact gradient $\bn$ using the following norms:
\begin{equation}
\label{eq:error_whole}
\begin{alignedat}{4}
&\left\Vert u_\theta - u\right\Vert _{L^2}  && =\frac{1}{\left\vert\Omega\right\vert}\left(\int_{\Omega}\left\vert u_\theta\left(\bx\right)-u\left(\bx\right) \right\vert ^2d\bx\right)^{1/2}, \quad
&& \left\Vert u_\theta - u\right\Vert_{L^\infty} &&=\underset{ \bx\in\Omega}{\max}\left\vert  u_\theta\left(\bx\right)-u\left(\bx\right) \right\vert, \\
&\left\Vert V_\theta - \bn\right\Vert _{L^2}  && =\frac{1}{\left\vert\Omega\right\vert}\left(\int_{\Omega}\left\vert V_\theta\left(\bx\right)-\bn\left(\bx\right) \right\vert ^2ds\right)^{1/2},  \quad
&&\left\Vert V_\theta - \bn\right\Vert_{L^\infty} && =\underset{ \bx\in\Omega}{\max}\left\vert  V_\theta\left(\bx\right)-\bn\left(\bx\right) \right\vert,
\end{alignedat}
\end{equation}
where $\left\vert\Omega\right\vert$ is the volume of the computational domain $\Omega$. Since the errors on the interface $\Gamma$ are crucial for the cases of evolving the interface, we also check the errors, $\left\Vert V_\theta - \bn\right\Vert _{L^2}^{\Gamma}$ and $\left\Vert V_\theta - \bn\right\Vert _{L^\infty}^{\Gamma}$, which are only measured on the interface $\Gamma$ for a few examples. For qualitative comparison, the results of \textit{ReSDF} are compared with the existing PINN approach and the first \cite{sethian1996fast} and second-order \cite{sethian1999level} fast marching method (FMM), where implementation builds on an open-source code \footnote{\url{https://github.com/scikit-fmm/scikit-fmm}}.  The PINN method based on \cite{fayolle2021signed} uses the residual of the eikonal equation \eqref{eq:pinn_eikonal} as the loss function. As suggested in \cite{fayolle2021signed}, 
a fully connected network with depth $8$ and width $512$ with skip connection from the input and the fourth hidden layer are used in which all weights and biases are initialized by the geometric initialization scheme \cite{atzmon2020sal}. We optimize the network by Adam optimizer \cite{kingma2014adam} with learning rate of $10^{-4}$. We also comply with all other experimental configurations as provided in \cite{fayolle2021signed}. Note that the results of the PINN approach can be improved because the performance relies on specific geometric network parameter initialization \cite{atzmon2020sal} and is sensitive to hyper-parameter selection.

For all examples of testing \textit{ReSDF}, 
the Adam optimizer is applied with a learning rate of $10^{-3}$, which decayed by $0.9$ if the loss function does not improve after $30$ evaluations. We evaluate the models per $100$ epochs and trained them until the learning rate decayed by $40$ times. In all numerical experiments, a single NVIDIA RTX 3090 GPU is used.

\begin{figure}
	\begin{center}
		\begin{tabular}{ccc}
			\includegraphics[height=0.35\textwidth]{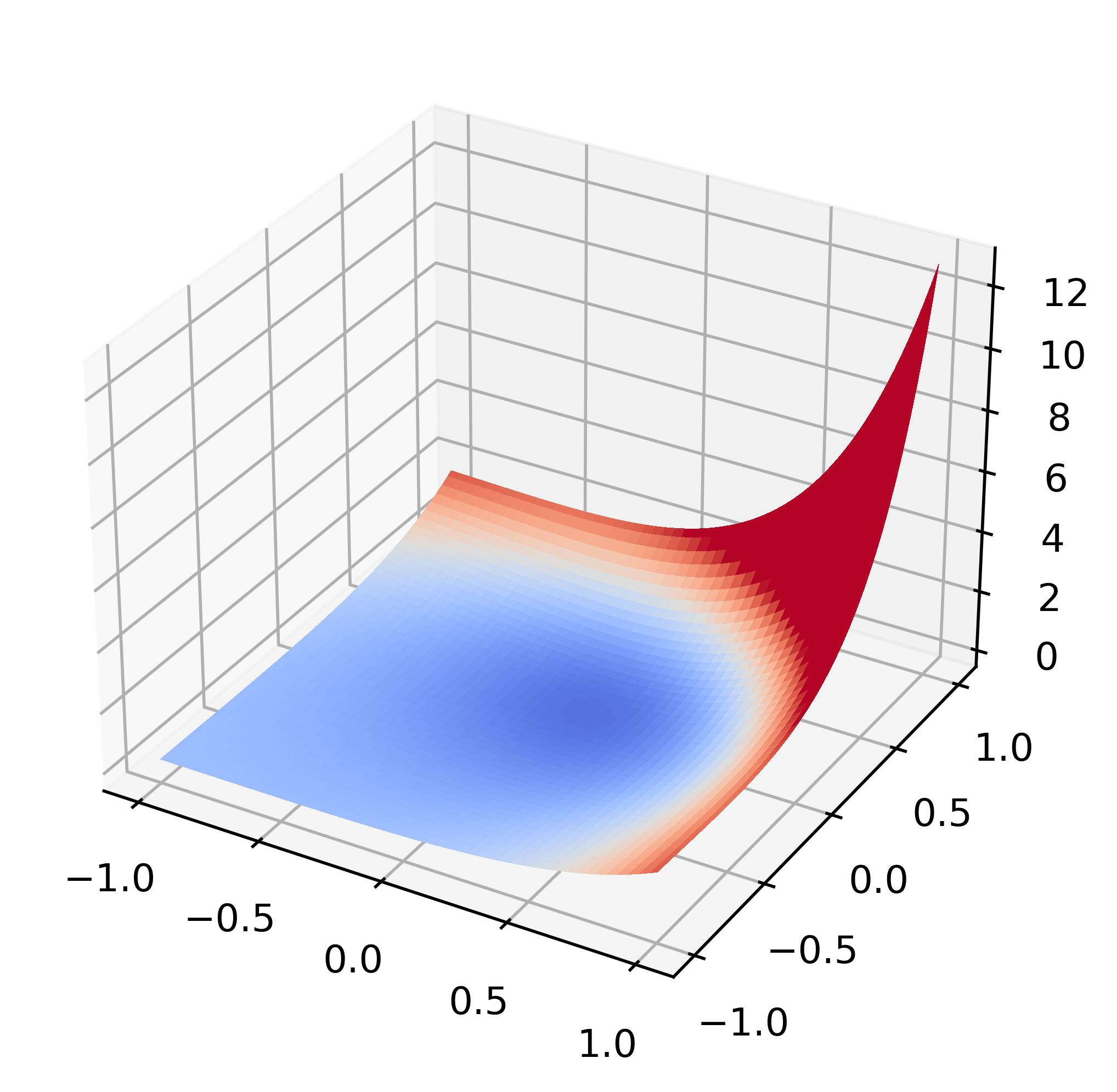} & 
			\includegraphics[height=0.35\textwidth]{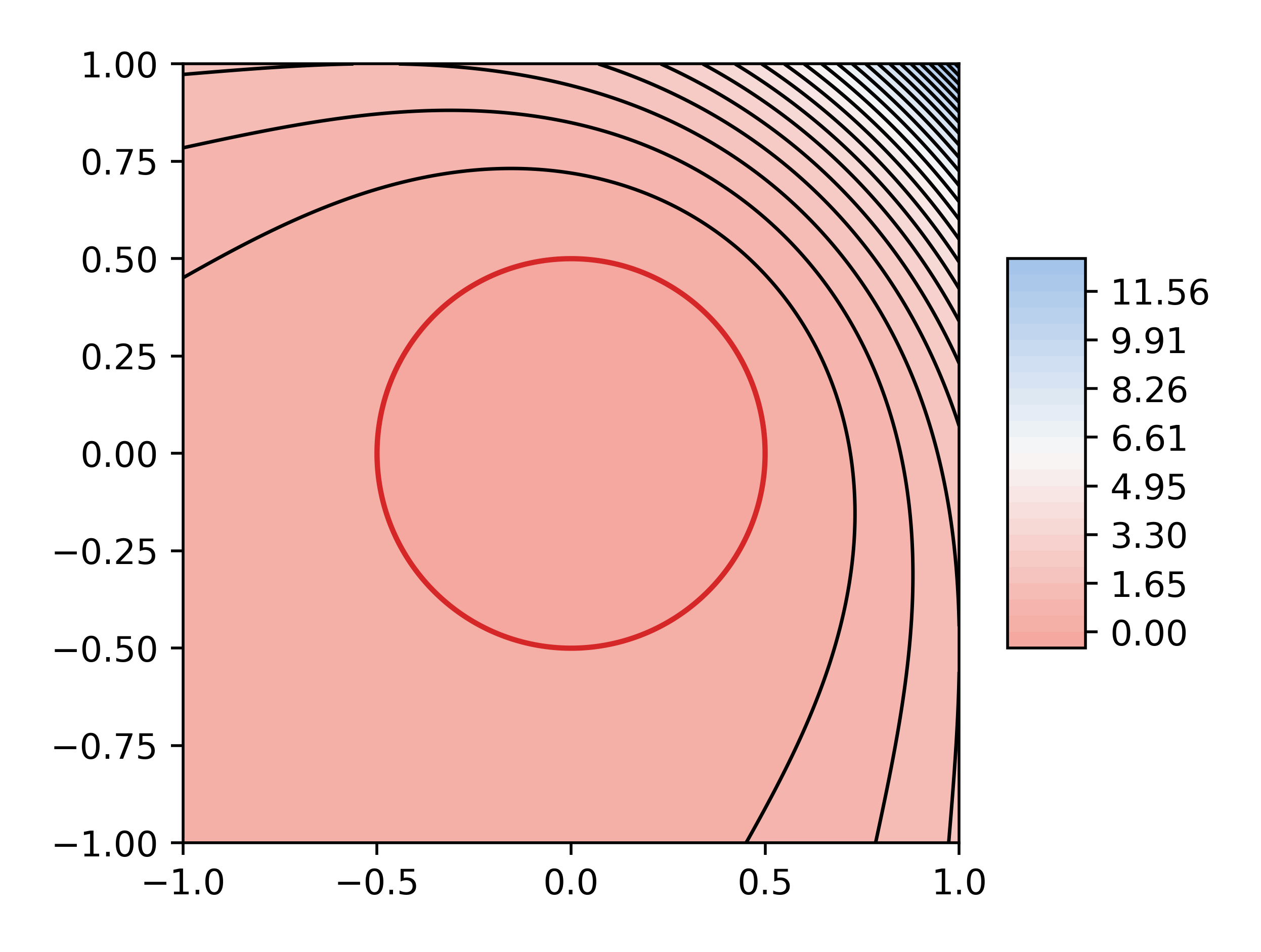} \\
			\multicolumn{2}{c}{(a) $\phi_1$} \\
			\includegraphics[height=0.35\textwidth]{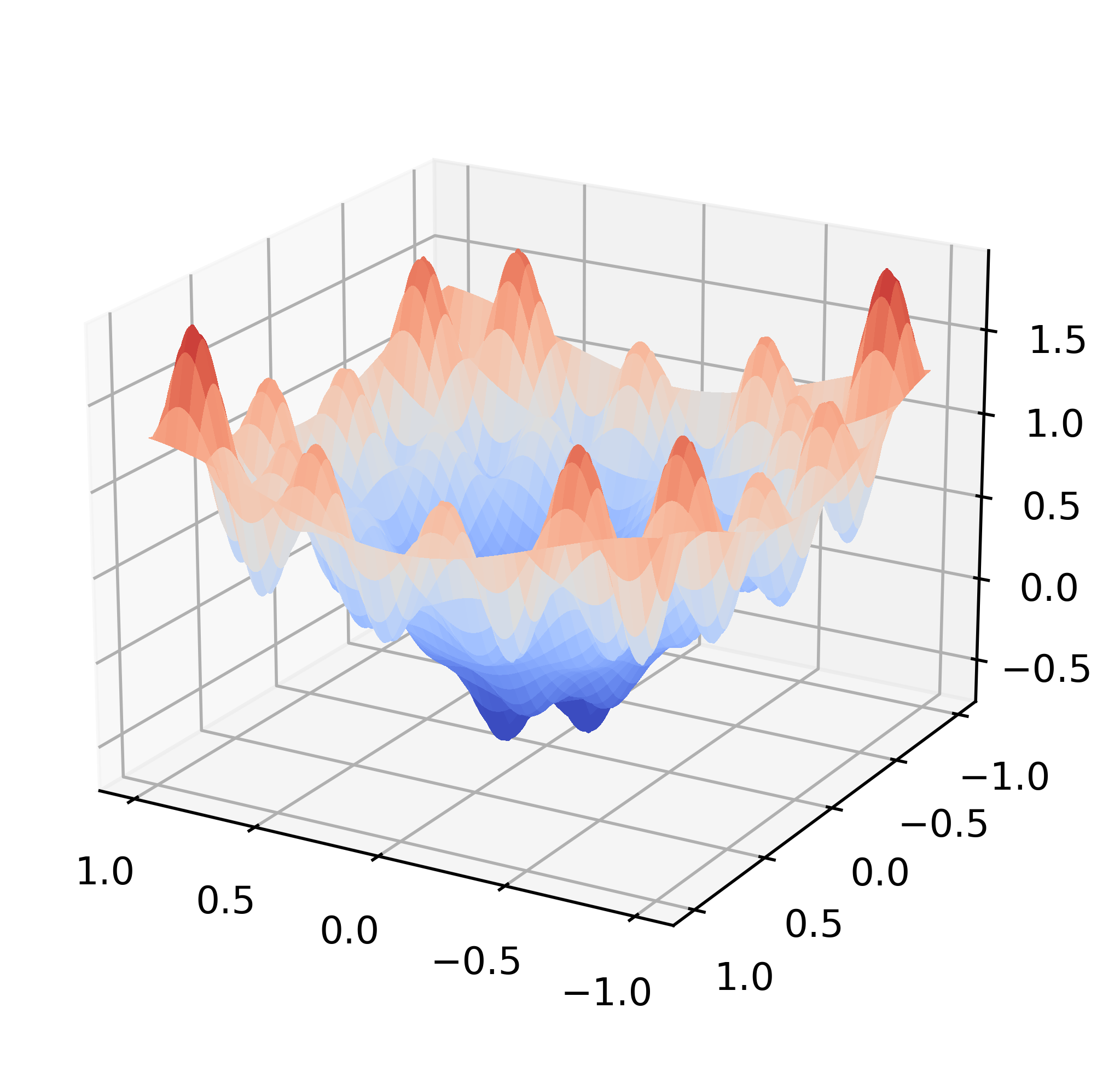} &
			\includegraphics[height=0.35\textwidth]{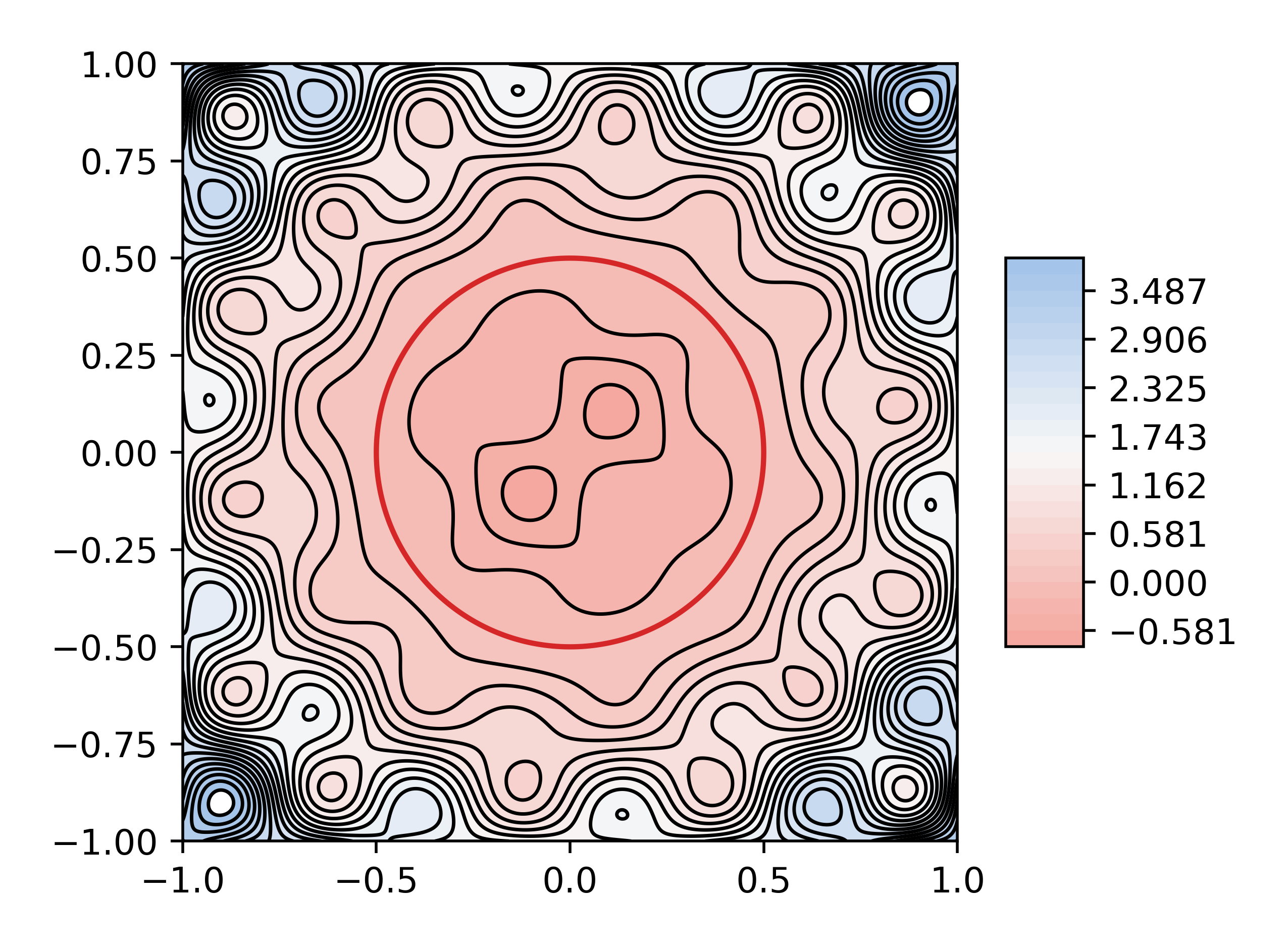} \\
			\multicolumn{2}{c}{(a) $\phi_2$} \\
			\includegraphics[height=0.35\textwidth]{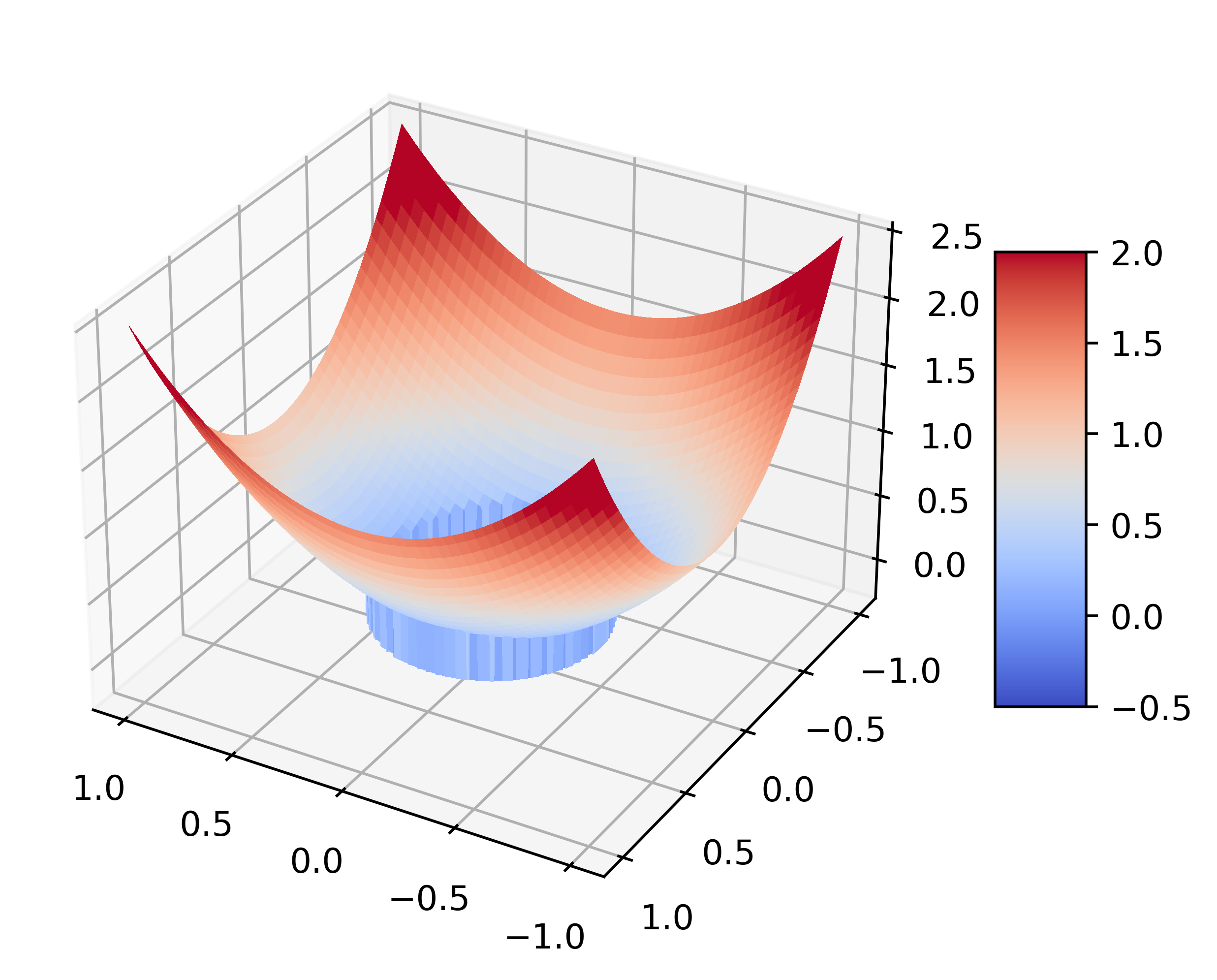} &
			\includegraphics[height=0.35\textwidth]{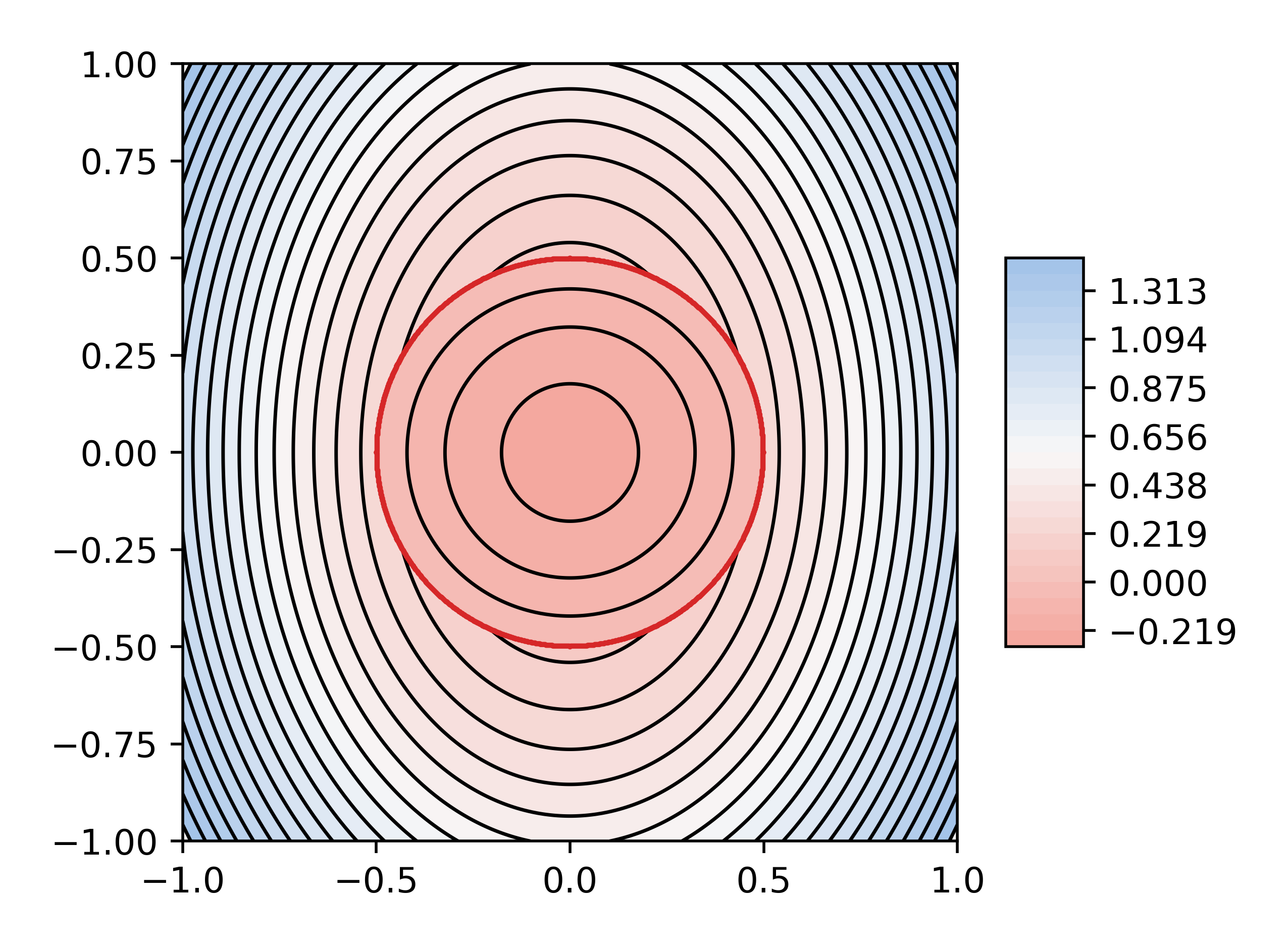} \\
			\multicolumn{2}{c}{(c) $\phi_3$}
		\end{tabular}
	\end{center}
	\caption{The graph (left) and iso-contours (right) of level set functions, a stiff shape $\phi_1$ \eqref{eq:circ_phi_exp}, an oscillatory shape $\phi_2$ \eqref{eq:circ_phi_oscill}, and a discontinuous shape $\phi_3$ \eqref{eq:circ_phi_jump} are presented.}\label{fig:circ_phi} 
\end{figure}

The computational domains of all examples are $\Omega^{\text{3d}} = [-1.1]^3$ or $\Omega = \Omega^{\text{3d}} \cap \left(\mathbb{R}^2\times\{0\}\right)$. The collocation points evenly distributed on the domain $\Omega^{\text{3D}}$ are presented by
\begin{align}\label{eq:collpts}
    \Omega_n^{\text{3d}} = \left\{ \left( -1 + ih, -1 + jh, -1 + kh\right) : i, j, k \in \{ 0, \ldots, n-1 \}, \: h = \frac{2}{n-1} \right\}.
\end{align}
Accordingly, $\Omega_n =\Omega_n^{\text{3d}} \cap \left(\mathbb{R}^2\times\{0\}\right)$. All interfaces and given level set functions used in numerical examples are listed below. In case of the closed form of the exact SDF is not known for given interfaces, we approximate it by creating a finite number of points, around 7000, on the interface and finding the minimum Euclidean distance from them.
\begin{enumerate}
    \item A circle centered at the origin with the radius $r=\frac{1}{2}$ is used in \nameref{subsec:circle}. The explicit form of the exact SDF is an equation of the cone:
    \begin{align}\label{eq:cone}
        u\left(x,y\right)=\sqrt{ x^2+ y ^2}-r.
    \end{align}
    Three level set functions are considered in \nameref{subsec:circle}:
    \begin{itemize} 
        \item An unbalanced function that is relatively flat and stiff on $\Omega^-$ and $\Omega^+$, respectively; see Figure \ref{fig:circ_phi}-(a):
            \begin{equation}\label{eq:circ_phi_exp}
                \phi_1\left(x,y\right)=e^{x+y}u(x,y).
            \end{equation}
        \item A rapidly oscillating non-monotone function; see Figure \ref{fig:circ_phi}-(b):
            \begin{equation}\label{eq:circ_phi_jump}
            \phi_2\left(x,y\right)=\begin{cases}
            x^2+\frac{3}{2}y^2 & \text{ if } u(x,y) > 0,  \\
            (u(x,y) - r_0)^2-r_0^2 & \text{ otherwise }.
            \end{cases}
            \end{equation}  
        \item A discontinuous function with a jump on the interface; see Figure \ref{fig:circ_phi}-(c):
            \begin{equation}\label{eq:circ_phi_oscill}
            \phi_3\left(x,y\right)=\left(1.2\sin\left(4\pi x\right)\sin\left(4\pi y\right)+2\right)e^{-\left(x^2+y^2\right)/2}u(x,y).
            \end{equation}
    \end{itemize}
    \item A square centered at the origin with the length of the side $r = 1$ is used in \nameref{subsec:square}. The level set function whose zero level is the square is given; see the coutours in Figure \ref{fig:rectangle}-(a):
    \begin{align}\label{eq:sq_phi}
        \phi_4\left(x,y\right)=\frac{1}{2}\max\left\{ \mid x\mid-\frac{r}{2},\mid y\mid-\frac{r}{2}\right\}.
    \end{align}
    \item Flower-shaped interfaces with five-fold symmetry represented by the zero level set of the level set function with the polar coordinates $\left(r,\theta\right)$ are used in \nameref{subsec:irr}; see the iso-contours in Figures \ref{fig:phi_irr}-(a) and (b):
    \begin{align} \label{eq:flower}
        \phi_5^{\alpha}\left(x,y\right)=\left(0.5\sin\left(6\pi x\right)\sin\left(6\pi y\right) + 1\right)\left(5r - 2 - \alpha\cos\left(5\theta\right)\right),
    \end{align}
    where $\alpha = 0.5$ and $1$.
    \item A dumbbell-shaped interface represented by the zero level set of the level set function is used in \nameref{subsec:irr}; see the iso-contours in Figure \ref{fig:phi_irr}-(c):
    \begin{align} \label{eq:dumbbell}
        \phi_6\left(x,y\right)=10x^4\left(2x^2-1\right)+y^2-\frac{1}{10}.
    \end{align}
    \item A heart-shaped interface represented by the zero level set of the level set function is used in \nameref{subsec:irr}; see the coutours in Figure \ref{fig:phi_irr}-(d):
    \begin{align} \label{eq:heart}
        \phi_7\left(x,y\right)=\left(\cos\left(x,y\right)+2\left(x+y\right)^2+\frac{3}{2}\right)\left(2.2\left(y+0.2-x^{\frac{2}{3}}\right)^2+1.7x^2-0.6\right).
    \end{align}
    \item Multiple circular interfaces represented by the zero level set of minimum of the level set functions $u_j\left(x,y\right)=(x - x_j) ^2+( y - y_j)^2 - r_j^2$ are used in \nameref{subsec:mul_intf}:
    \begin{align}
        \phi_8\left(x,y\right) &= e^{x+y} \min_{x,y} \left\{u_j(x,y) : j = 1, 2\right\}, \label{eq:two_cir} \\
        \phi_9\left(x,y\right) &= e^{x+y} \min_{x,y} \left\{u_j(x,y) : j = 3, 4, 5\right\}, \label{eq:three_cir}
    \end{align}
    where $(x_1, y_1, r_1) = (-0.2,0,0.3)$, $(x_2, y_2,r_2) = (-0.2,0,0.3)$, $(x_3, y_3,r_3) = (-0.4,0.3,0.45)$, $(x_4, y_4,r_4) = (0.5,0.3,0.3)$, and $(x_5, y_5,r_5) = (0.3,-0.5,0.4)$; see the iso-contours in Figures \ref{fig:multcircle}-(a) and (c).
    \item A sphere centered at the origin with the radius $r=\frac{1}{2}$ is represented by the zero level sest of the level set function used in \nameref{subsec:3D}; see the coutours on $z=0$ in Figure \ref{fig:3d}-(a):
    \begin{align} \label{eq:3d_sphere}
    \phi_{10}\left(x,y,z\right) = \left(\left(x-1\right)^2+\left(y-1\right)^2+\left(z+1\right)^2+0.1\right)\left(x^2+y^2+z^2-r^2\right).
    \end{align}
    Multiple ellipsoidal interfaces represented by the zero level set of minimum of the level set functions $v_j(x,y,z) = 10 (x - x_j)^2 + 5 (y - y_j)^2 + (z - z_j)^2 - 1$ are used in \nameref{subsec:3D}; see the iso-contours on $z=0$ in Figure \ref{fig:3d}-(c):
    \begin{align} \label{eq:3d_two_ellip}
    \phi_{11}\left(x,y,z\right) = \min_{x,y,z} \left\{v_j(x,y,z):j=1,2\right\},
    \end{align}
    where $(x_1,y_1,z_1)=(-0.2,0,0)$ and $(x_2,y_2,z_2)=(0.2,0,0)$.
\end{enumerate}

\subsection*{Example 1}\label{subsec:circle}

\begin{table}
    \centering
      \setlength\tabcolsep{4pt}
    \caption{The $L^2$ and $L^\infty$ errors between exact and predicted SDF of \textit{ReSDF} with $\phi_1$ \eqref{eq:circ_phi_exp} are listed. The network depth is fixed to $L=4$.} \label{tab:acc_u}
    \scalebox{0.84}{
    \begin{tabular}{ccccccccc}
    \toprule
    \multirow{2}{*}{\diagbox[innerwidth=\textwidth*1/7]{$\mid\cD\mid$}{width} }
     & \multicolumn{2}{c}{32} & \multicolumn{2}{c}{64} & \multicolumn{2}{c}{128} & \multicolumn{2}{c}{256} \\ 
    \cmidrule(lr){2-3} \cmidrule(lr){4-5} \cmidrule(lr){6-7} \cmidrule(lr){8-9}
      & $\left\Vert u_\theta - u\right\Vert_{2}$ & $\left\Vert u_\theta - u\right\Vert_{\infty}$  & $\left\Vert u_\theta - u\right\Vert_{2}$ & $\left\Vert u_\theta - u\right\Vert_{\infty}$  & $\left\Vert u_\theta - u\right\Vert_{2}$ & $\left\Vert u_\theta - u\right\Vert_{\infty}$  & $\left\Vert u_\theta - u\right\Vert_{2}$ & $\left\Vert u_\theta - u\right\Vert_{\infty}$ \\
    \midrule
    $\Omega_6$  & $1.25\cdot 10^{-3}$ & $1.66\cdot 10^{-2}$ & $4.81\cdot 10^{-4}$ & $9.33\cdot 10^{-3}$ & $1.43\cdot 10^{-4}$  & $4.51\cdot 10^{-3}$ & $1.60\cdot 10^{-4}$ & $4.41\cdot 10^{-3}$\\
    $\Omega_7$  & $2.38\cdot 10^{-3}$ & $3.51\cdot 10^{-2}$ & $1.84\cdot 10^{-4}$ & $2.05\cdot 10^{-3}$ & $1.01\cdot 10^{-4}$  & $1.81\cdot 10^{-3}$ & $8.70\cdot 10^{-5}$ & $2.44\cdot 10^{-3}$ \\
    $\Omega_8$ & $1.50\cdot 10^{-3}$ & $1.87\cdot 10^{-2}$ & $5.65\cdot 10^{-4}$ & $9.97\cdot 10^{-3}$ & $8.15\cdot 10^{-5}$ & $1.23\cdot 10^{-3}$ & $6.90\cdot 10^{-5}$ & $1.96\cdot 10^{-3}$  \\
    \bottomrule
  \end{tabular}}
\end{table} 

\begin{table}
    \centering
      \setlength\tabcolsep{4pt}
    \caption{The $L^2$ and $L^\infty$ errors between exact gradient and the gradient of the predicted SDF of \textit{ReSDF} with $\phi_1$ \eqref{eq:circ_phi_exp} on the interface $\Gamma$ are listed. The network depth is fixed to $L=4$.} \label{tab:acc_V}
    \scalebox{0.84}{
    \begin{tabular}{ccccccccc}
    \toprule
    \multirow{2}{*}{\diagbox[innerwidth=\textwidth*1/7]{$\mid\cD\mid$}{width} }
     & \multicolumn{2}{c}{32} & \multicolumn{2}{c}{64} & \multicolumn{2}{c}{128} & \multicolumn{2}{c}{256} \\ 
    \cmidrule(lr){2-3} \cmidrule(lr){4-5} \cmidrule(lr){6-7} \cmidrule(lr){8-9}
      & $\left\Vert V_\theta -\bn\right\Vert_{2}^{\Gamma}$ & $\left\Vert V_\theta -\bn\right\Vert_{\infty}^{\Gamma}$ & $\left\Vert V_\theta -\bn\right\Vert_{2}^{\Gamma}$ & $\left\Vert V_\theta -\bn\right\Vert_{\infty}^{\Gamma}$ & $\left\Vert V_\theta -\bn\right\Vert_{2}^{\Gamma}$ & $\left\Vert V_\theta -\bn\right\Vert_{\infty}^{\Gamma}$ & $\left\Vert V_\theta -\bn\right\Vert_{2}^{\Gamma}$ & $\left\Vert V_\theta -\bn\right\Vert_{\infty}^{\Gamma}$ \\
    \midrule
    $\Omega_6$  & $4.40\cdot 10^{-3}$ & $2.01\cdot 10^{-2}$ & $4.85\cdot 10^{-3}$ & $2.97\cdot 10^{-2}$ & $2.41\cdot 10^{-3}$ & $1.10\cdot 10^{-2}$ & $6.10\cdot 10^{-3}$ & $3.08\cdot 10^{-2}$ \\
    $\Omega_7$ & $4.93\cdot 10^{-3}$ & $2.51\cdot 10^{-2}$ & $1.76\cdot 10^{-3}$ & $7.86\cdot 10^{-3}$ & $1.14\cdot 10^{-3}$ & $5.43\cdot 10^{-3}$ & $4.10\cdot 10^{-3}$ & $2.15\cdot 10^{-2}$ \\
    $\Omega_8$ & $3.71\cdot 10^{-3}$ & $2.00\cdot 10^{-2}$ & $3.50\cdot 10^{-3}$ & $1.99\cdot 10^{-2}$ & $1.35\cdot 10^{-3}$ & $7.64\cdot 10^{-3}$ & $1.59\cdot 10^{-3}$ & $8.64\cdot 10^{-3}$  \\
    \bottomrule
  \end{tabular}}
\end{table} 

In this example, we would like to numerically check an accuracy and robustness of the proposed model using an exact SDF \eqref{eq:cone}. The accuracy of the learned solution is investigated with respect to different network sizes and batch sizes $\mid\cD\mid$ of the training data points. Numerical tests of \textit{ReSDF} are performed by using the level set function $\phi_1$ defined in \eqref{eq:circ_phi_exp} with network widths $2^m$, $m=5,\ldots,8$ and collocation points $\Omega_n$ with $n=6,\ldots,8$. In Table \ref{tab:acc_u}, the $L^2$ and $L^\infty$ errors \eqref{eq:error_whole} between the exact SDF and learned $u_\theta$ over the entire computational domain are presented. The accuracy of the obtained $V_\theta$ measured by $L^2$ and $L^\infty$ errors \eqref{eq:error_whole} on the interface between $V_\theta$ and the exact outward normal vector $\bn$ is summarized in Table \ref{tab:acc_V}. Note that only the number of collocation points and the width are changed and the other experimental settings are kept the same. The results show that the errors do not alter significantly with different batch sizes of training collocation points. Even with a coarse grid, we still obtain accurate predictions. We can also observe that the resulting errors decrease as the network width size increases.
This is a general phenomenon in which the expressive power increases as the size of the network increases.

\begin{table}
    \centering
    \setlength\tabcolsep{8pt}
    \caption{Comparison of the accuracy of the learned SDFs for the circular interface obtained by PINN and \textit{ReSDF} for three initial level set functions $\phi_i$; a stiff shape $\phi_1$ \eqref{eq:circ_phi_exp}, an oscillatory shape $\phi_2$ \eqref{eq:circ_phi_oscill}, and a discontinuous shape $\phi_3$ \eqref{eq:circ_phi_jump}.} \label{tab:acc_circ_phis}
    \vspace{2pt}
    \scalebox{0.9}{
    \begin{tabular}{cccccccccc}
    \toprule
    \multirow{2}{*}{Model} & \multirow{2}{*}{\# Params}  & \multicolumn{2}{c}{$\phi_1$} & \multicolumn{2}{c}{$\phi_2$} & \multicolumn{2}{c}{$\phi_3$} \\ 
    \cmidrule(lr){3-4} \cmidrule(lr){5-6} \cmidrule(lr){7-8}
     & & $\left\Vert u_\theta - u\right\Vert_{2}$ & $\left\Vert u_\theta - u\right\Vert_{\infty}$  & $\left\Vert u_\theta - u\right\Vert_{2}$ & $\left\Vert u_\theta - u\right\Vert_{\infty}$  & $\left\Vert u_\theta - u\right\Vert_{2}$ & $\left\Vert u_\theta - u\right\Vert_{\infty}$ \\
    \midrule
    \textit{ReSDF} & 17,027 & $1.84 \cdot 10^{-4}$ & $2.05 \cdot 10^{-3}$ & $1.10\cdot 10^{-4}$ & $2.03\cdot 10^{-3}$ & $4.31\cdot 10^{-3}$  & $1.25\cdot 10^{-2}$ \\
    PINN & 1,840,641 & $2.74\cdot 10^{-1}$ &$7.99\cdot 10^{-1}$ & $8.32\cdot 10^{-1}$ & $2.26\cdot 10^{0}$ & $5.22\cdot 10^{0}$  & $19.92\cdot 10^{0}$ \\
    \bottomrule
  \end{tabular}}
\end{table} 

\begin{figure}
	\begin{center}
		\begin{tabular}{cc}
			\includegraphics[height=0.35\textwidth]{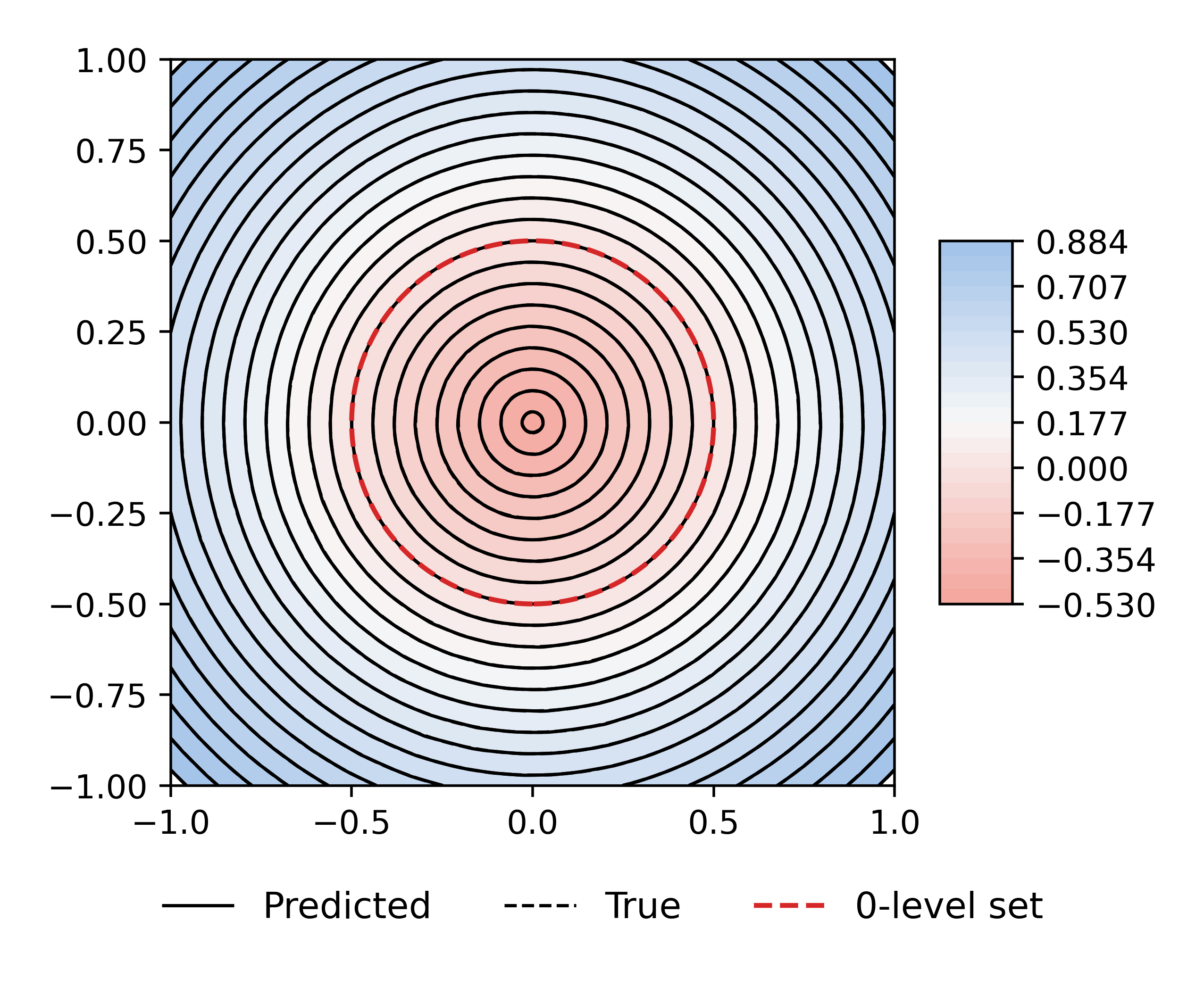} & 
			\includegraphics[height=0.35\textwidth]{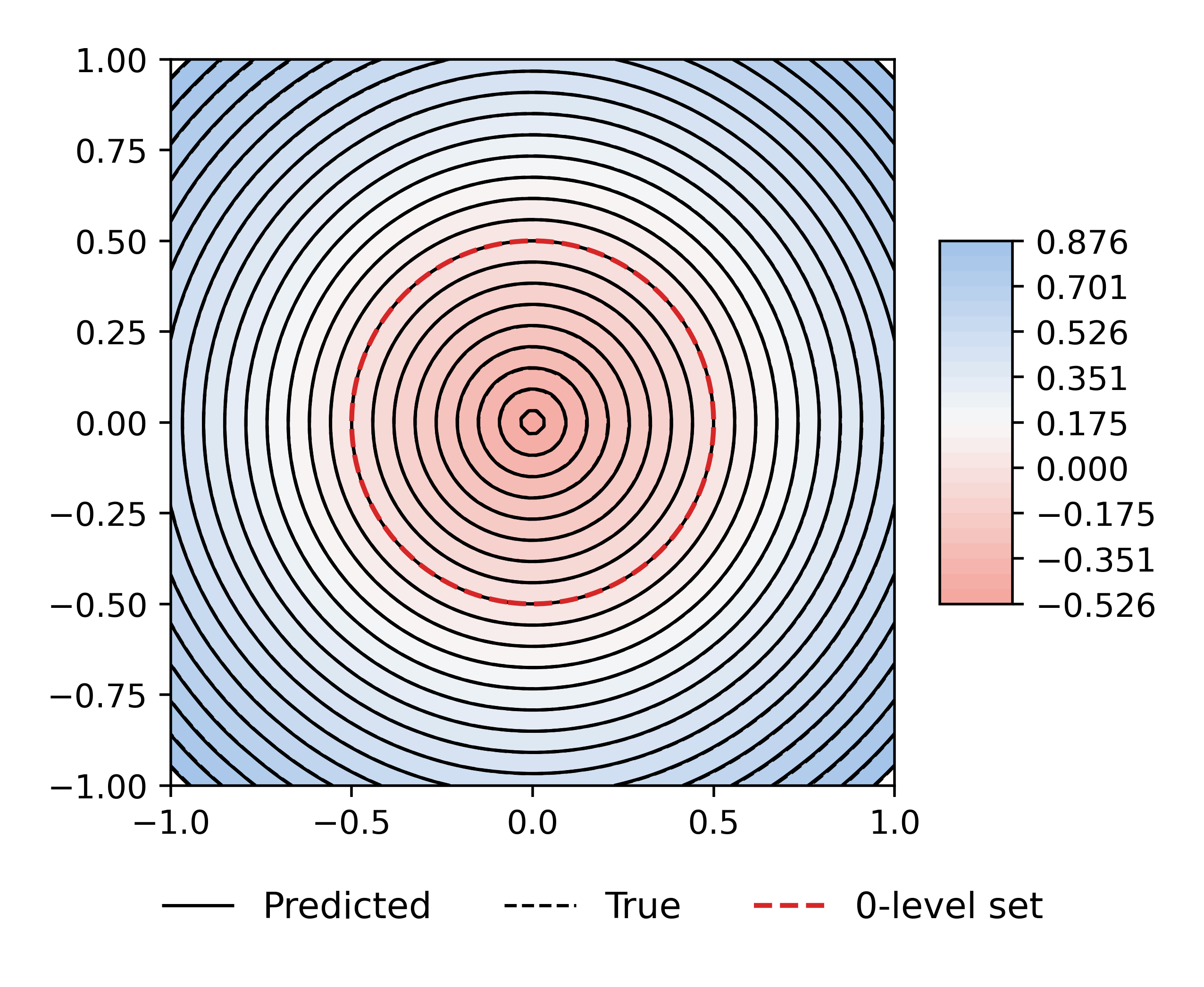} \\
			(a1) \textit{ReSDF} with $\phi_1$ & (b1) PINN with $\phi_1$ \\
			\includegraphics[height=0.35\textwidth]{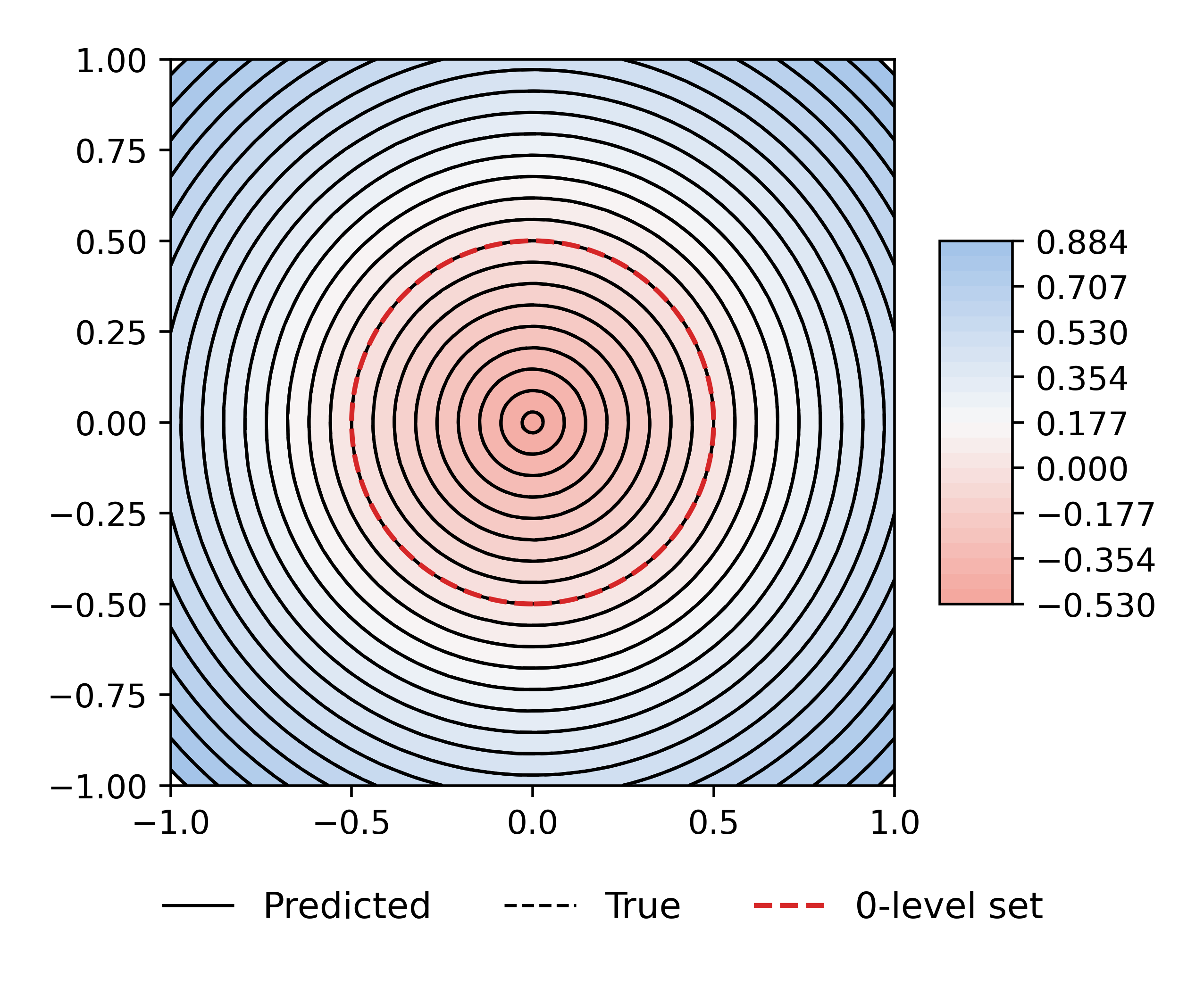} &
			\includegraphics[height=0.35\textwidth]{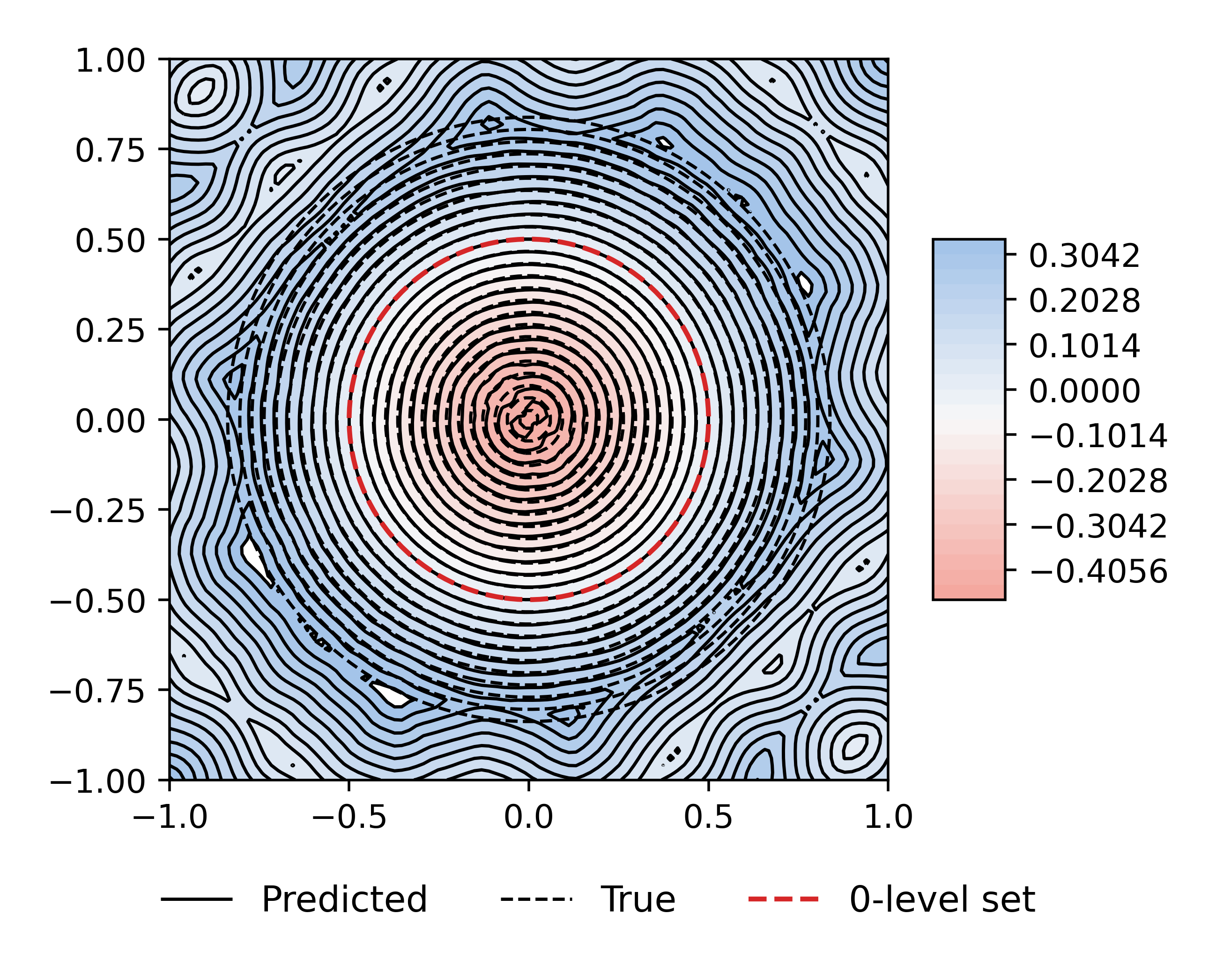} \\
			(a2) \textit{ReSDF} with $\phi_2$ & (b2) PINN with $\phi_1$ \\
			\includegraphics[height=0.35\textwidth]{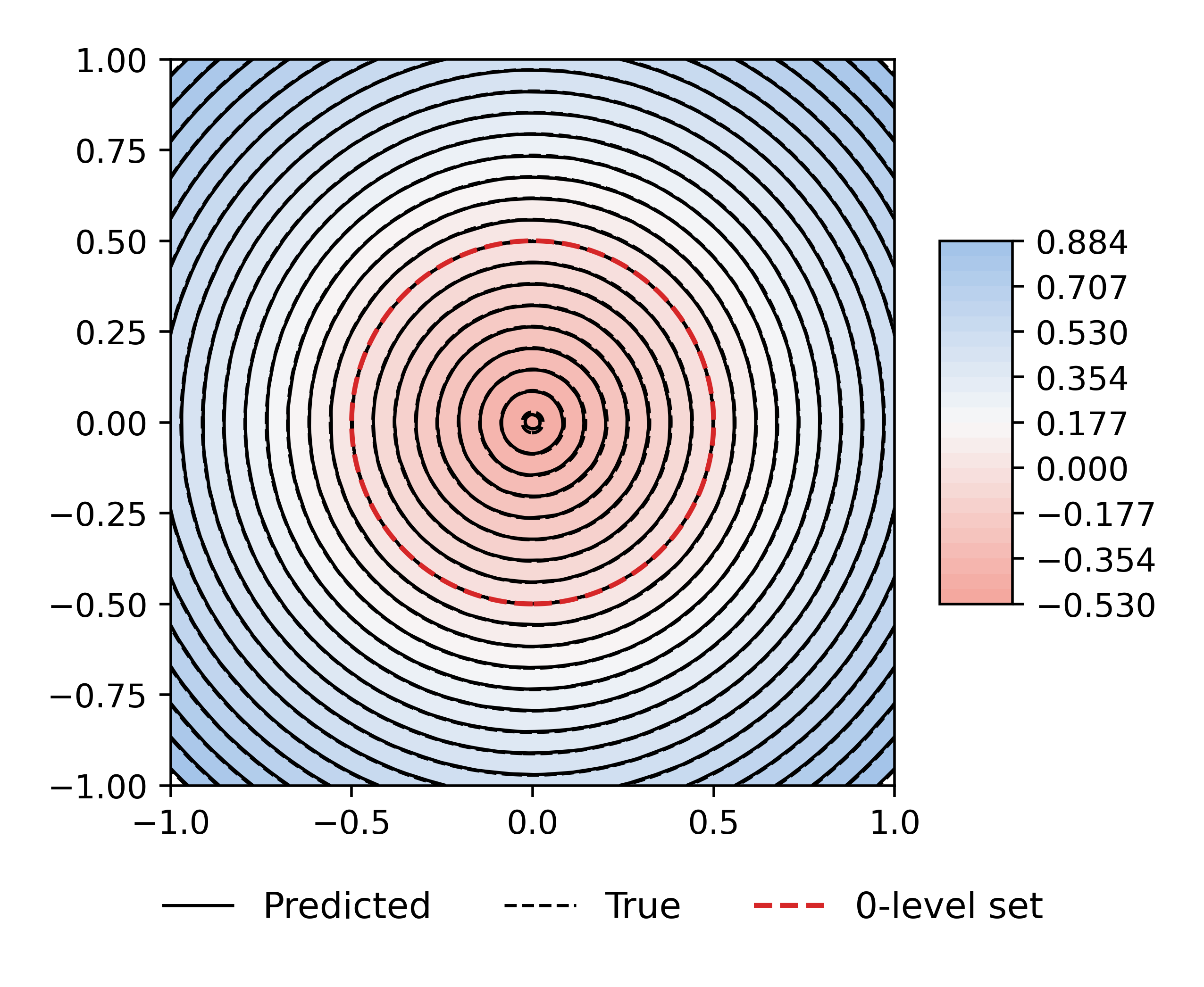} & 
			\includegraphics[height=0.35\textwidth]{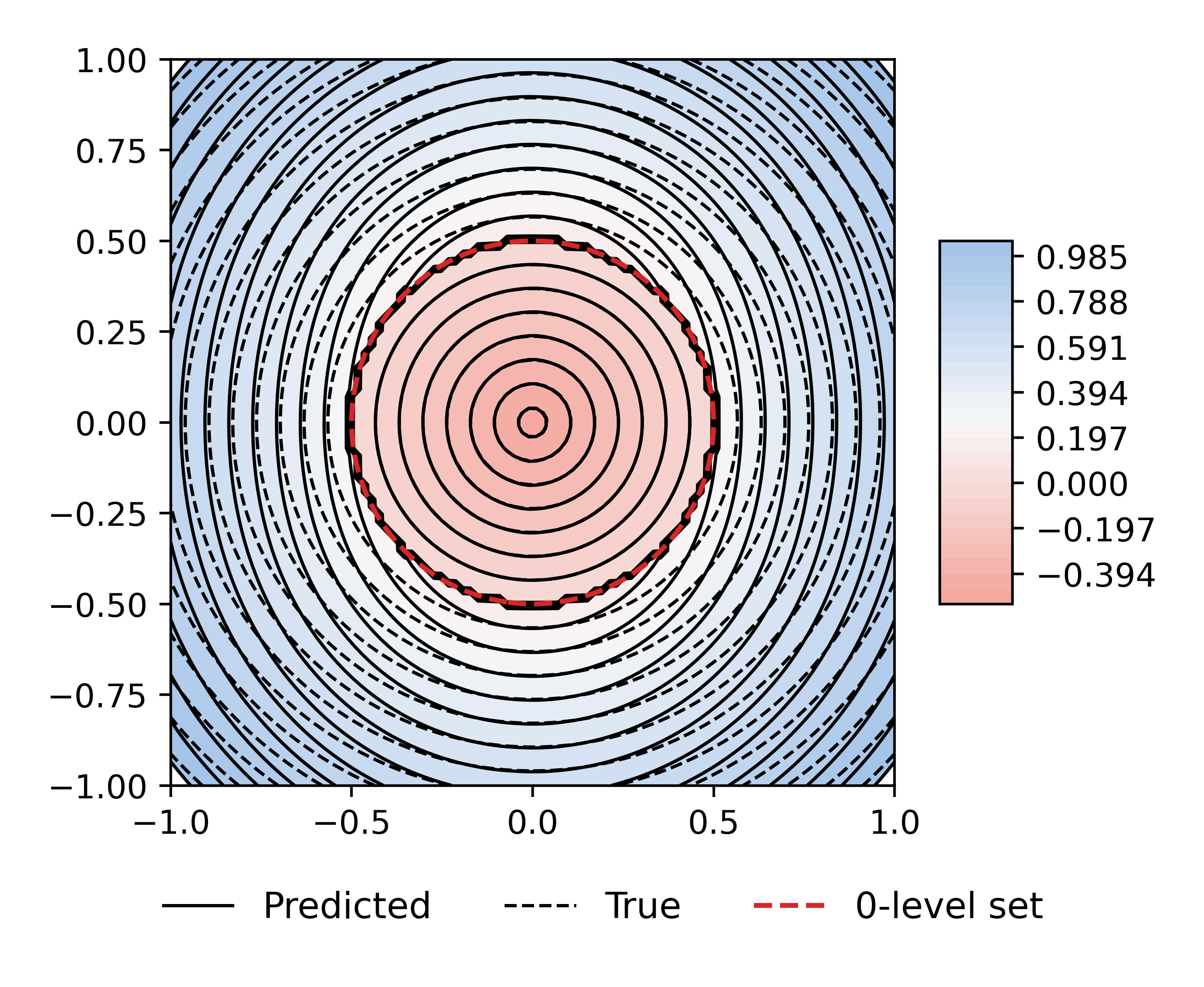} \\
			(a3) \textit{ReSDF} with $\phi_3$ & (b3) PINN with $\phi_3$
		\end{tabular}
	\end{center}
	\caption{The iso-contours of the results of \textit{ReSDF} (left) and PINN approach (right) from three level set functions, a stiff shape $\phi_1$ \eqref{eq:circ_phi_exp}, an oscillatory shape $\phi_2$ \eqref{eq:circ_phi_oscill}, and a discontinuous shape $\phi_3$ \eqref{eq:circ_phi_jump} are presented with the exact solution. The red solid curve is the zero level set of mentioned level set functions.}\label{fig:circle} 
\end{figure}

The robustness of changing the level set functions is examined by using diverse initial level set functions; a stiff shape $\phi_1$ \eqref{eq:circ_phi_exp}, an oscillatory shape $\phi_2$ \eqref{eq:circ_phi_oscill}, and a discontinuous shape $\phi_3$ \eqref{eq:circ_phi_jump}. In this case, the model is trained with a fixed width of size $64$ on collocation points $\Omega_7$. The results are compared with that of the PINN approach. The accuracy of two models is listed in Table \ref{tab:acc_circ_phis}. The results of \textit{ReSDF} have a smaller error than the results of the PINN approach by a factor of $10^{-3}$ and $10^{-2}$ in the $L^2$ and $L^\infty$ norms, respectively. In order to present visual differences, we compare the predictions of the trained \textit{ReSDF} and PINN against the exact SDF in Figure \ref{fig:circle}. It shows a similar result with $\phi_1$. However, the PINN approach fails to predict SDF with $\phi_2$ and $\phi_3$.

\subsection*{Example 2}\label{subsec:square}

\begin{figure}
	\begin{center}
		\begin{tabular}{cc}
			\includegraphics[height=0.35\textwidth]{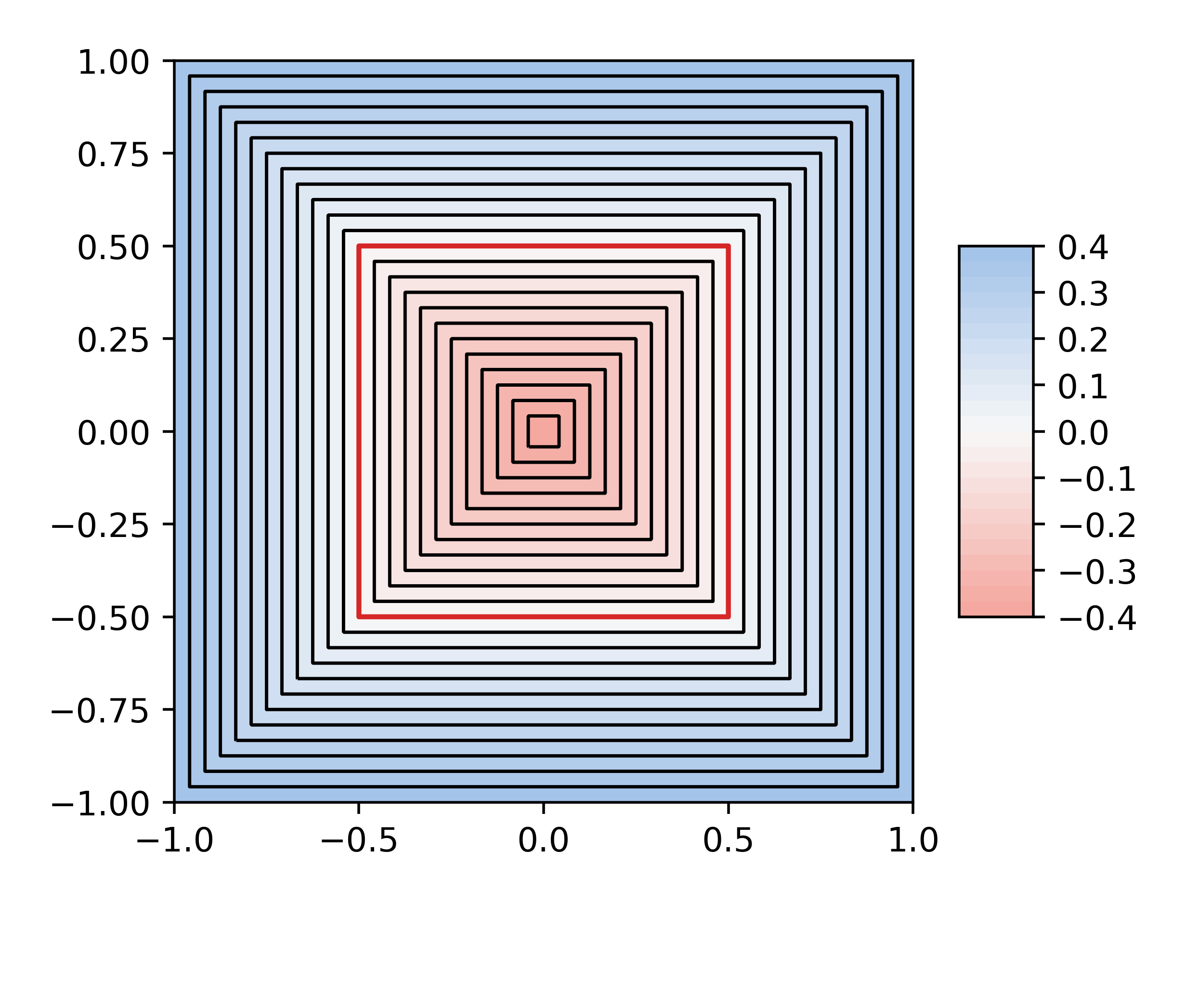} &
			\includegraphics[height=0.35\textwidth]{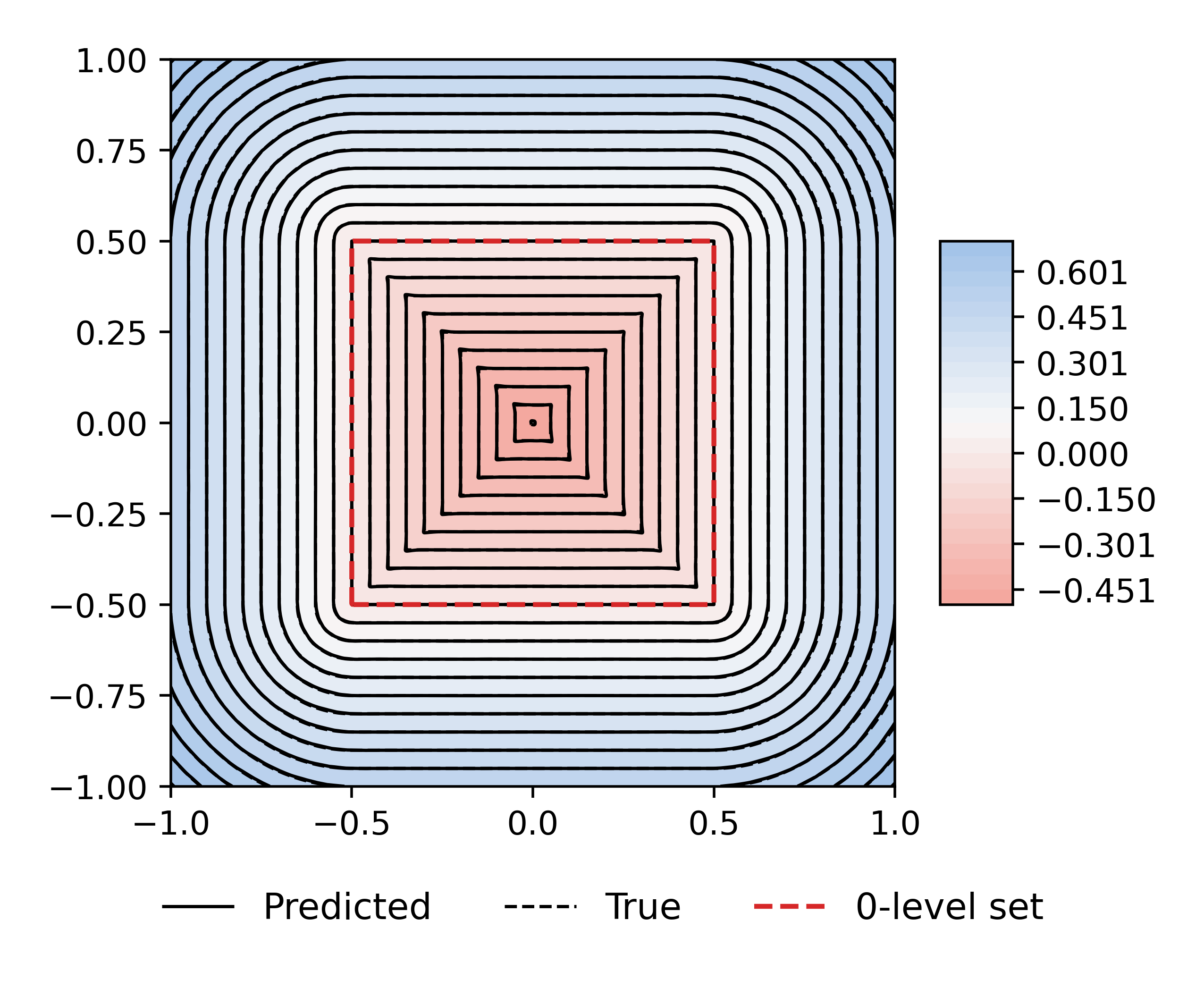} \\
			(a) $\phi_4$ & (b) \textit{ReSDF} with $\phi_4$
		\end{tabular}
	\end{center}
	\caption{The iso-contours of the level set function $\phi_4$ \eqref{eq:sq_phi} (left) and the results of \textit{ReSDF} with $\phi_4$ (right) are presented with the exact solution. The red solid curve is the zero level set of $\phi_4$.} \label{fig:rectangle}
\end{figure}

\begin{table}
    \centering
    \setlength\tabcolsep{8pt}
    \caption{The $L^2$ and $L^\infty$ errors between exact solution and predicted SDF of \textit{ReSDF} of \textit{ReSDF} with $\phi_4$ \eqref{eq:sq_phi} or the second order FMM ($\text{FMM}^2$).
    Errors between the gradient of the exact solution $\bn$ and the predicted gradient are also reported.} \label{tab:sq}
    \vspace{2pt}
    \scalebox{0.95}{
    \begin{tabular}{cccccc}
    \toprule 
    Method & $\cD$ & $\left\Vert u_\theta - u\right\Vert_{2}$ & $\left\Vert u_\theta - u\right\Vert_{\infty}$ & $\left\Vert V_\theta -\bn\right\Vert_{2}$ & $\left\Vert V_\theta -\bn\right\Vert_{\infty}$ \\
    \midrule
   \textit{ReSDF} & $\Omega_6$ & $2.7\cdot10^{-4}$ &  $1.0\cdot10^{-2}$ & $3.6\cdot10^{-2}$ &  $1.0\cdot10^{0}$  \\
   $\text{FMM}^2$ & $\Omega_6$ & $5.1\cdot10^{-4}$ & $6.3\cdot10^{-3}$ & $5.7\cdot10^{-2}$ & $7.1\cdot10^{-2}$ \\
   $\text{FMM}^2$ & $\Omega_8$ & $1.1\cdot10^{-4}$ & $1.5\cdot10^{-3}$ & $2.8\cdot10^{-2}$ & $7.1\cdot10^{-2}$ \\
    \bottomrule
  \end{tabular}}
\end{table}  

For a square-shaped interface, we would like to compare the result of \textit{ReSDF} with the level set function $\phi_4$ \eqref{eq:sq_phi} with the $\text{FMM}^2$. The network is a $4$-layer fully connected network with width $128$ which is trained on the collocation points $\Omega_6$. In Figure \ref{fig:rectangle}, the iso-contours of the level set function $\phi_4$ is presented and the exact solution and the result of \textit{ReSDF} are shown. As shown in Figure \ref{fig:rectangle}, the exact SDF contains discontinuities of the gradient along the diagonal lines of the zero level set and the sharp corners become rounded in the outer region. The singularities arising from these sharp corners make it difficult to approximate the exact SDF. The numerical result depicted in Figure \ref{fig:rectangle} confirms that the proposed model approximates both the smooth outer region and the sharp inner points. In order to see how good numerical solution it is, we compare the results with the second order FMM ($\text{FMM}^2$) in Table~\ref{tab:sq}. The $L^2$ errors of the SDF and its gradient predicted by \textit{ReSDF} on $\Omega_6$ are between the $L^2$ errors of the result of $\text{FMM}^2$ on $\Omega_6$ and $\Omega_8$. It means that the accuracy of \textit{ReSDF} on a small number of collocation points is comparable to $\text{FMM}^2$ on a large number of collocation points and it is consistent with previous findings in \nameref{subsec:circle}. 


\subsection*{Example 3}\label{subsec:irr}

\begin{figure}
	\begin{center}
		\begin{tabular}{cc}
		    \includegraphics[height=0.35\textwidth]{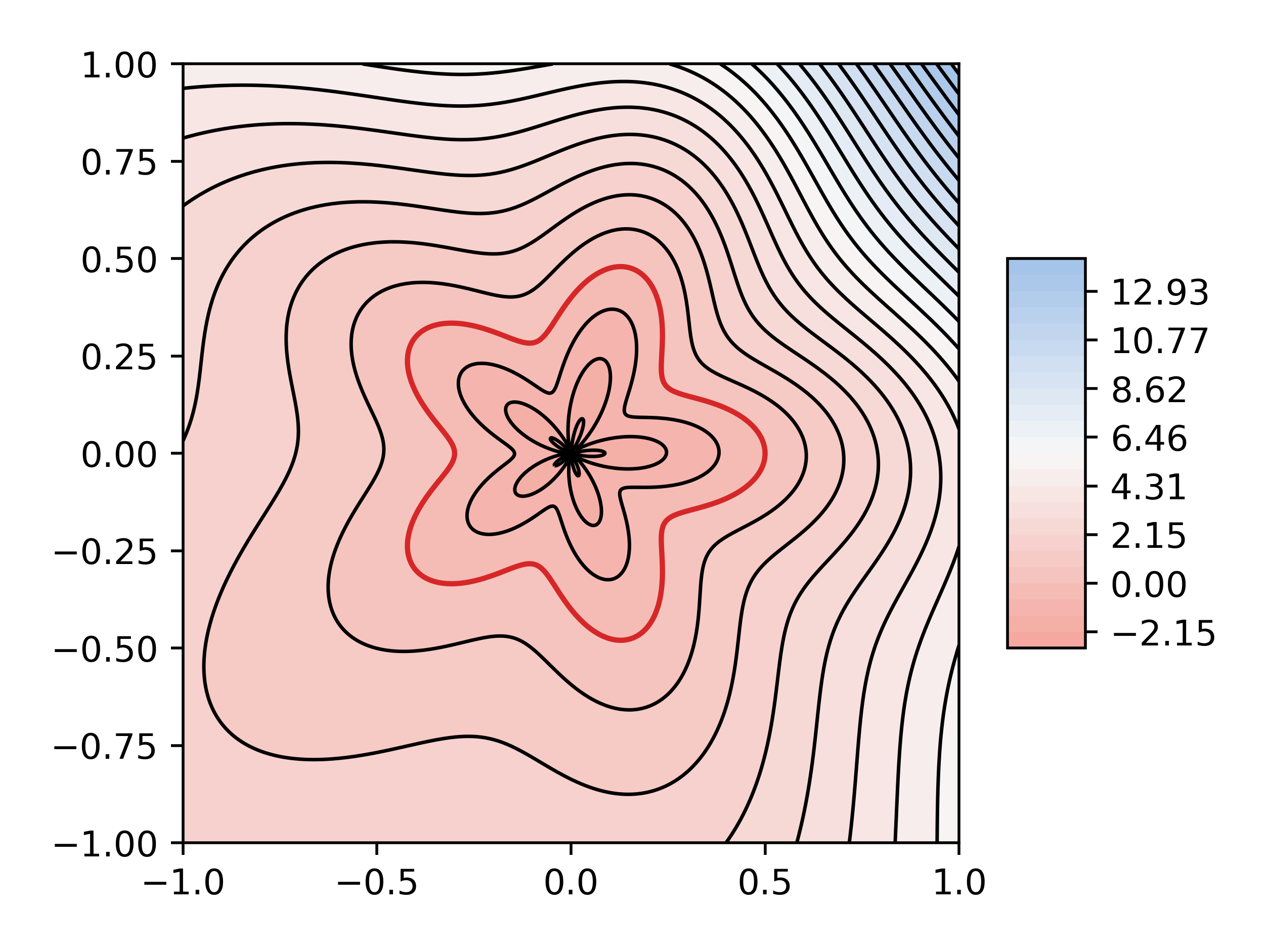} &
		    \includegraphics[height=0.35\textwidth]{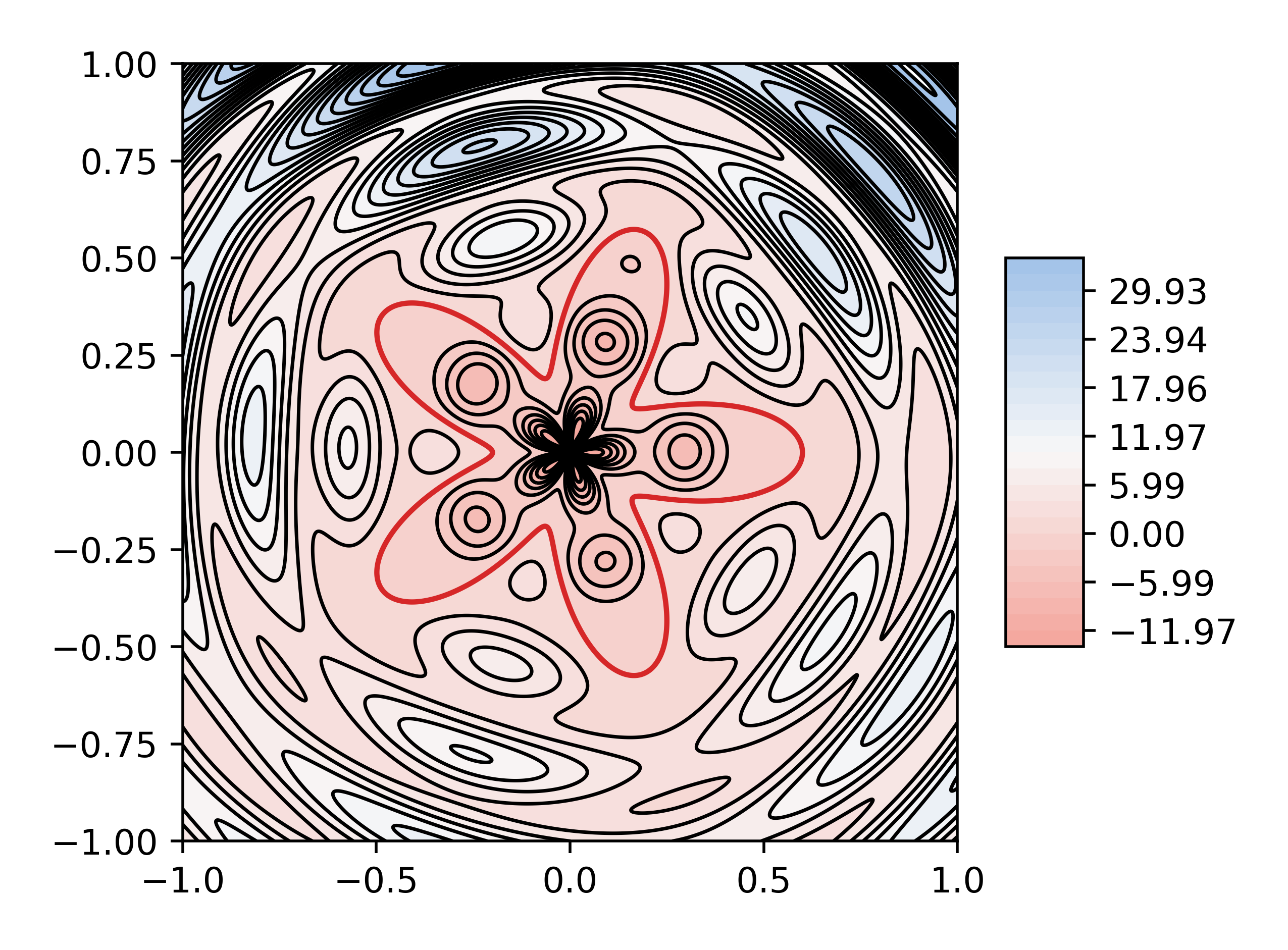} \\
		    (a) $\phi_5^{0.5}$ & (b) $\phi_5^{1}$ \\
		    \includegraphics[height=0.35\textwidth]{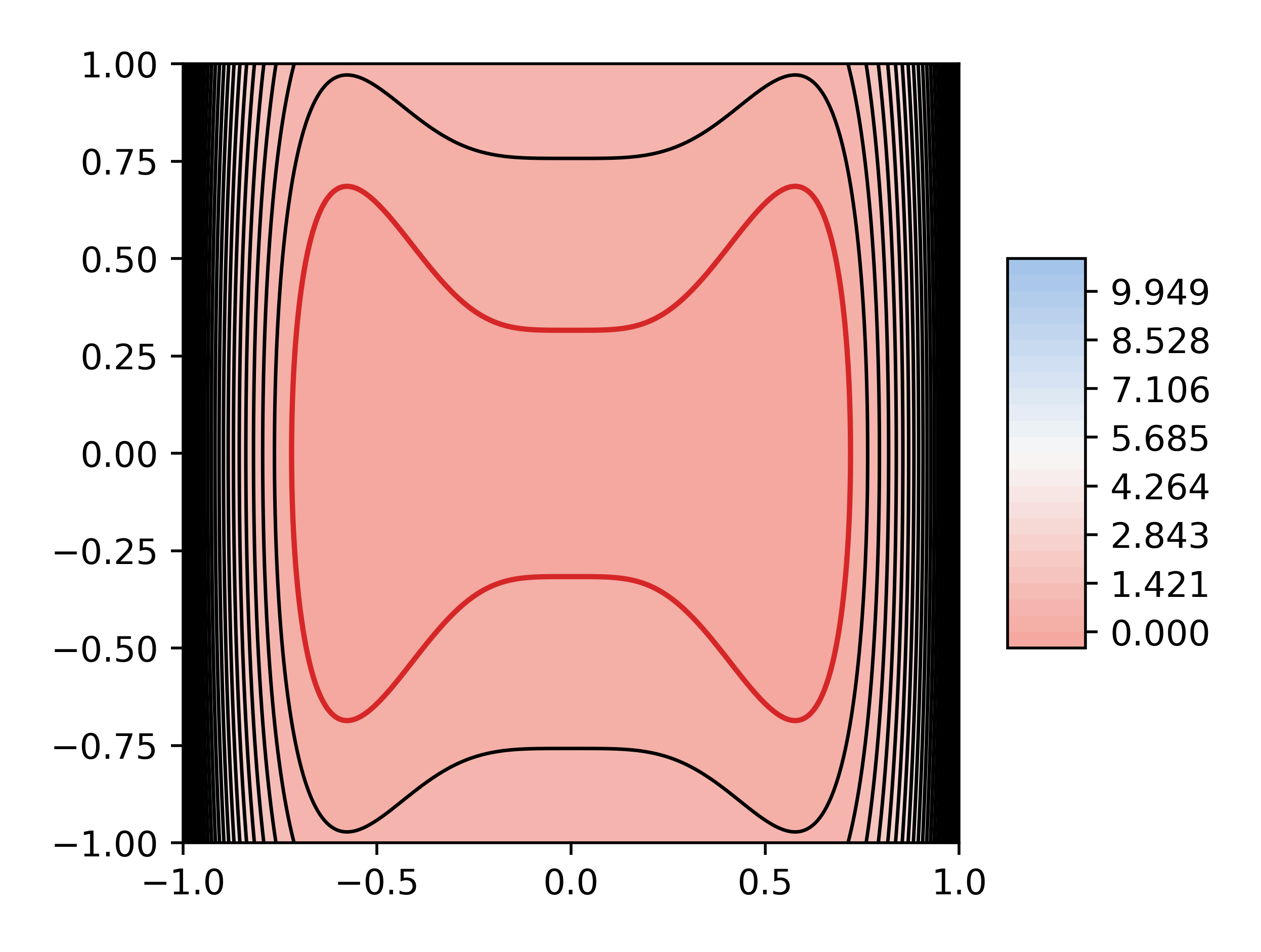} & 
		    \includegraphics[height=0.35\textwidth]{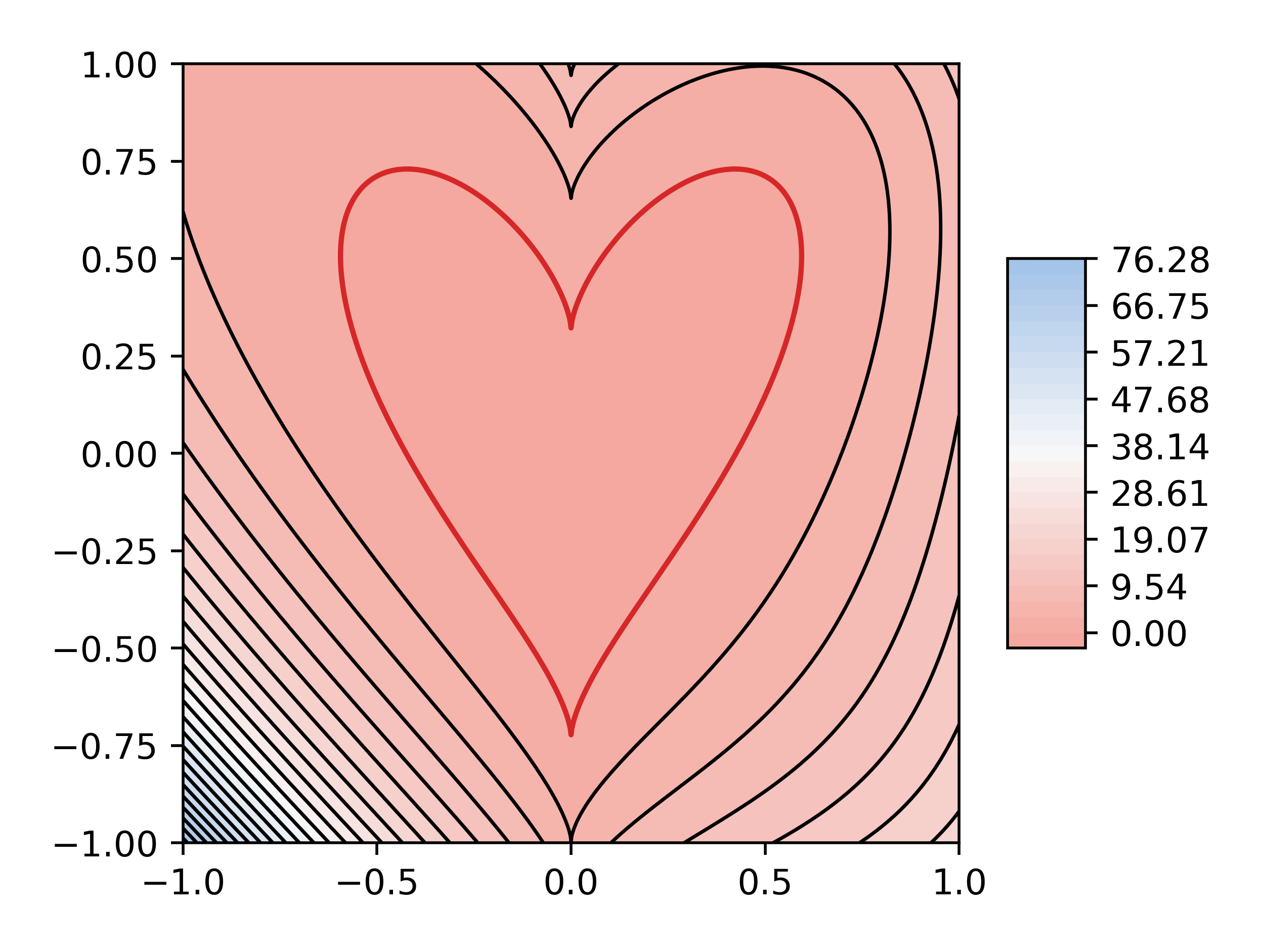} \\
		    (c) $\phi_6$ & (d) $\phi_7$
		\end{tabular}
	\end{center}
	\caption{The iso-contours of the level set functions, $\phi_\alpha^{0.5}$ \eqref{eq:flower} with $\alpha = 0.5$ and $1$, $\phi_6$ \eqref{eq:dumbbell}, and $\phi_7$ \eqref{eq:heart} are presented. } \label{fig:phi_irr}
\end{figure}

To demonstrate that the proposed model can be applied to various interfaces, the level set functions, $\phi_5^{\alpha}$ \eqref{eq:flower} with $\alpha = 0.5$ and $1$, $\phi_6$ \eqref{eq:dumbbell}, and $\phi_7$ \eqref{eq:heart}, are used to be reinitialized; see in Figure \ref{fig:phi_irr}. Due to the rapid change of the gradient field of the level set functions $\phi_5^{1}$ and $\phi_6$, the clipping parameter $M=5$ is used to improve the representability of $V_\theta$. Other parameters are consistent with other cases. 

\begin{figure}
	\begin{center}
		\begin{tabular}{cc}
		    \includegraphics[height=0.35\textwidth]{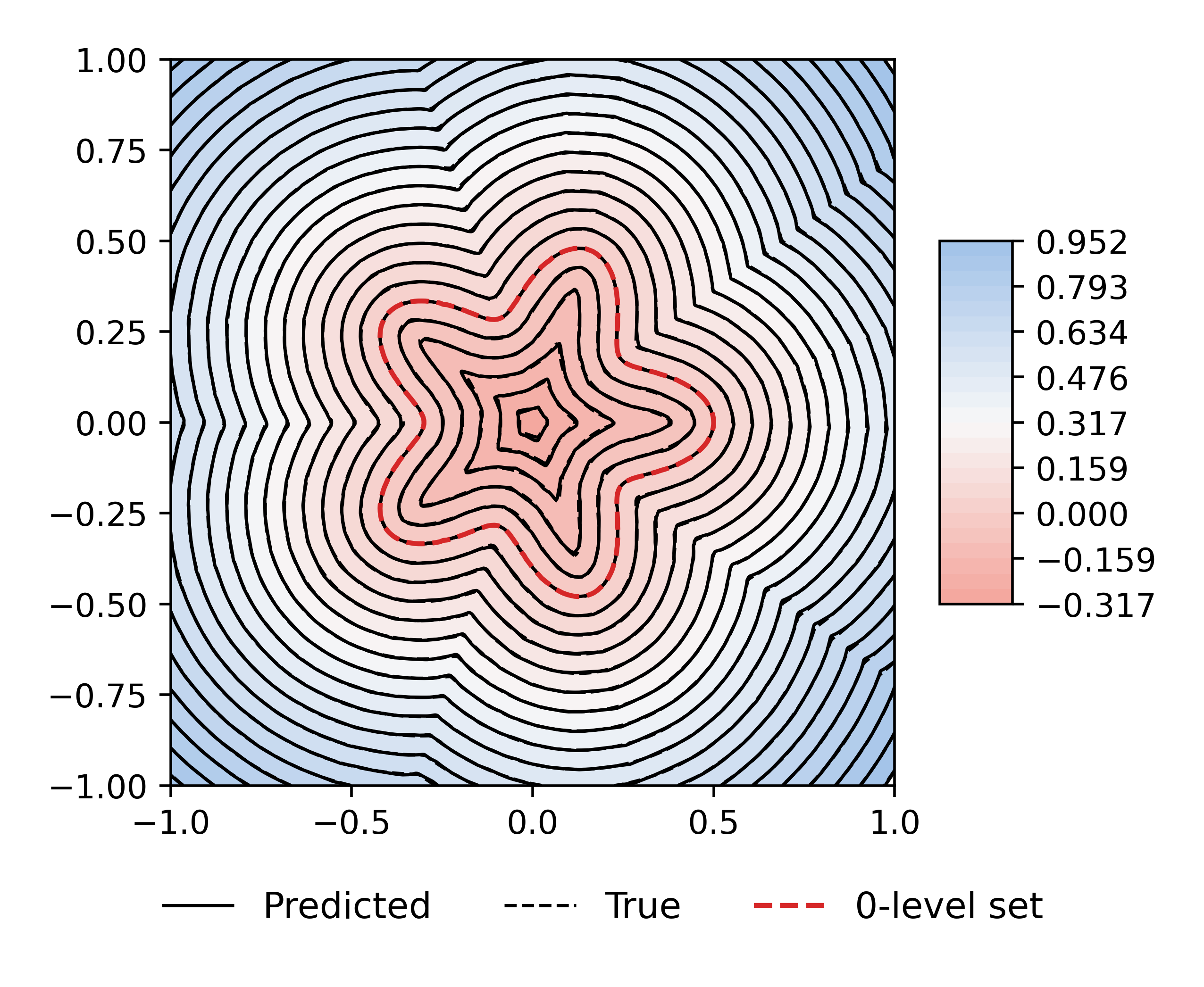} &
		    \includegraphics[height=0.35\textwidth]{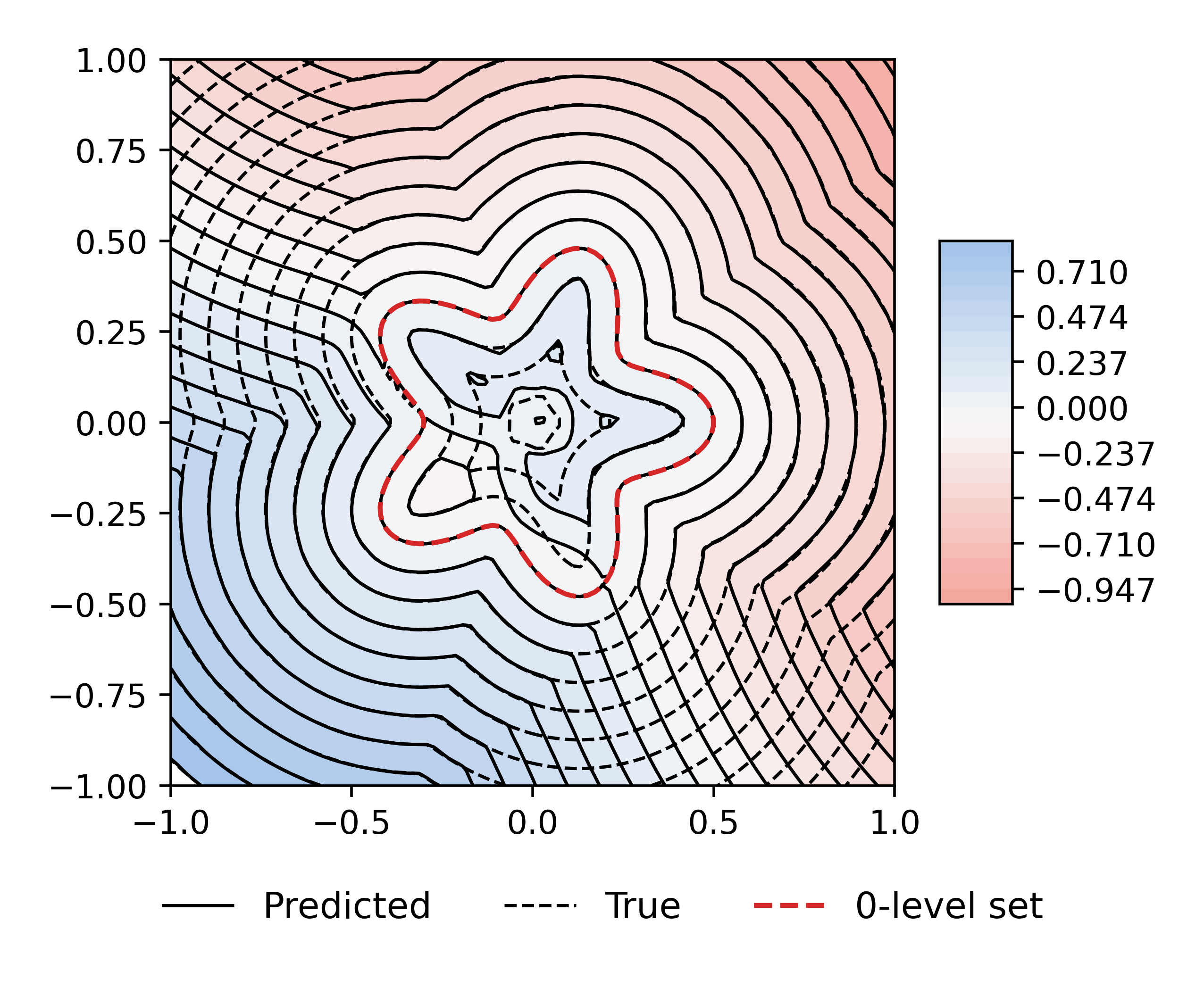} \\
		    (a1) \textit{ReSDF} with $\phi_5^{0.5}$ & (b1) PINN with $\phi_5^{0.5}$ \\
		    \includegraphics[height=0.35\textwidth]{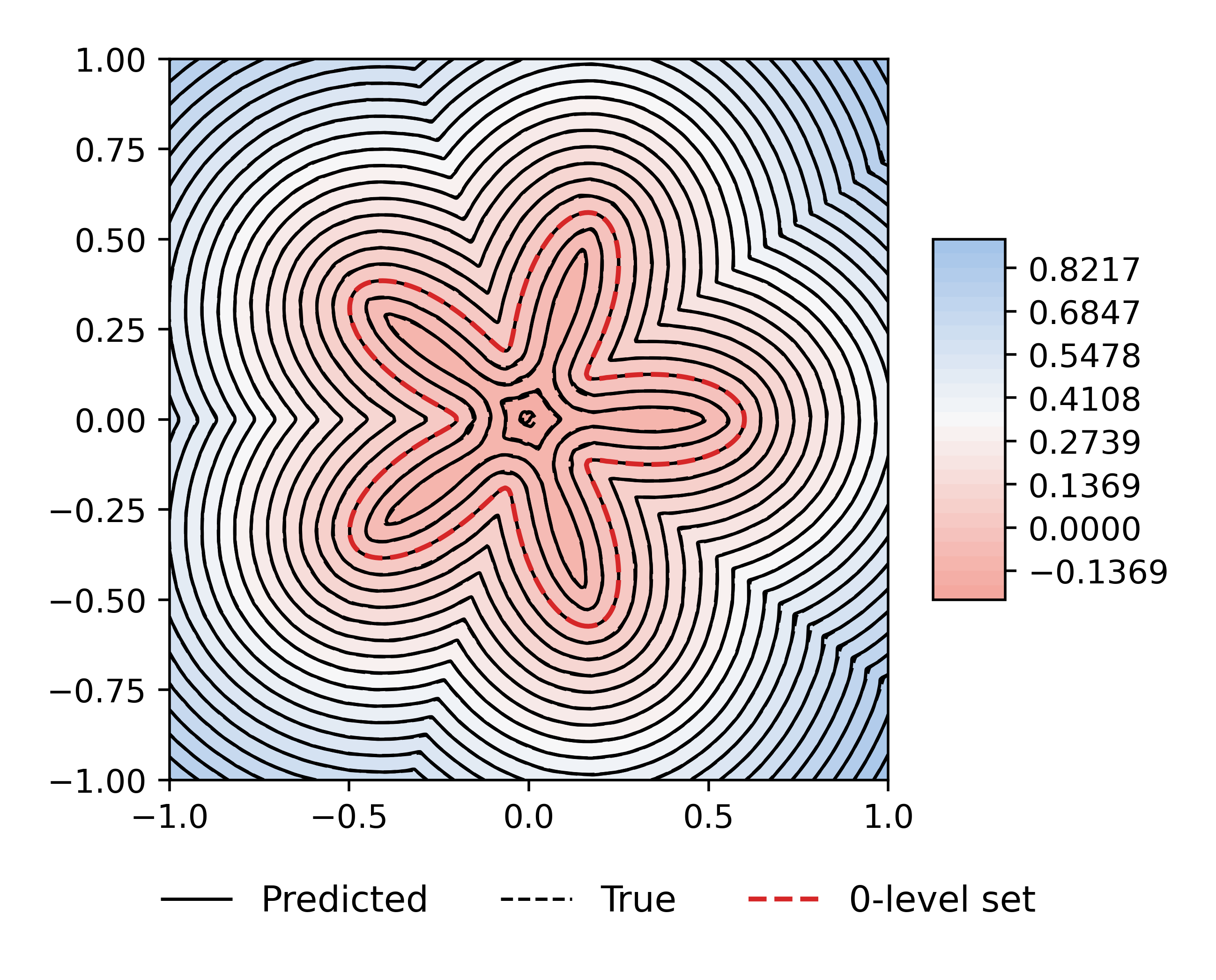} & 
		    \includegraphics[height=0.35\textwidth]{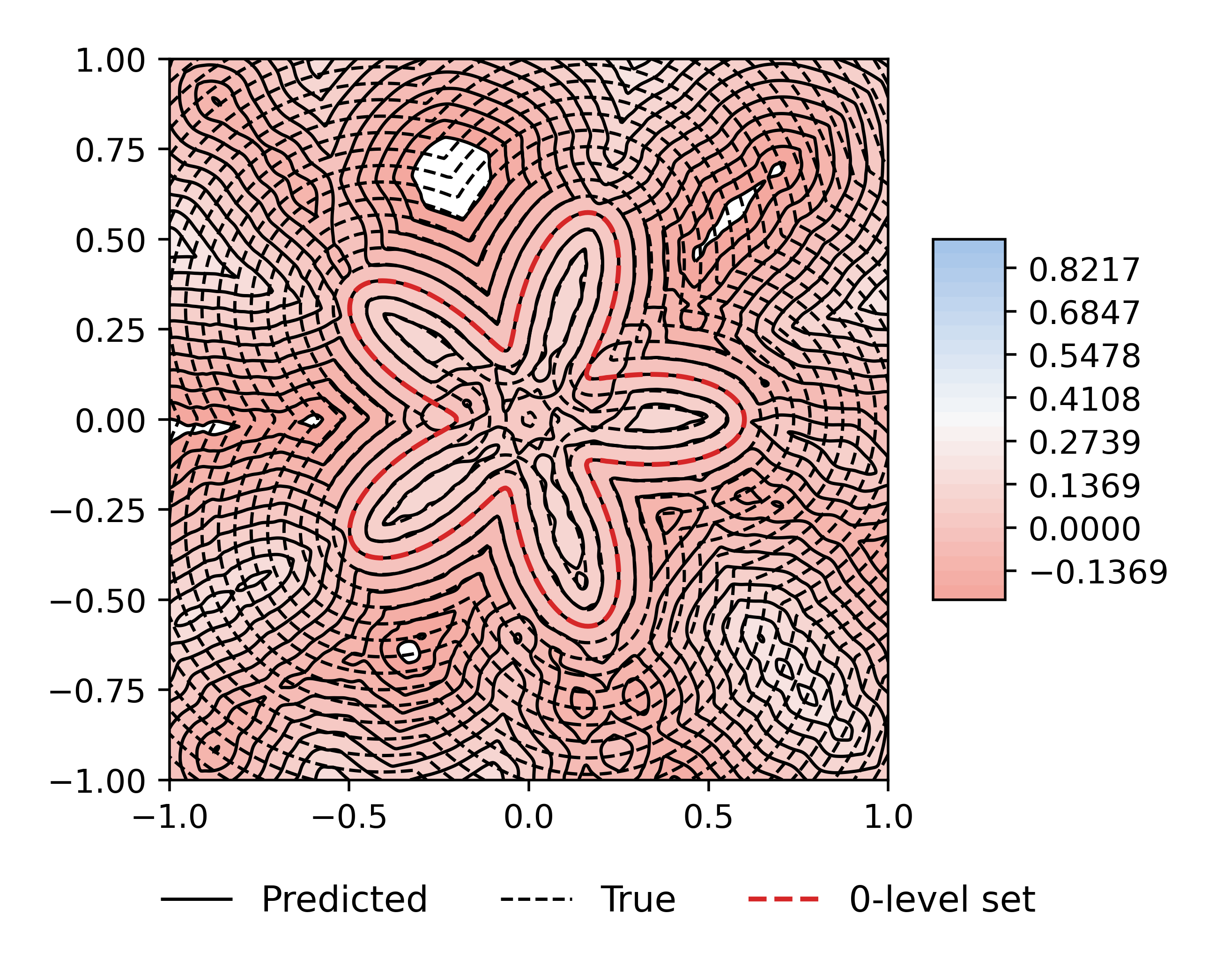} \\
		    (a2) \textit{ReSDF} with $\phi_5^{1}$ & (b2) PINN with $\phi_5^{1}$ \\
		    \includegraphics[height=0.35\textwidth]{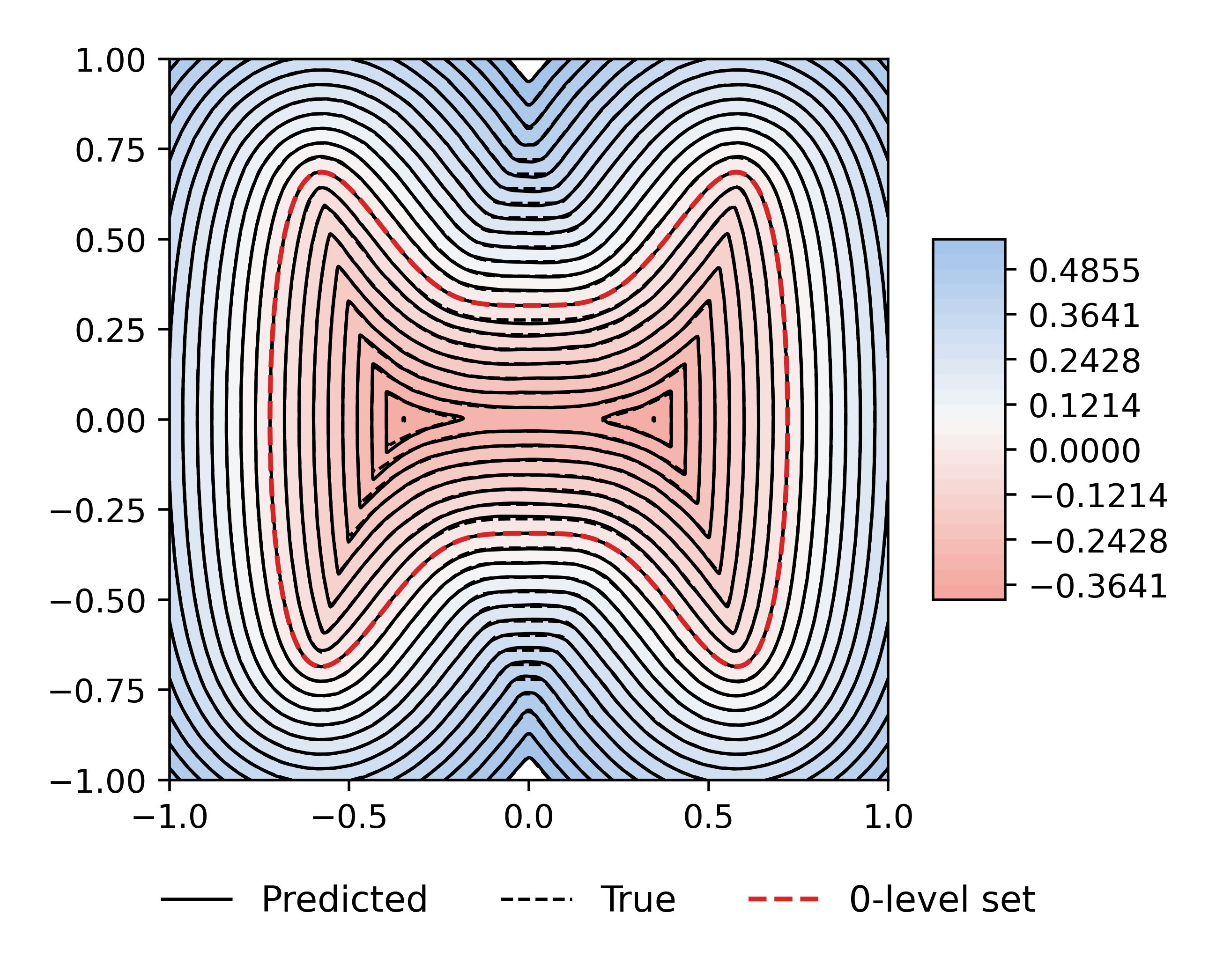} & 
		    \includegraphics[height=0.35\textwidth]{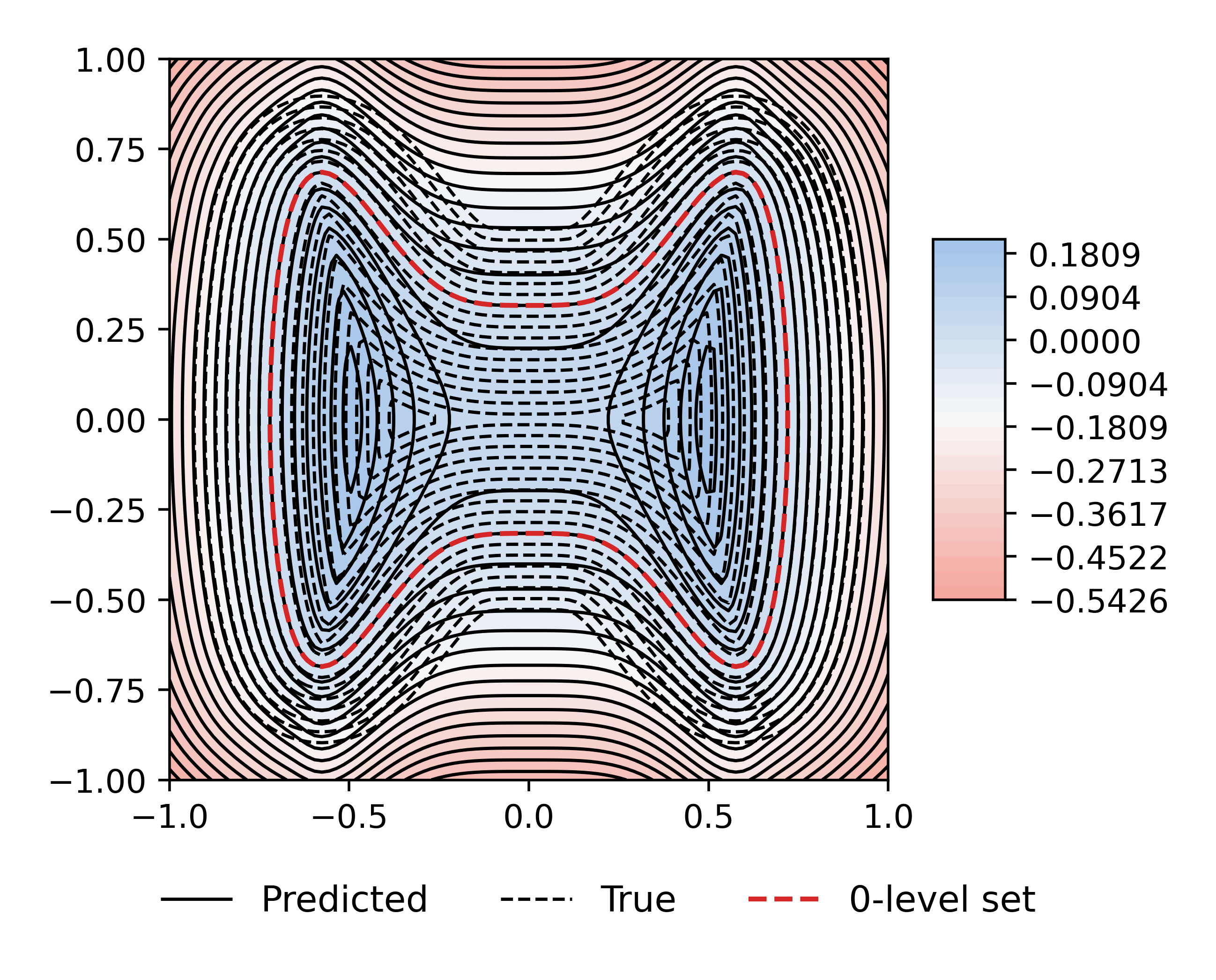} \\
		    (a3) \textit{ReSDF} with $\phi_6$ & PINN with $\phi_6$
		\end{tabular}
	\end{center}
	\caption{The iso-contours of the results of \textit{ReSDF} (left) and PINN approach (right) from three level set functions, $\phi_5^{\alpha}$ \eqref{eq:flower} with $\alpha = 0.5$ and $1$ and $\phi_6$ \eqref{eq:dumbbell} are presented with the approximated exact solution. The red solid curve is the zero level set of mentioned level set functions.} \label{fig:comp_irr}
\end{figure}

\begin{figure}
	\begin{center}
		\begin{tabular}{cc}
			\includegraphics[height=0.35\textwidth]{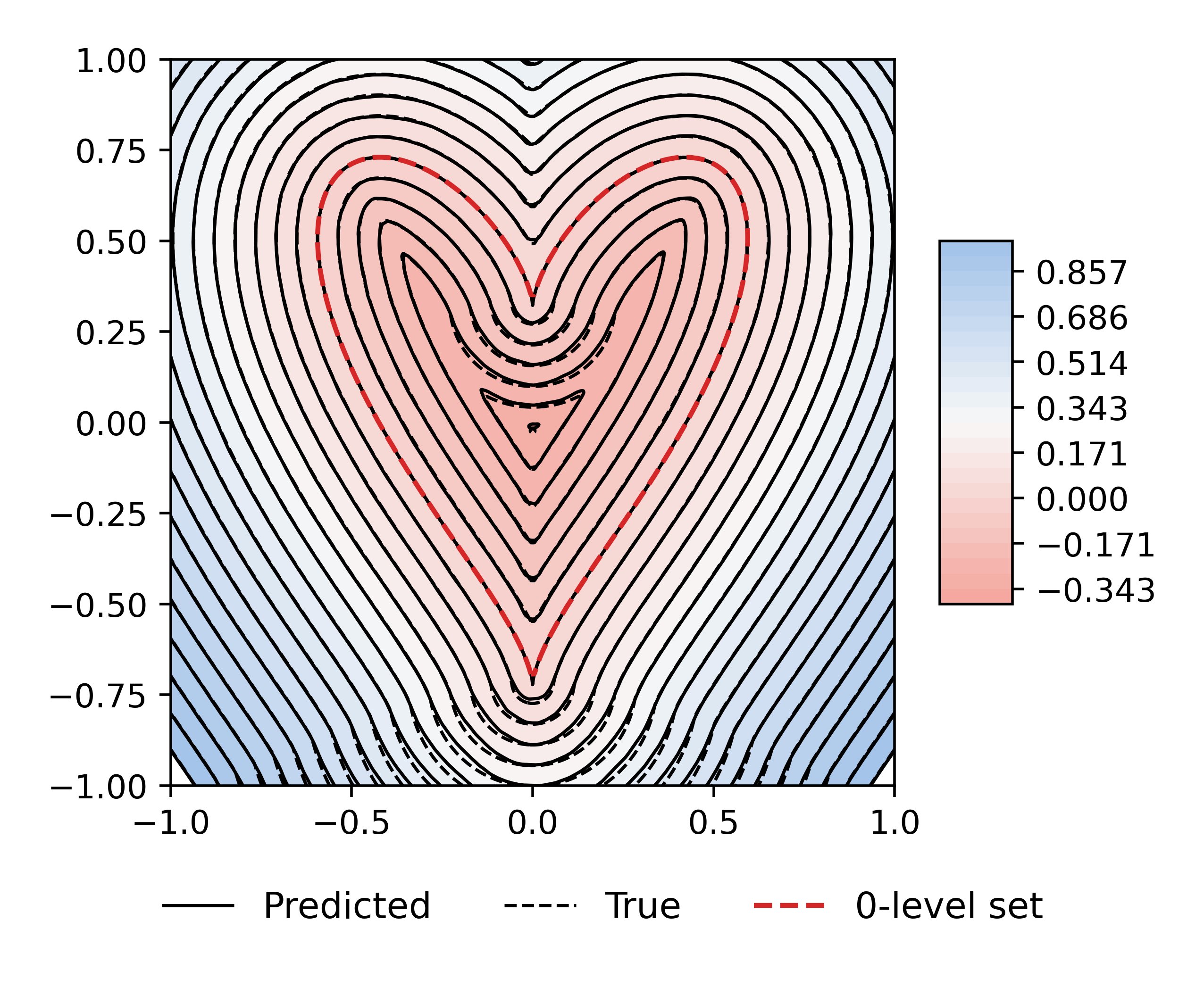} & 
			\includegraphics[height=0.35\textwidth]{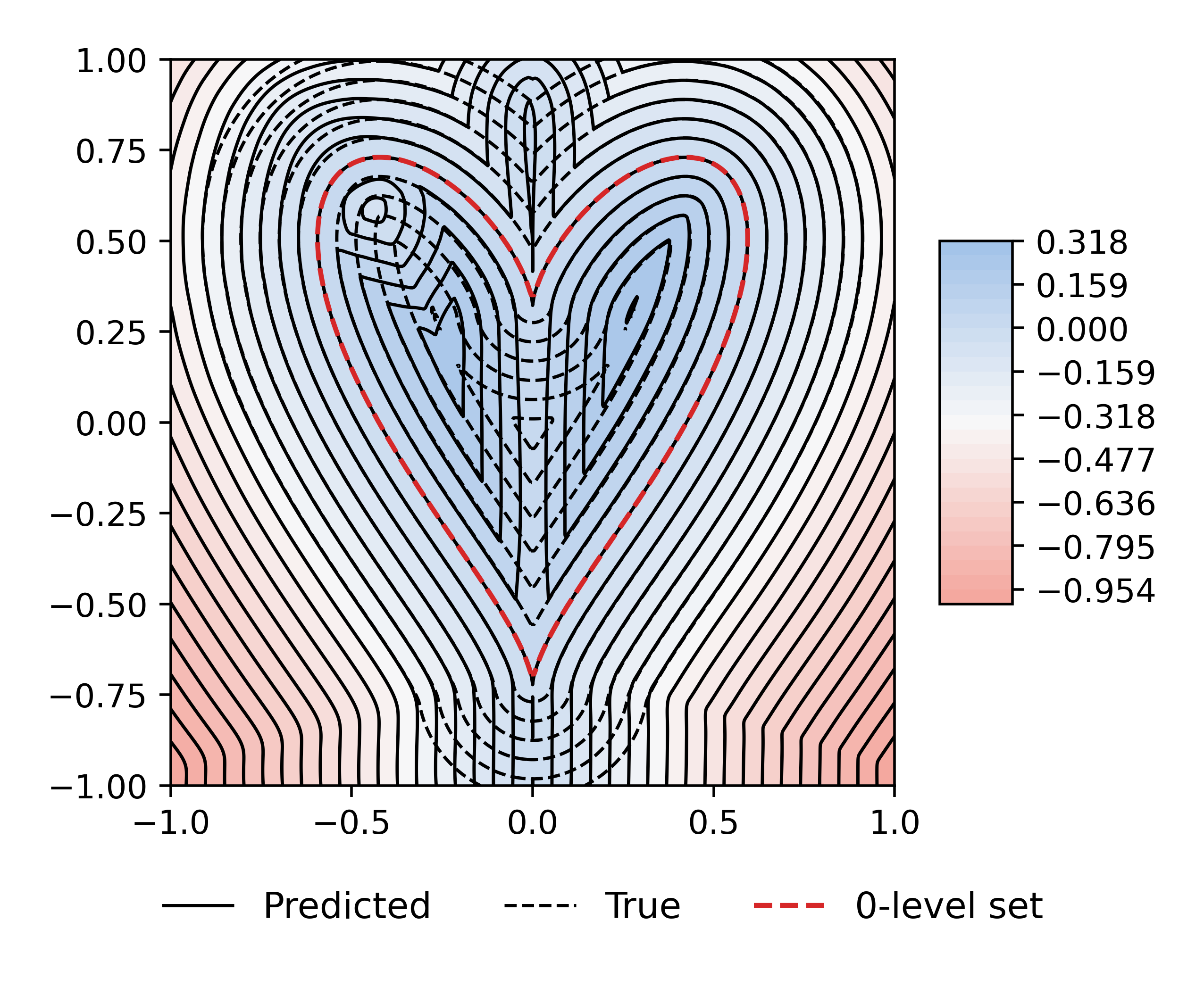} \\
			\textit{ReSDF} on $\Omega_6$ & PINN on $\Omega_7$\\
			\includegraphics[height=0.35\textwidth]{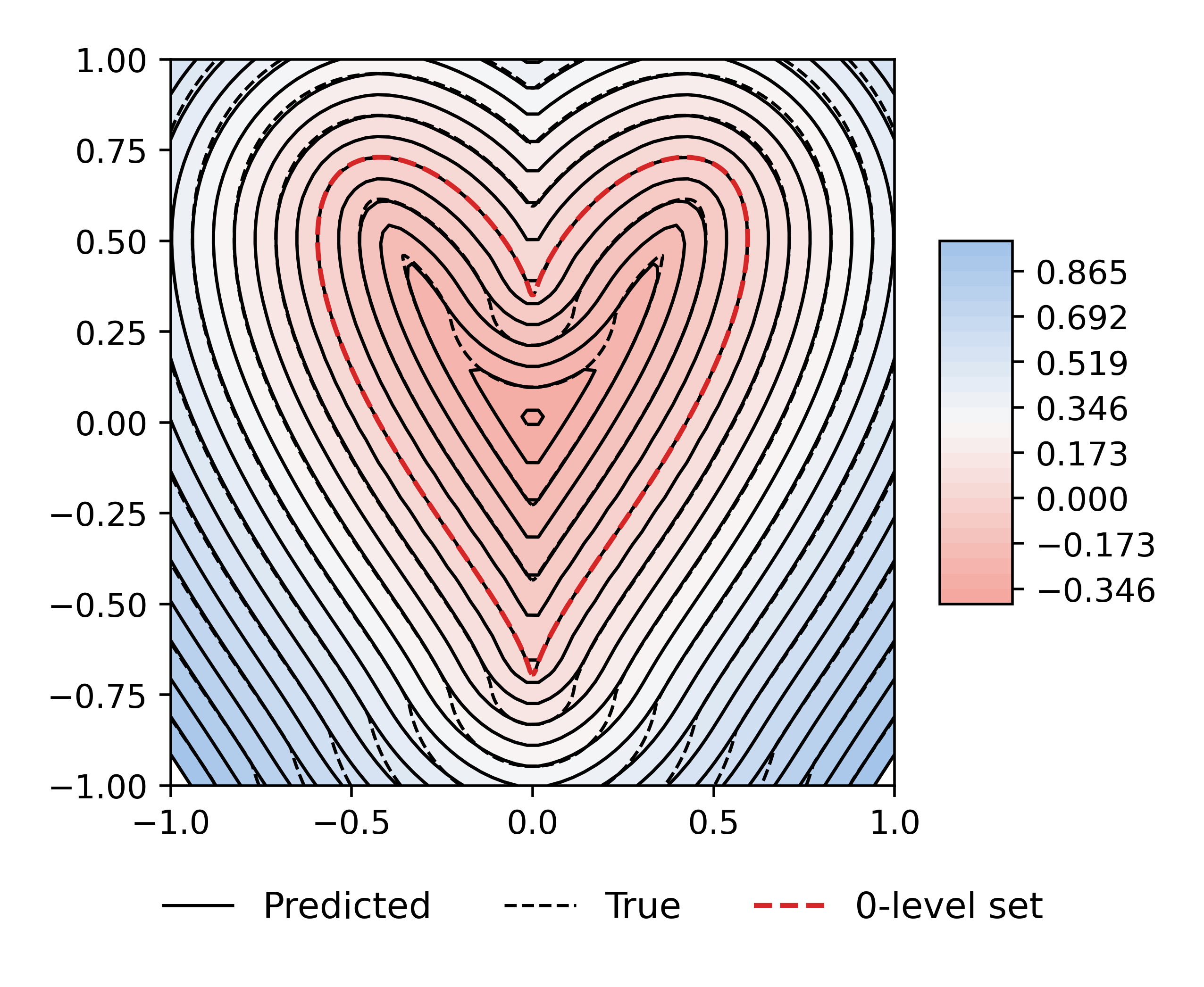} & 
			\includegraphics[height=0.35\textwidth]{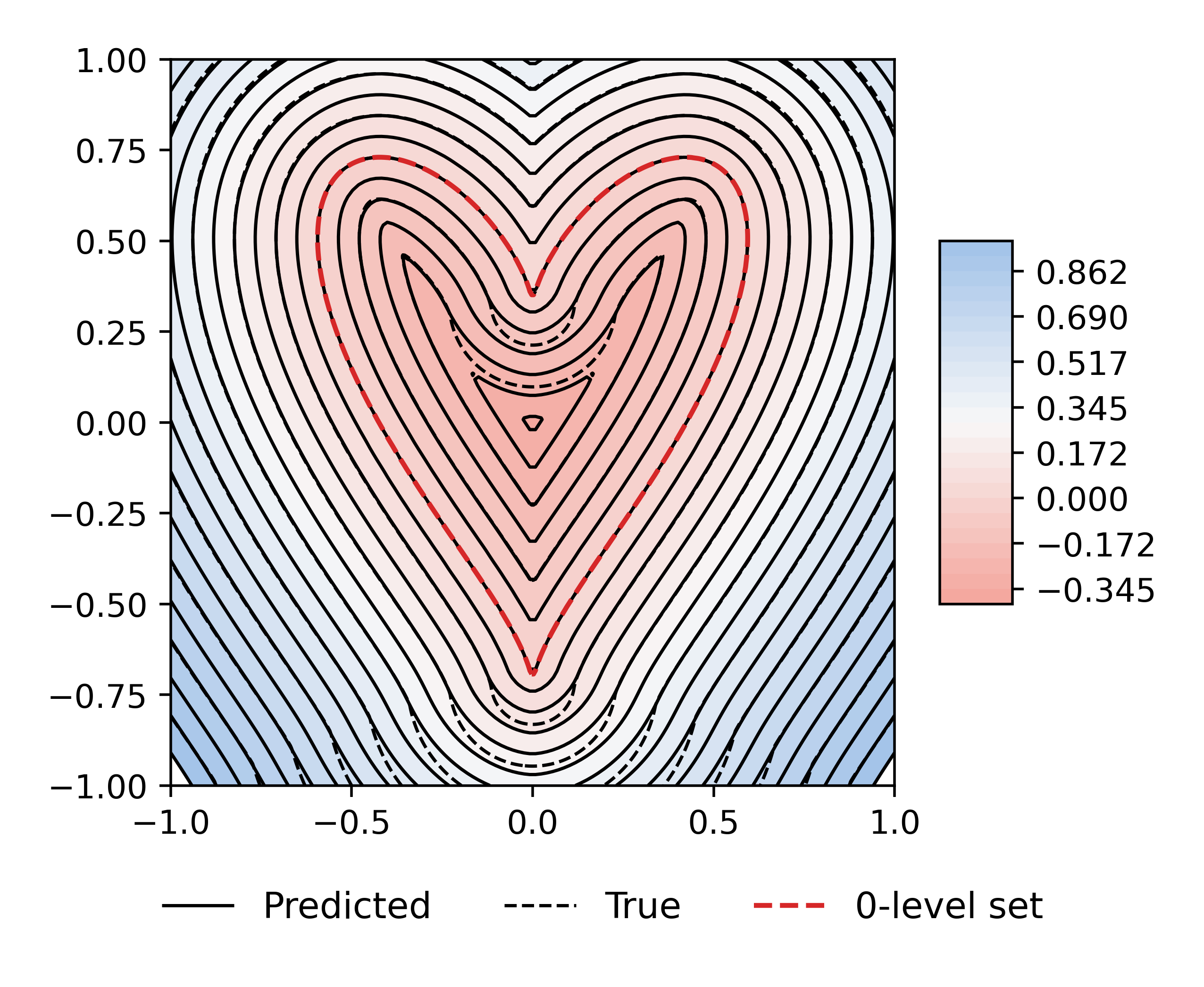} \\
			FMM on $\Omega_6$ & FMM on $\Omega_7$ \\
			\includegraphics[height=0.35\textwidth]{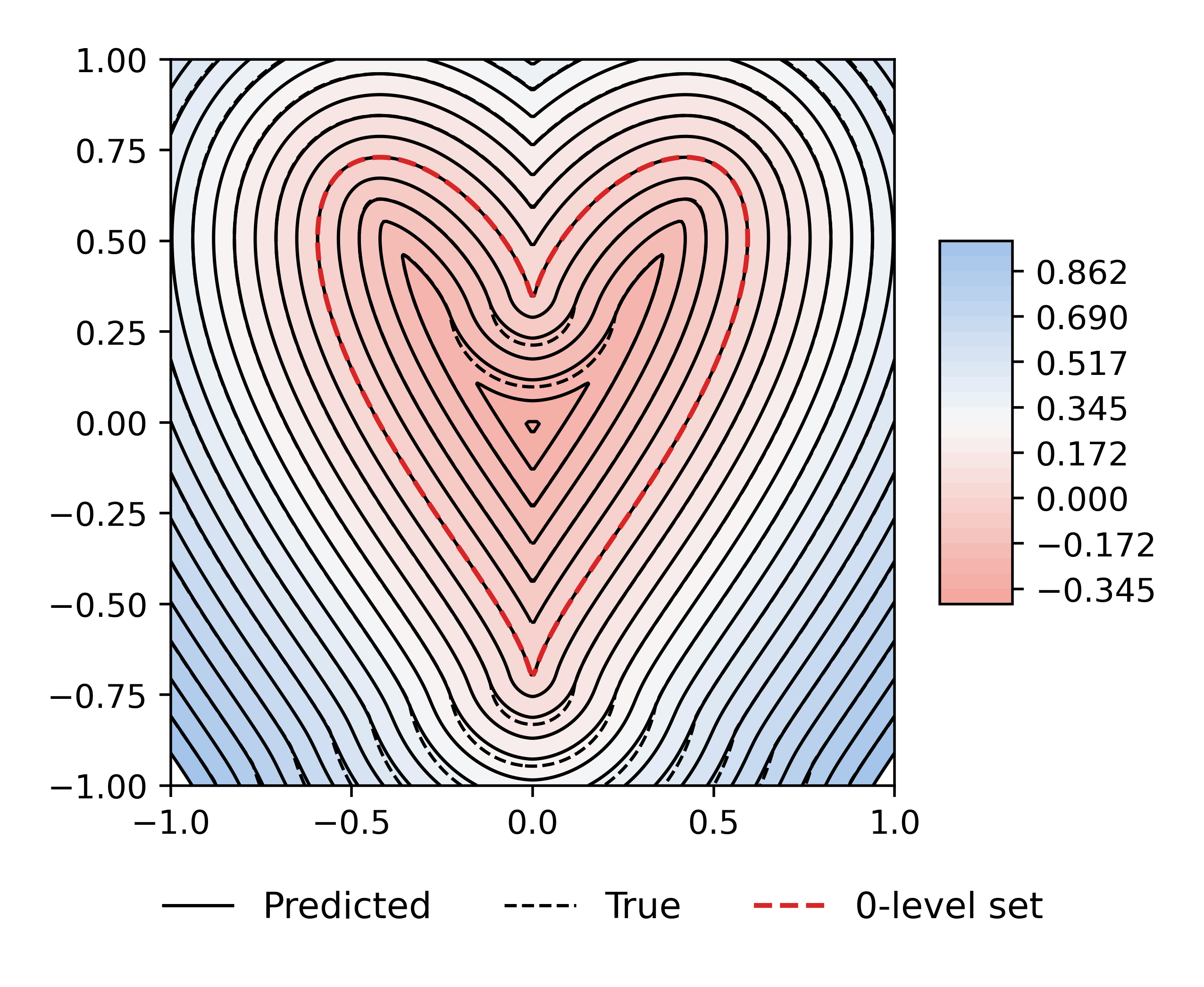} &
			\includegraphics[height=0.35\textwidth]{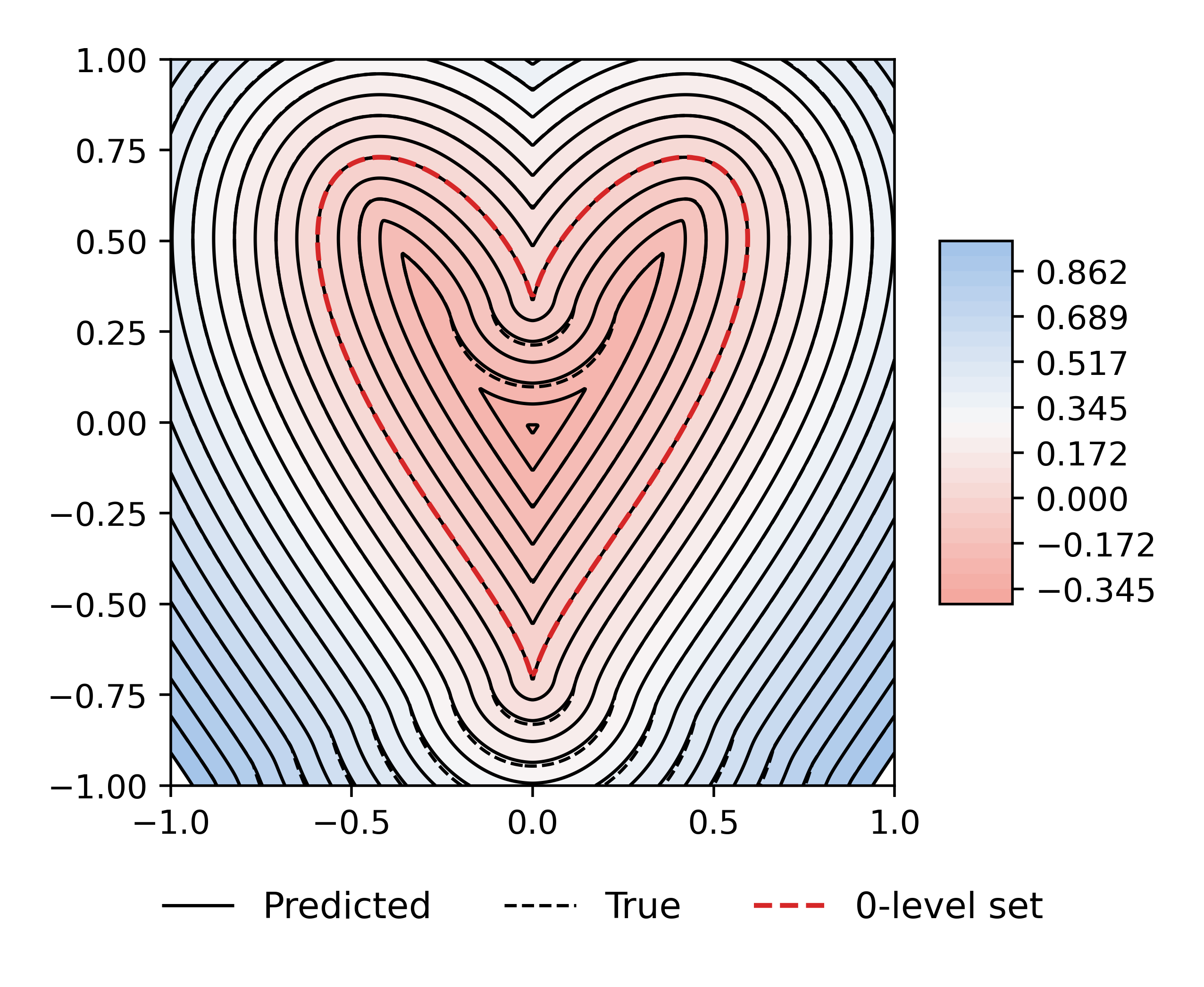} \\
			FMM on $\Omega_8$ & FMM on $\Omega_9$
		\end{tabular}
	\end{center}
	\caption{The result of \textit{ReSDF} with the level set function $\phi_7$ \ref{eq:heart} on $\Omega_6$ is presented and compared to PINN approach. The results of using FMM on $\Omega_n$, $n=6,\ldots,9$ are also presented.} \label{fig:heart_FMM}
\end{figure}

\begin{figure}
	\begin{center}
		\begin{tabular}{cc}
			\includegraphics[height=0.35\textwidth]{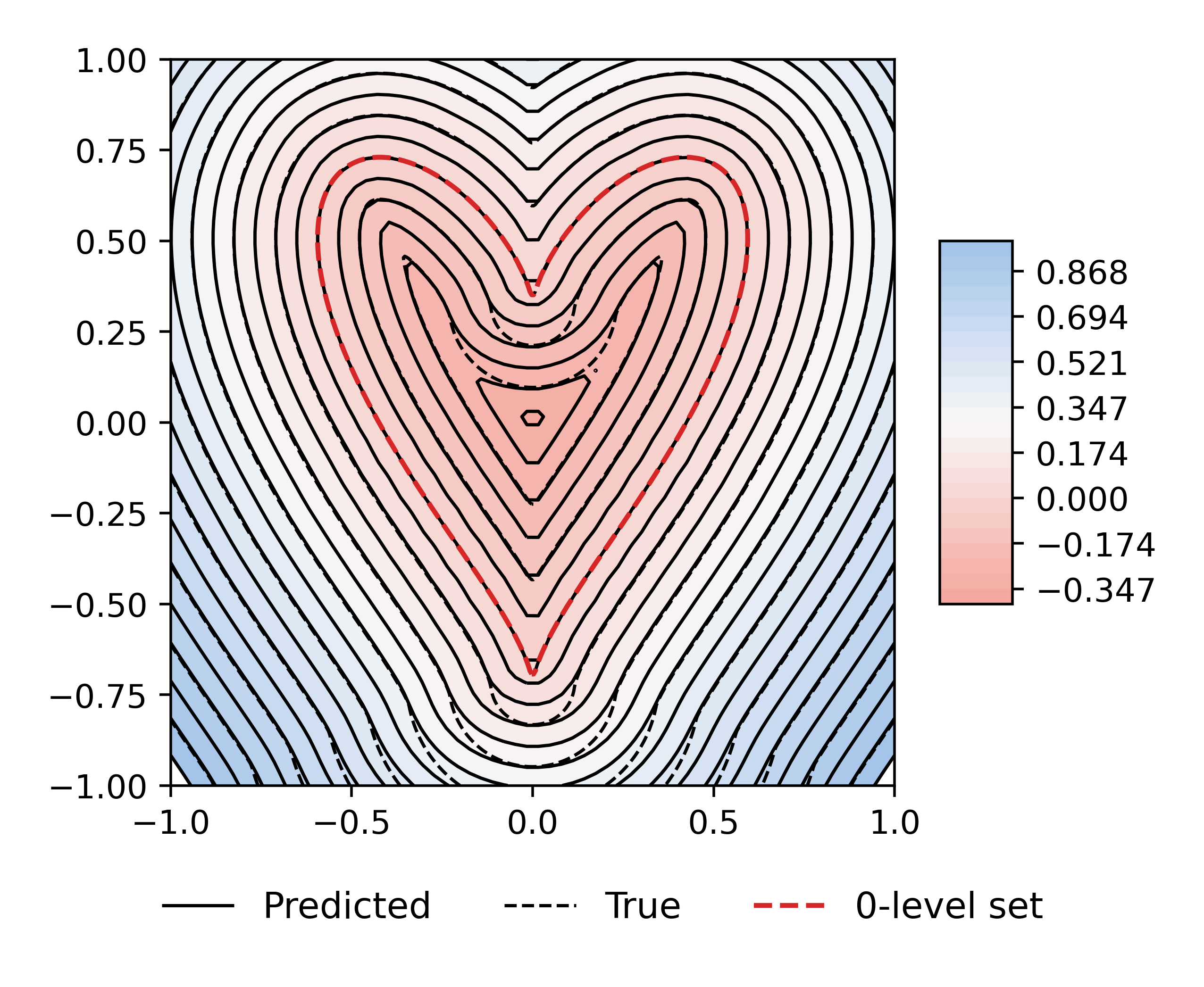} & 
			\includegraphics[height=0.35\textwidth]{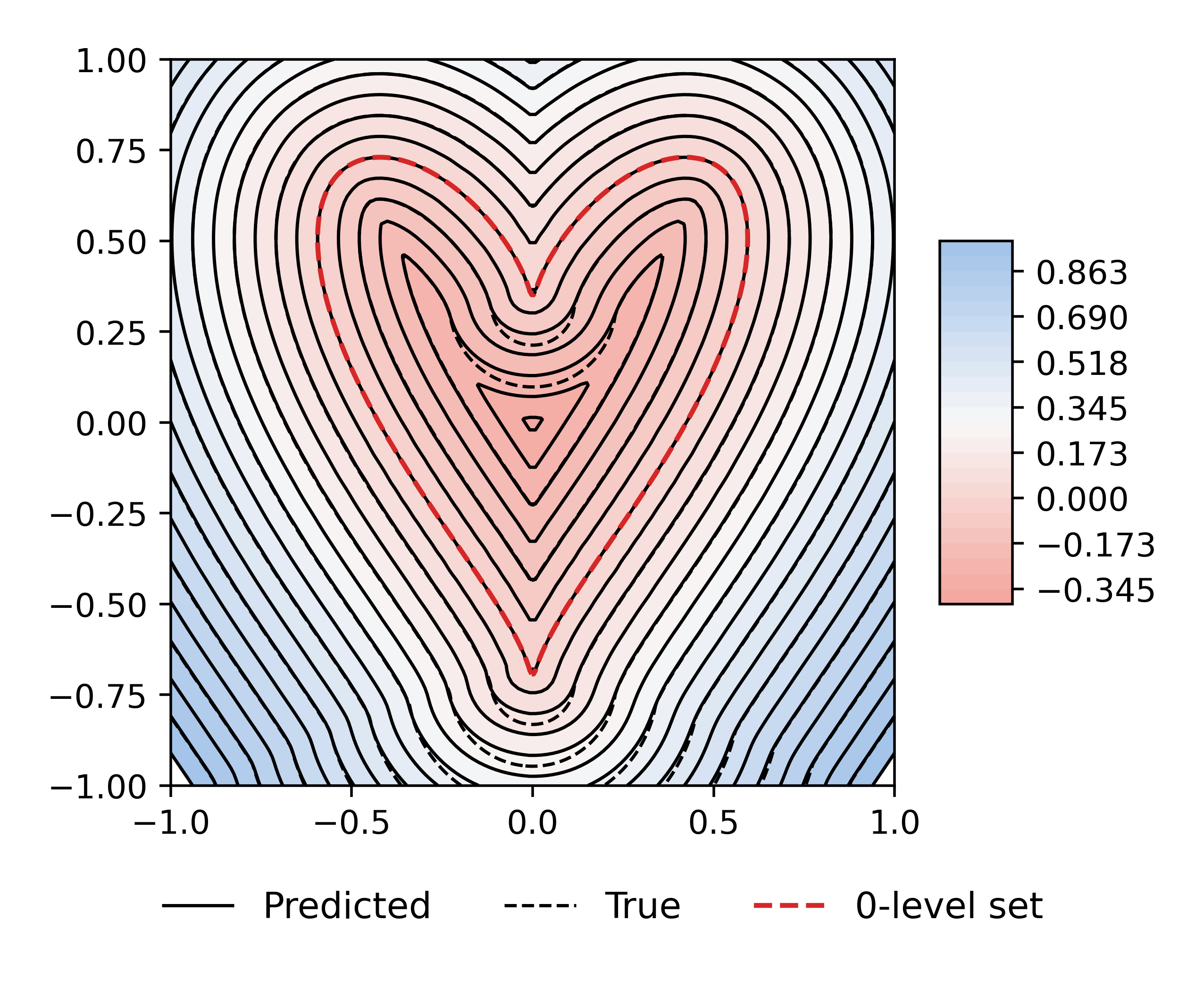} \\
			$\text{FMM}^2$ on $\Omega_6$ & $\text{FMM}^2$ on $\Omega_7$ \\
			\includegraphics[height=0.35\textwidth]{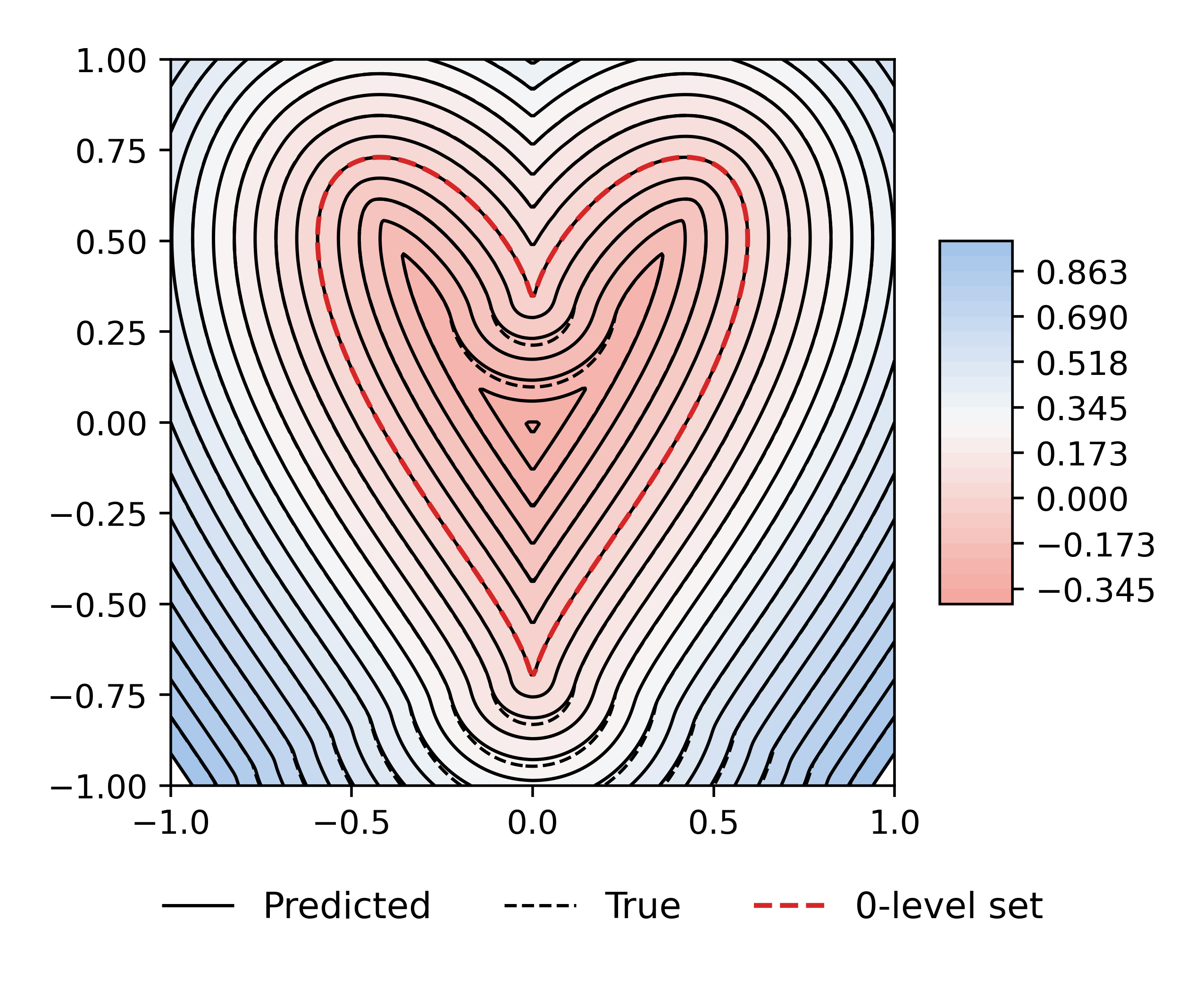} &
			\includegraphics[height=0.35\textwidth]{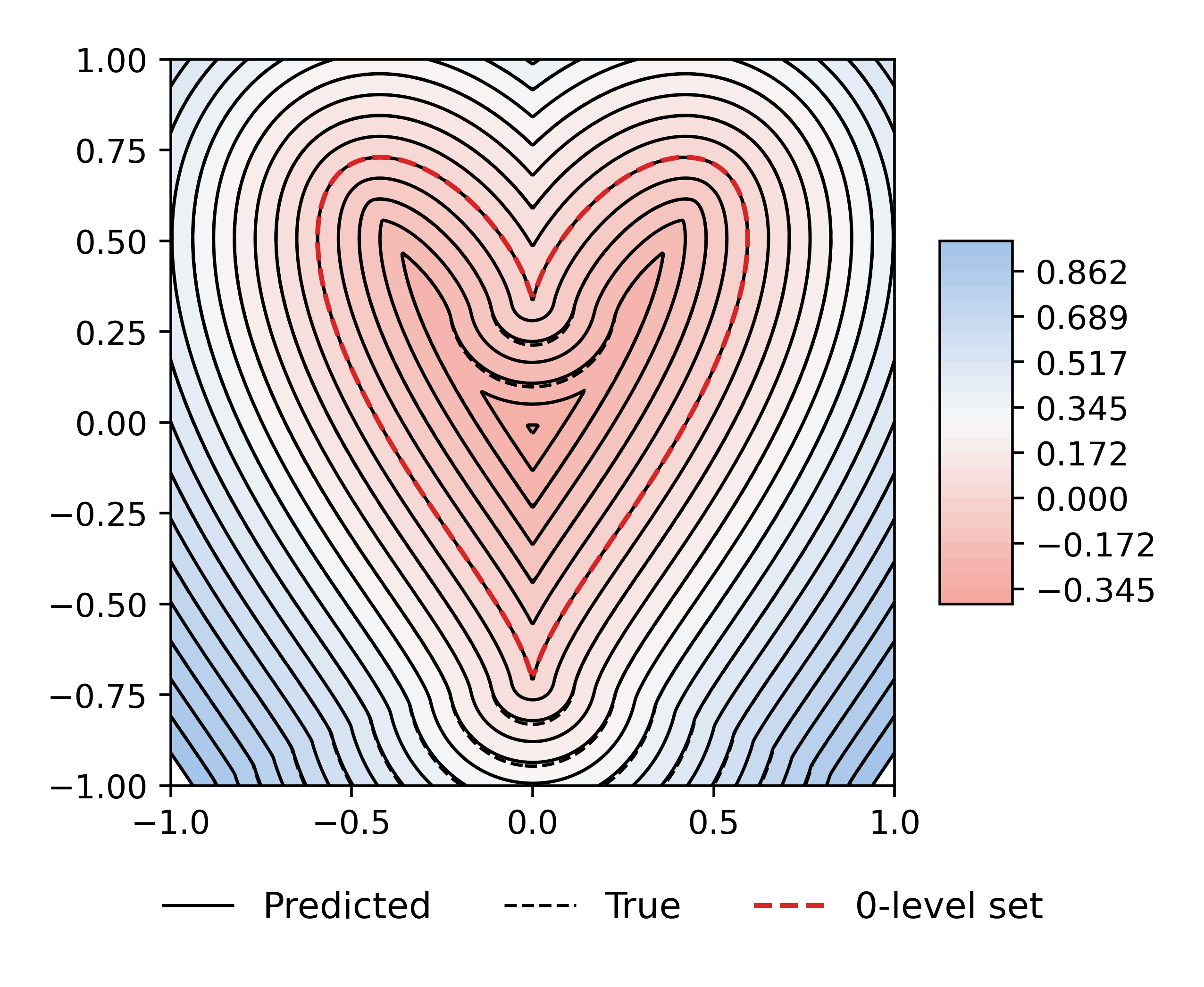} \\
			$\text{FMM}^2$ on $\Omega_8$ & $\text{FMM}^2$ on $\Omega_9$
		\end{tabular}
	\end{center}
	\caption{Signed distance functions of the zero level set of $\phi_7$ \ref{eq:heart} computed by the second-order FMM ($\text{FMM}^2$) on $\Omega_n$, $n=6,\ldots,9$ are presented.} \label{fig:heart_FMM2}
\end{figure}

Throughout the examples in Figure \ref{fig:comp_irr}, the results are compared with the existing PINN approach. The numerical results verify that the proposed method has better agreement with the exact SDF than PINN approach which is prone to returning erroneous predictions for chosen irregular interfaces. Furthermore, the result of \textit{ReSDF} with the level set function $\phi_7$ \ref{eq:heart} on $\Omega_6$ is qualitatively compared with the results of FMM and $\text{FMM}^2$ on $\Omega_n$, $n=6,\ldots,9$ in Figures \ref{fig:heart_FMM} and \ref{fig:heart_FMM2}, respectively. In the neural network of \textit{ReSDF}, we use the width $128$ and depth $4$. Since the batch size is small to $\Omega_6$, the learning rate is initialized by $5\cdot 10^{-4}$. FMM produces reliable results when the mesh is sufficiently refined. Comparing the results in Figure \ref{fig:heart_FMM}, the result of the first-order FMM on $\Omega_9$ is similar to the result of \textit{ReSDF} trained on the collocation points $\Omega_6$. Similarly, in Figure \ref{fig:heart_FMM2}, the result of \textit{ReSDF} on $\Omega_6$ is a comparable to the results of the second-order FMM on  between $\Omega_8$ and $\Omega_9$. 

\subsection*{Example 4}\label{subsec:mul_intf}

\begin{figure}
	\begin{center}
		\begin{tabular}{cc}
			\includegraphics[height=0.35\textwidth]{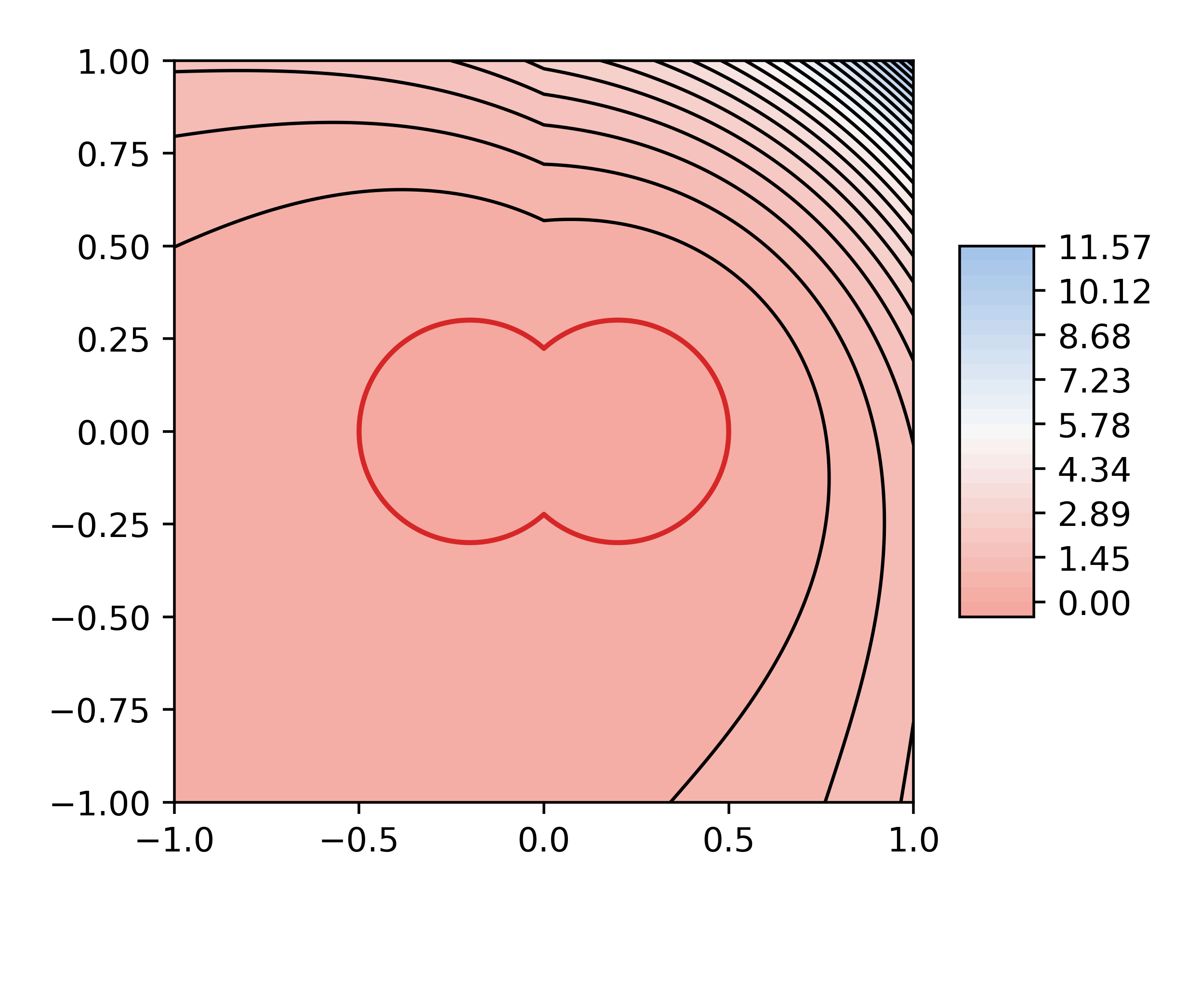} & 
			\includegraphics[height=0.35\textwidth]{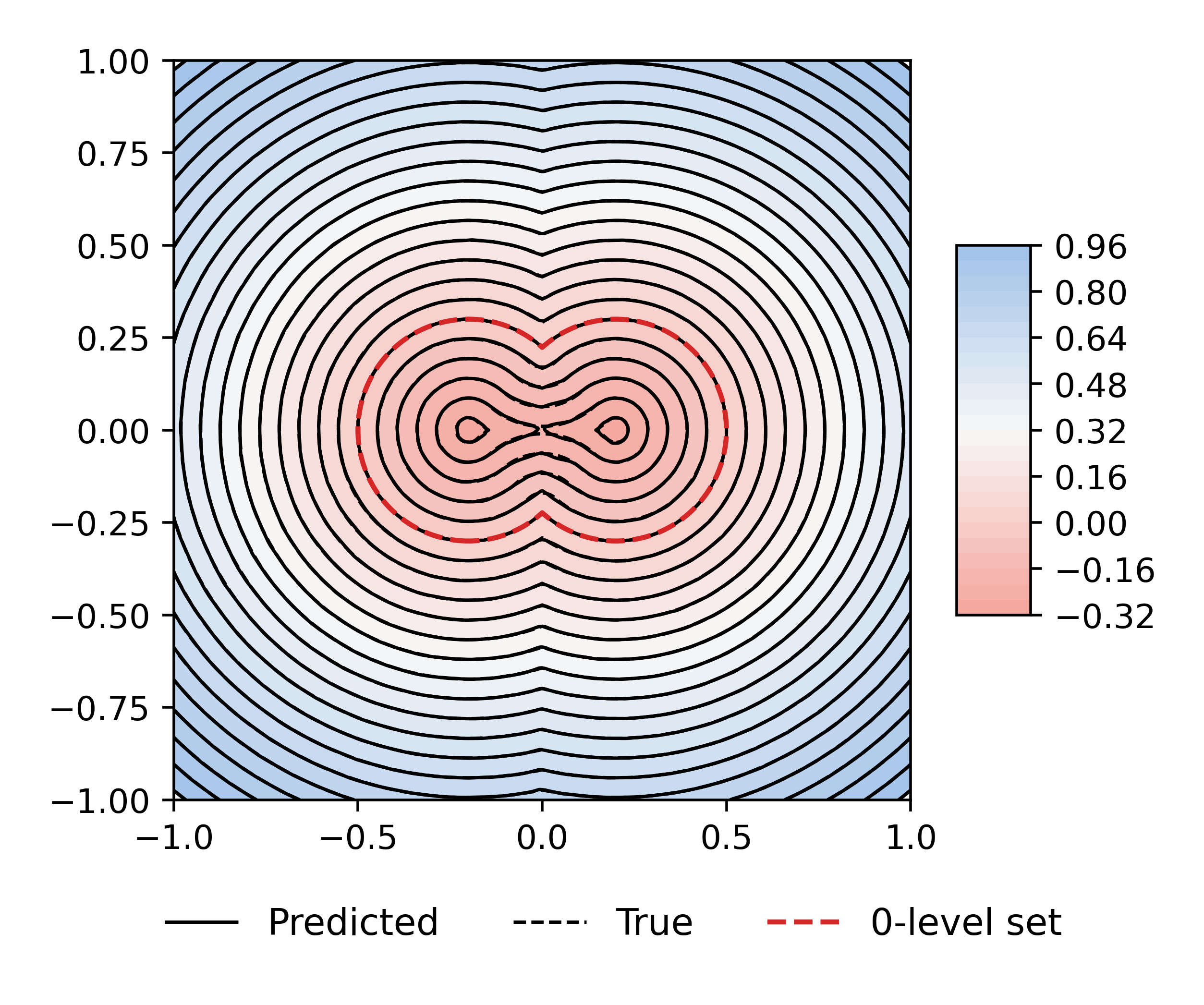} \\
			(a) $\phi_8$  & (b) \textit{ReSDF} with $\phi_8$ \\
			\includegraphics[height=0.35\textwidth]{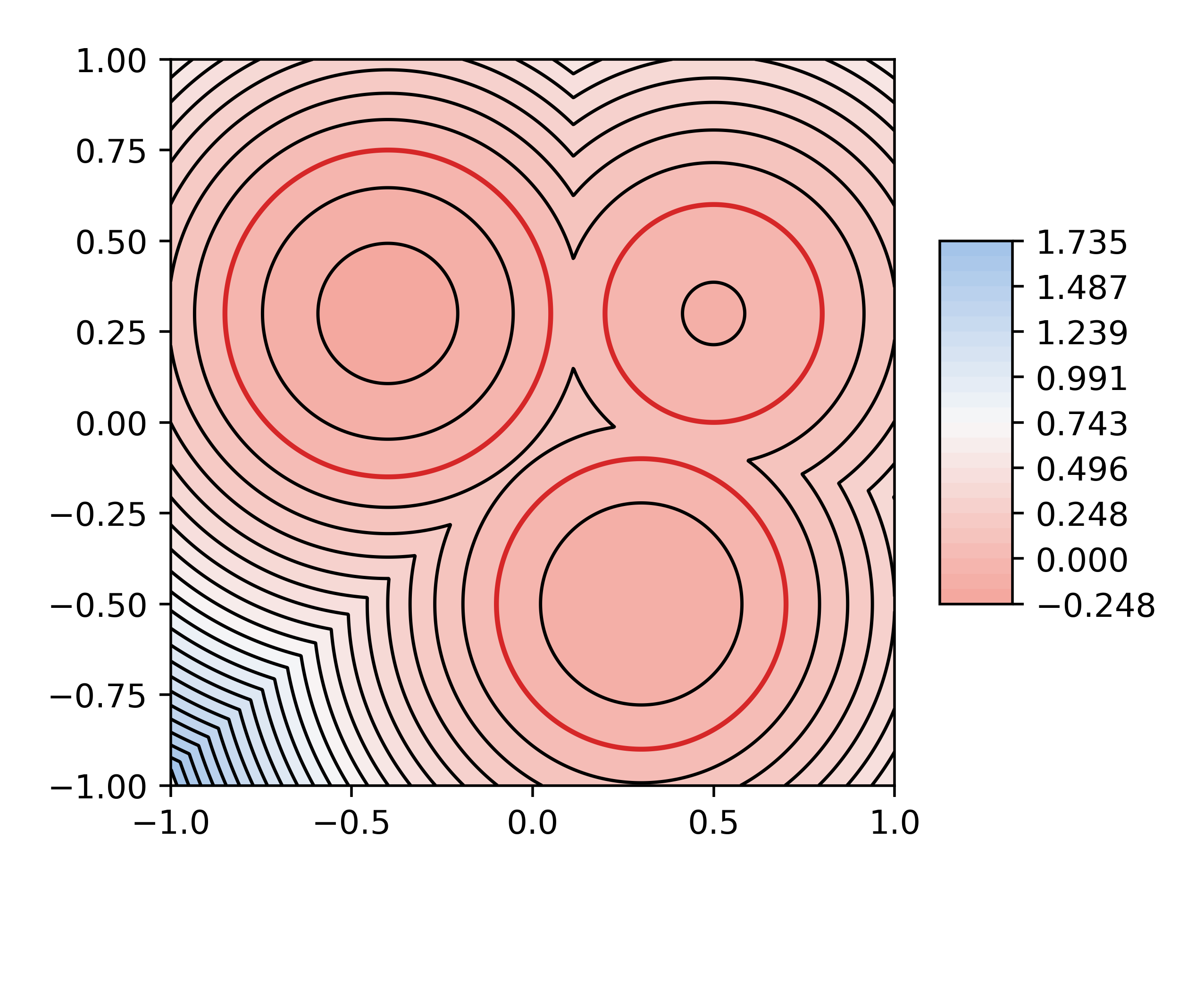} &
			\includegraphics[height=0.35\textwidth]{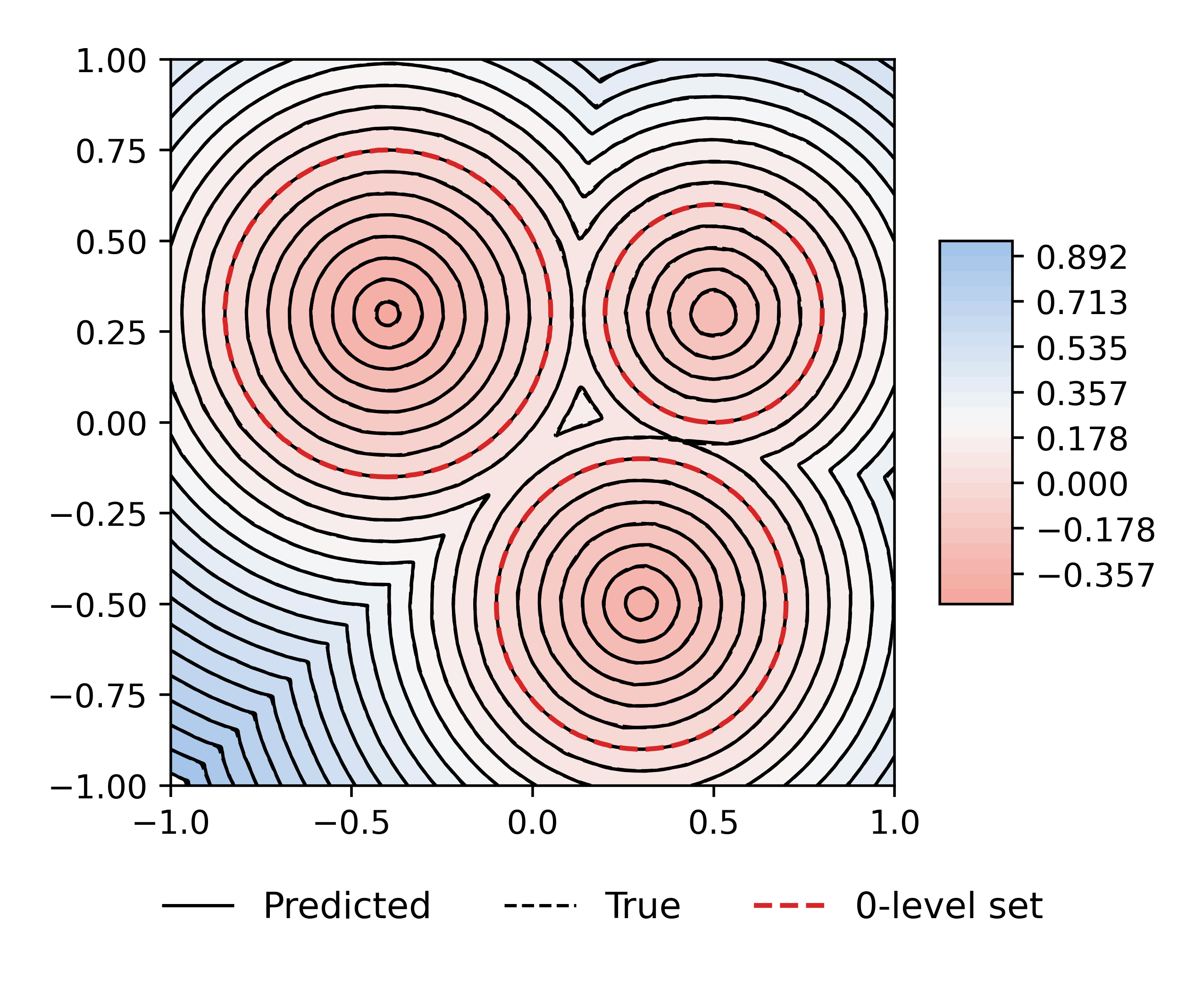} \\
			(c) $\phi_9$ & (d) \textit{ReSDF} with $\phi_9$
		\end{tabular}
	\end{center}
	\caption{The iso-contours of the level set functions $\phi_8$ \eqref{eq:two_cir} and $\phi_9$ \eqref{eq:three_cir} (left) and the results of \textit{ReSDF} with $\phi_8$ and $\phi_9$ (right) are presented with the exact solution. The red solid curves are the zero level set of $\phi_8$ or $\phi_9$.} \label{fig:multcircle}
\end{figure}

We investigate the ability of \textit{ReSDF} to approximate the SDF when the interface is a union of several interfaces. Such a complex mixed interface is relevant to practical problems such as the rising of multiple bubbles in multi-phase flows. The zero level sets of the level set functions $\phi_8$ \eqref{eq:two_cir} or $\phi_9$ \eqref{eq:three_cir} could depict cases of merged bubbles or multiple bubbles in the water, respectively. In Figure \ref{fig:multcircle}, predicted SDFs and the exact solutions are presented. Although the SDF has singularities over the computational domain, \textit{ReSDF} provides accurate interface representations for multiple nested interfaces.

\begin{table}
    \centering
    \setlength\tabcolsep{8pt}
    \caption{The $L^2$ and $L^\infty$ errors between exact and predicted SDF with $\phi_8$ \eqref{eq:two_cir} and $\phi_9$ \eqref{eq:three_cir} are listed. Uniform nodes are fixed by the collocation points $\Omega_7$ and random nodes are sampled from the uniform distribution.} \label{tab:rand}
    \vspace{2pt}
    \scalebox{0.95}{
    \begin{tabular}{ccccc}
    \toprule
    & \multicolumn{2}{c}{Uniform nodes} & \multicolumn{2}{c}{Random nodes} \\ 
    \cmidrule(lr){2-3} \cmidrule(lr){4-5}
    & $\left\Vert u_\theta - u\right\Vert_{2}$ & $\left\Vert u_\theta - u\right\Vert_{\infty}$  & $\left\Vert u_\theta - u\right\Vert_{2}$ & $\left\Vert u_\theta - u\right\Vert_{\infty}$ \\
    \midrule
    $\phi_8$ & $6.16\cdot10^{-4}$ &  $1.43\cdot10^{-2}$ & $6.99\cdot10^{-4}$ & $1.48\cdot 10^{-2}$ \\
    $\phi_9$ &  $6.46\cdot10^{-4}$ &  $1.89\cdot10^{-2}$ & $6.01\cdot10^{-4}$ & $1.98\cdot 10^{-2}$\\
    \bottomrule
    \end{tabular}}
\end{table} 

Another strength of the deep-learning-based approach in \textit{ReSDF} is that scattered collocation points can be used without any methodological modifications. We investigate the effect of the deployment of training collocation points with level set functions $\phi_8$ \eqref{eq:two_cir} or $\phi_9$ \eqref{eq:three_cir}. The uniform collocation points $\Omega_7$ and random points of the same number of $\Omega_7$ sampled from the uniform distribution $\cU\left(-1,1\right)$ are used. Running five times with different random seeds, the average values of the $L^2$ and $L^\infty$ errors are reported in Table \ref{tab:rand}. The results confirm that the errors obtained by different distributions are of the same order of magnitude. In other words, it verifies that the performance of the model is not sensitively tied to the uniform distribution of the training collocation points.


\subsection*{Example 5}\label{subsec:3D}

\begin{table}
    \centering
      \setlength\tabcolsep{10pt}
    \caption{The $L^2$ and $L^\infty$ errors between exact and predicted SDF of \textit{ReSDF} with $\phi_{10}$ \eqref{eq:3d_sphere} are listed. The network depth is fixed to $L=4$.} \label{tab:acc_u_3D}
    \vspace{2pt}
    \scalebox{0.85}{
    \begin{tabular}{ccccccc}
    \toprule
    \multirow{2}{*}{\diagbox[innerwidth=\textwidth*1/7]{$\mid\cD\mid$}{width} }
     & \multicolumn{2}{c}{32} & \multicolumn{2}{c}{64} & \multicolumn{2}{c}{128} \\ 
    \cmidrule(lr){2-3} \cmidrule(lr){4-5} \cmidrule(lr){6-7} 
      & $\left\Vert u_\theta - u\right\Vert_{2}$ & $\left\Vert u_\theta - u\right\Vert_{\infty}$  & $\left\Vert u_\theta - u\right\Vert_{2}$ & $\left\Vert u_\theta - u\right\Vert_{\infty}$  & $\left\Vert u_\theta - u\right\Vert_{2}$ & $\left\Vert u_\theta - u\right\Vert_{\infty}$ \\
    \midrule
    $\Omega^{\text{3d}}_6$  & $2.05\cdot10^{-3}$ & $1.35\cdot 10^{-1}$ & $4.40\cdot 10^{-4}$ & $3.37\cdot 10^{-2}$ & $2.47\cdot 10^{-4}$ & $1.51\cdot 10^{-2}$\\
    $\Omega^{\text{3d}}_7$   & $2.96\cdot10^{-3}$ & $1.60\cdot 10^{-1}$ & $2.46\cdot 10^{-4}$ & $1.56\cdot 10^{-2}$ & $1.58\cdot 10^{-4}$ & $1.40\cdot 10^{-2}$  \\
    \bottomrule
  \end{tabular}}
\end{table} 

\begin{table}
    \centering
      \setlength\tabcolsep{8pt}
    \caption{The $L^2$ and $L^\infty$ errors between exact gradient and the gradient of the predicted SDF of \textit{ReSDF} with $\phi_{11}$ \eqref{eq:3d_two_ellip} on the interface $\Gamma$ are listed. The network depth is fixed to $L=4$.} \label{tab:acc_V_3D}
    \vspace{2pt}
    \scalebox{0.85}{
    \begin{tabular}{ccccccc}
    \toprule
    \multirow{2}{*}{\diagbox[innerwidth=\textwidth*1/7]{$\mid\cD\mid$}{width} }
     & \multicolumn{2}{c}{32} & \multicolumn{2}{c}{64} & \multicolumn{2}{c}{128} \\ 
    \cmidrule(lr){2-3} \cmidrule(lr){4-5} \cmidrule(lr){6-7} 
     & $\left\Vert V_\theta -\bn\right\Vert_{2}^{\Gamma}$ & $\left\Vert V_\theta -\bn\right\Vert_{\infty}^{\Gamma}$ & $\left\Vert V_\theta -\bn\right\Vert_{2}^{\Gamma}$ & $\left\Vert V_\theta -\bn\right\Vert_{\infty}^{\Gamma}$ & $\left\Vert V_\theta -\bn\right\Vert_{2}^{\Gamma}$ & $\left\Vert V_\theta -\bn\right\Vert_{\infty}^{\Gamma}$ \\
    \midrule
    $\Omega^{\text{3d}}_6$  & $1.01\cdot10^{-2}$ & $6.98\cdot10^{-2}$ & $6.27\cdot 10^{-3}$ & $2.33\cdot 10^{-2}$ & $5.02\cdot 10^{-3}$ & $1.84\cdot 10^{-2}$ \\
    $\Omega^{\text{3d}}_7$ & $1.44\cdot10^{-2}$ & $6.42\cdot 10^{-2}$ &  $5.00\cdot 10^{-3}$ & $1.81\cdot 10^{-2}$ & $5.76\cdot 10^{-3}$ & $2.33\cdot 10^{-2}$  \\
    \bottomrule
  \end{tabular}}
\end{table} 

We demonstrate that the proposed method can be scaled into 3D. The number of training points is set as $\Omega^{\text{3d}}_6$ and $\Omega^{\text{3d}}_7$, and the number of neurons is $64$ in each hidden layer with depth $4$. The results of the sphere are listed in Tables \ref{tab:acc_u_3D} and \ref{tab:acc_V_3D}. We can observe that \textit{ReSDF} can still achieve a similar level of accuracy as before in the two-dimensional circular problem. It is important to point out that the number of parameters is similar to that of the two-dimensional case since only the input dimension of the input layer and the output dimension of the last output layer are increased by one. In contrast, the computational cost of classical numerical methods is prohibitively expensive when the dimensionality of the problem increases. Therefore, the numerical results validate the capability of the proposed method for three-dimensional problems without increasing computational complexity.

\begin{figure}
	\begin{center}
		\begin{tabular}{cc}
		\includegraphics[height=0.35\textwidth]{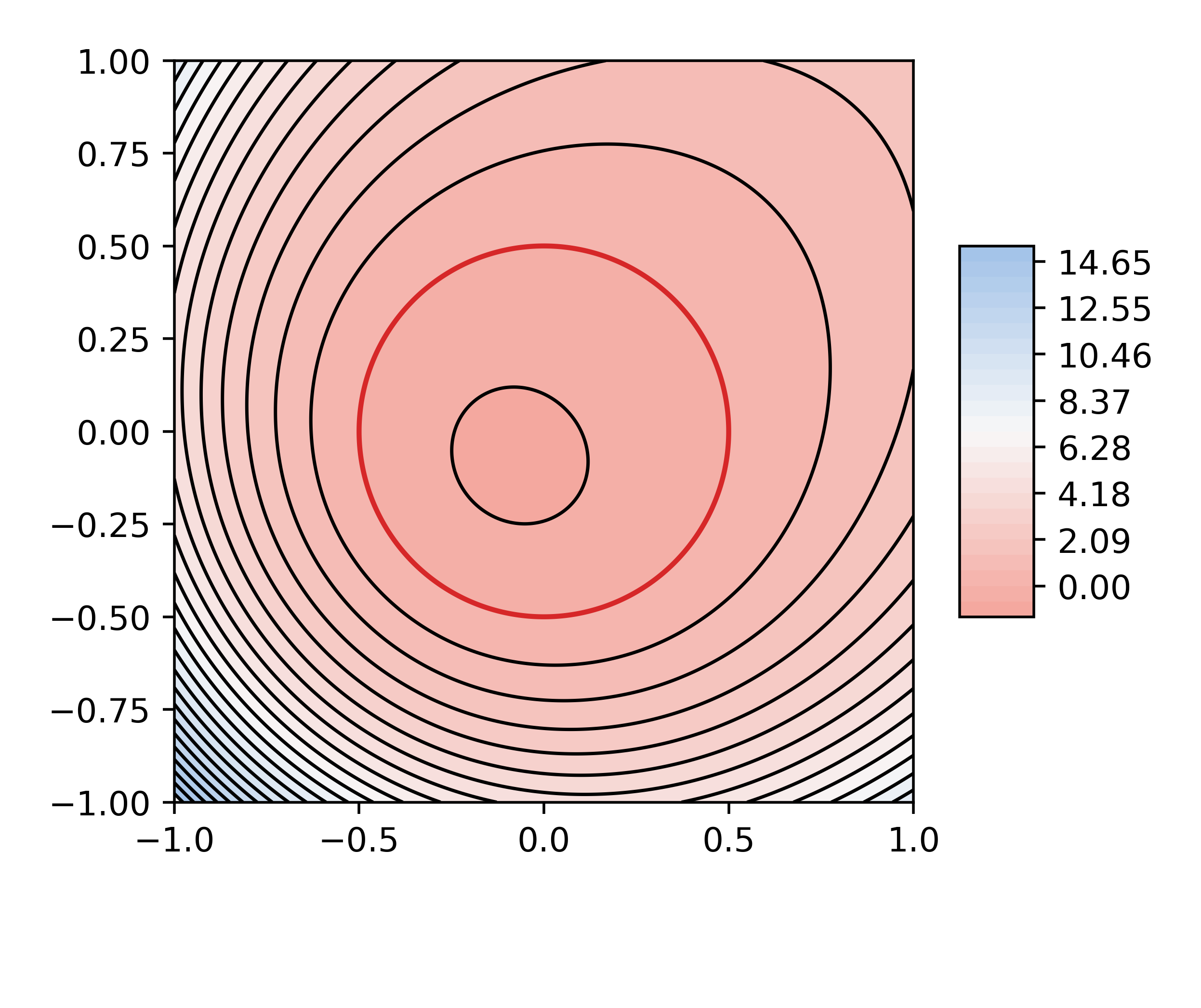} & 
			\includegraphics[height=0.35\textwidth]{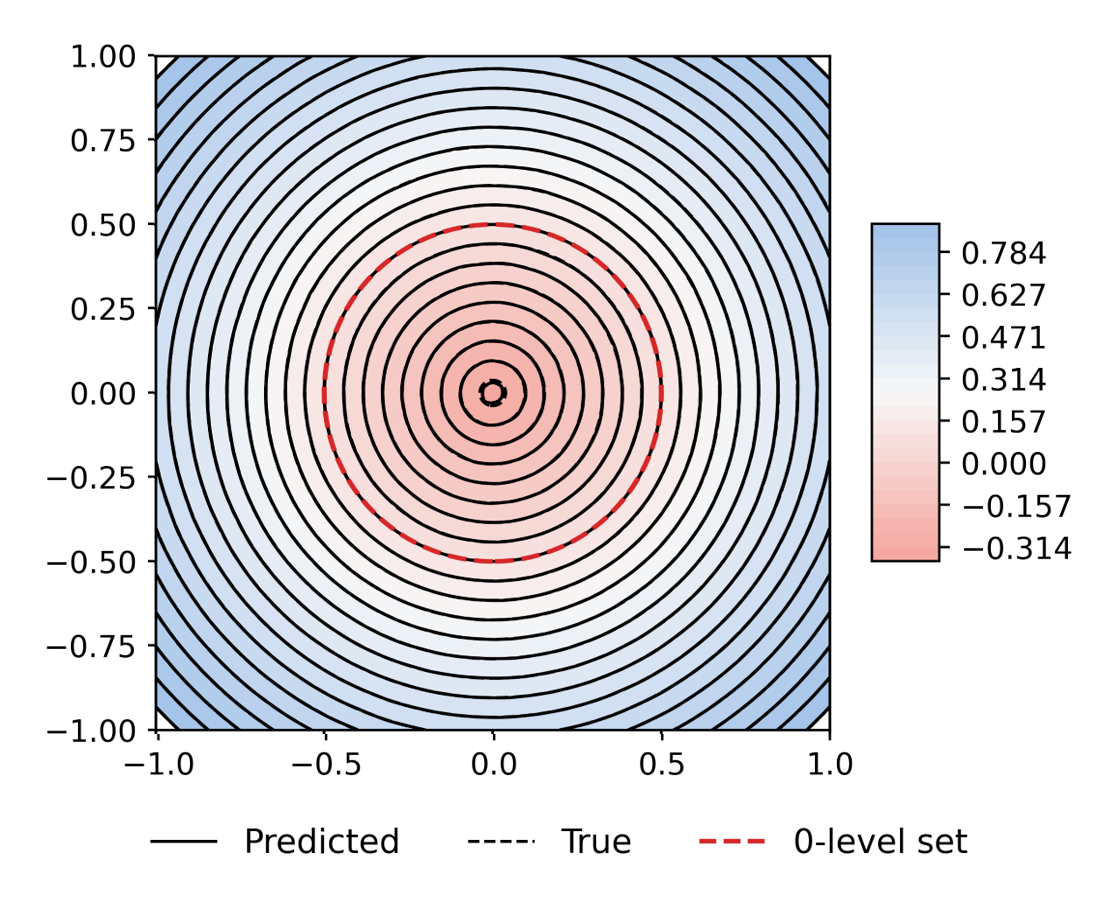} \\
			(a) $\phi_{10}$ on $z=0$ & (b) \textit{ReSDF} with $\phi_{10}$ \\
			\includegraphics[height=0.35\textwidth]{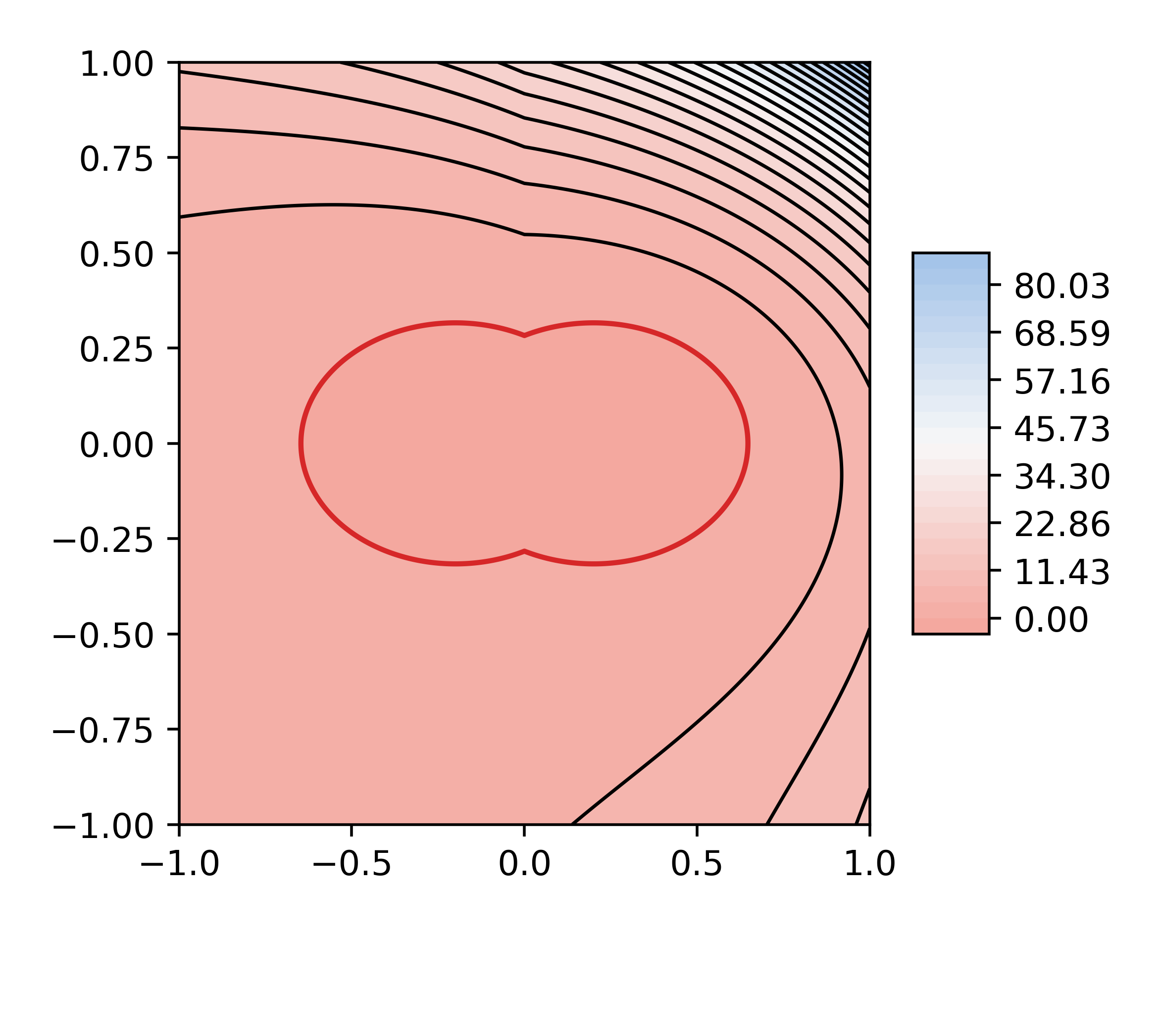} & 
			\includegraphics[height=0.35\textwidth]{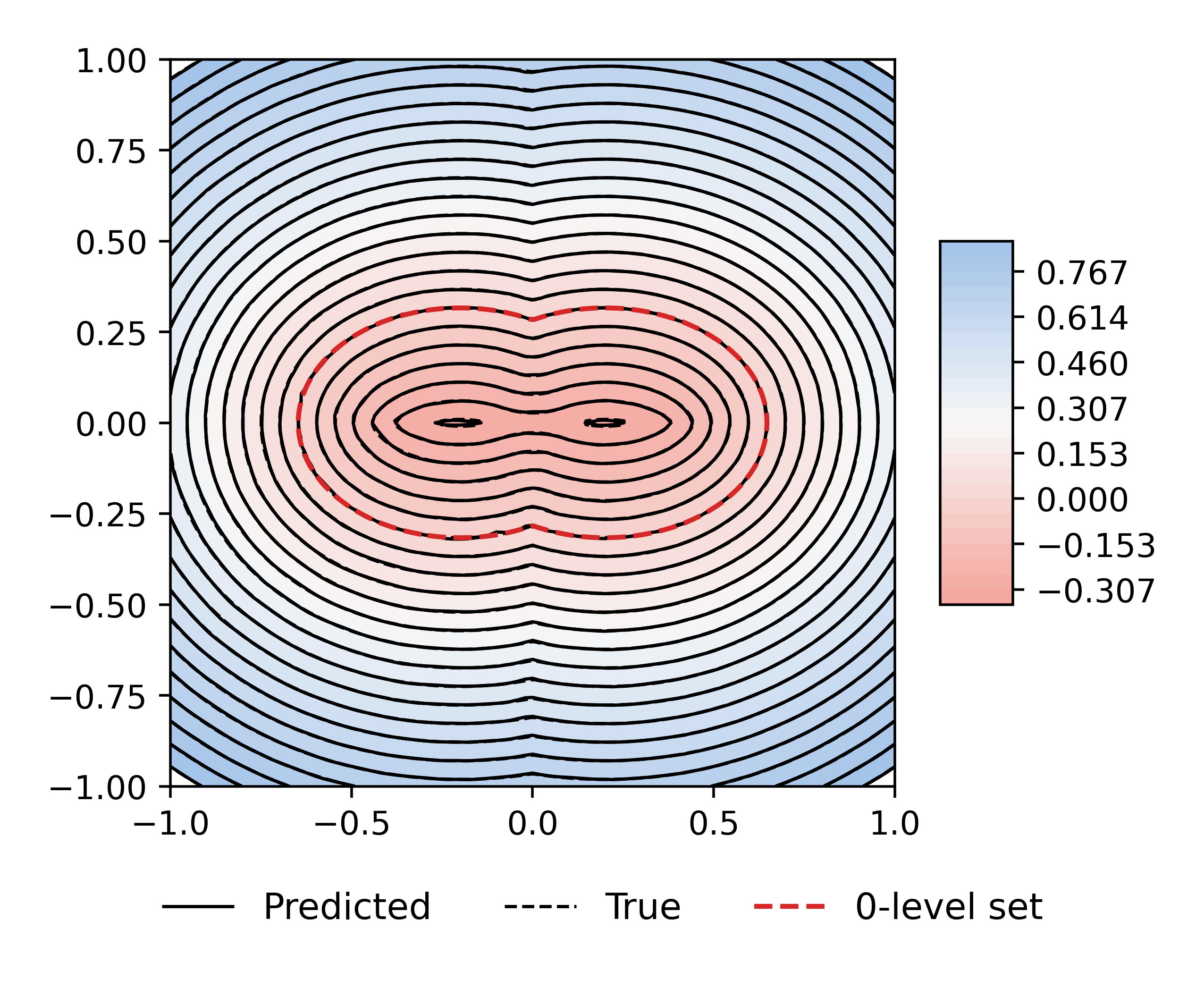} \\
			(c) $\phi_{11}$ on $z=0$  & (d) \textit{ReSDF} with $\phi_{11}$
		\end{tabular}
	\end{center}
	\caption{The iso-contours of the level set functions $\phi_{10}$ \eqref{eq:two_cir} and $\phi_{11}$ \eqref{eq:three_cir} (left) and the results of \textit{ReSDF} with $\phi_{10}$ and $\phi_{11}$ (right) are presented on the $z=0$ plans with the exact solution. The red solid curves are the zero level set of $\phi_{10}$ and $\phi_{11}$.} \label{fig:3d}
\end{figure}

\begin{figure}
	\begin{center}
		\begin{tabular}{cc}
			\includegraphics[height=0.35\textwidth]{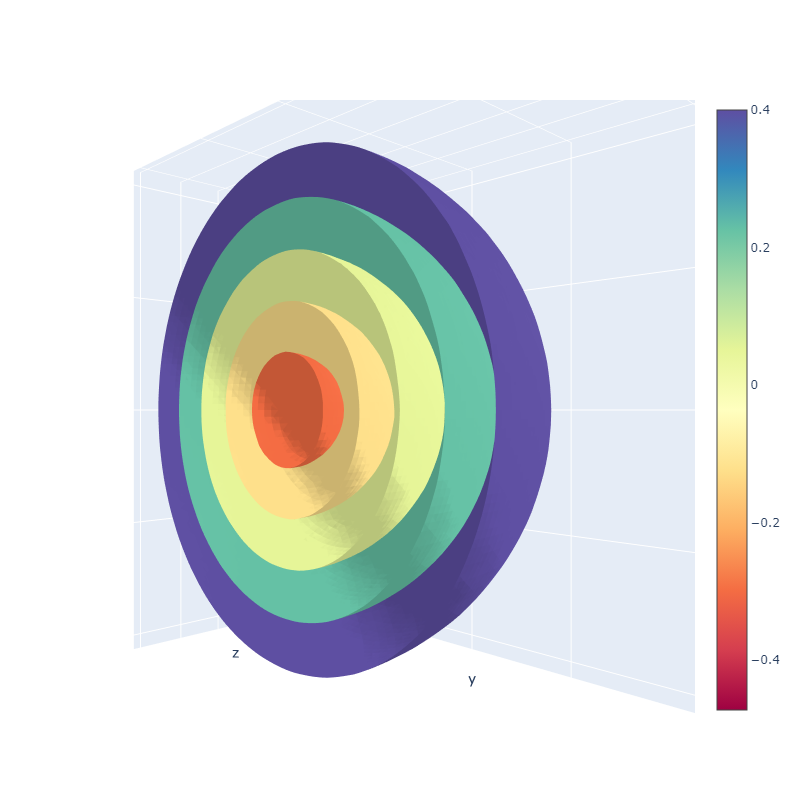} & 
			\includegraphics[height=0.35\textwidth]{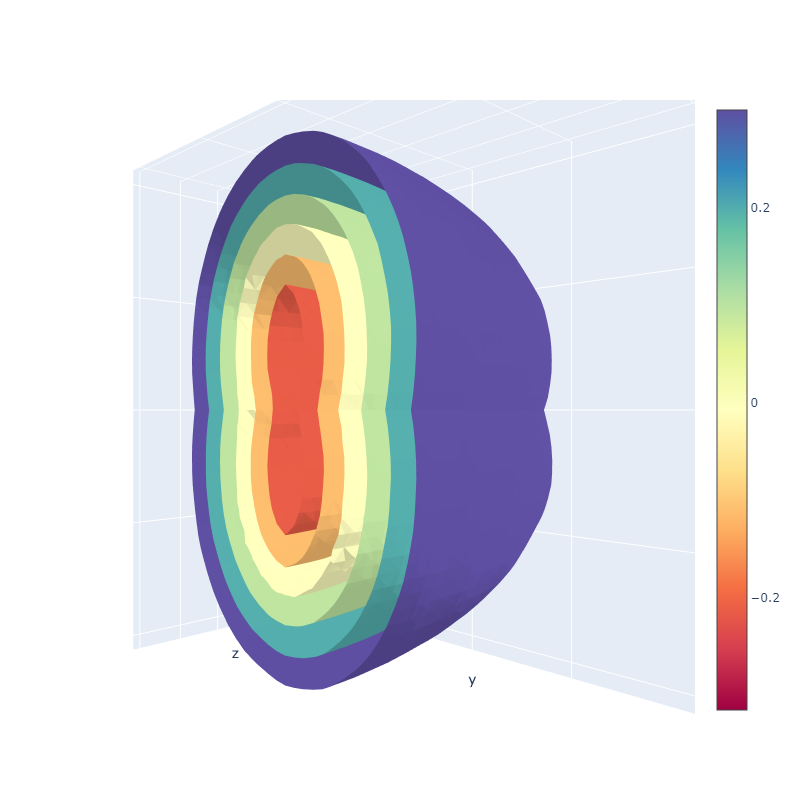} \\
			(a) \textit{ReSDF} with $\phi_{10}$ & (b) \textit{ReSDF} with $\phi_{11}$
		\end{tabular}
	\end{center}
	\caption{Iso-surfaces cut into a section $z=0$ of the predicted SDF with $\phi_{10}$ \eqref{eq:two_cir} (left) and $\phi_{11}$ \eqref{eq:three_cir} (right) are presented.} \label{fig:multi3d_isosurfaces} \label{fig:iso3d}
\end{figure}

In Figures \ref{fig:3d}-(a) and (c), iso-contours of level set functions $\phi_{10}$ \eqref{eq:two_cir} and $\phi_{11}$ \eqref{eq:three_cir} (left) on the $z=0$ plans are presented. 
The results confirm that \textit{ReSDF} produces reliable results also on three-dimensional space.


\section{Conclusions}{\label{sec:conclusion}}

A novel neural approach is proposed to recover the signed distance function of a given hypersurface implicitly represented by the zero contour of a level set function. To enhance the expressive power of the neural network, an auxiliary output is employed to learn the gradient of the SDF in addition to the network output as a neural surrogate of the SDF. Two outputs of the gradient-augmented network are designed to satisfy certain properties as hard constraints. Moreover, underpinned by geometrical properties, we devise a training objective that imposes global properties between the interface and the whole computational domain and alleviates the singularity of the SDF. We confirm that the proposed method produces accurate and robust results without tunable parameter adjustments through various experimental examples ranging from highly distorted or discontinuous level set functions to complex and irregular interfaces. 

\section{Acknowledgements}
This work was supported by the NRF grant [2021R1A2C3010887], the ICT R\&D program of MSIT/IITP[1711117093, 2021-0-00077], and the European Union´s Horizon 2020 Research and Innovation Programme under the Programme SASPRO 2 COFUND Marie Sklodowska-Curie grant agreement No. 945478.

\appendix
\section{Theoretical Justification} \label{appen:analysis}
The following theorem proves that the SDF is the minimizer of $\cL(u,V)$ \eqref{eq:total} except on a set of measure zero.
\begin{theorem}
Let $\Gamma$ be a hypersurface in a domain $\Omega\subset \bR^n$. The optimal solution to the functional \eqref{eq:total} is the SDF to $\Gamma$, except on a set of measure zero.
\end{theorem}
\begin{proof}
Let $u^{\ast}$ be the optimal solution to the functional \eqref{eq:total}.
For any point $\bx\in\Omega$ and a unit vector $\bv\in\bR^n$,
define a function $f_\bV:\bR^{+}\rightarrow\bR$ by 
 \[
 f_{\bv}\left(t\right)\coloneqq u^{\ast}\left(\bx-t\bv\right).
 \]
Clearly, it follows that $f_{\bv}\left(0\right)=u^{\ast}\left(\bx\right)$ and $\left\Vert f'_{\bv}\left(\bx\right)\right\Vert \leq 0$, because
 \[
 \left\Vert f'_{\bv}\left(t\right)\right\Vert=\left\Vert \bv \cdot \nabla u^{\ast}\left(\bx-t\bv\right)\right\Vert\leq \left\Vert \bv \right\Vert\cdot\left\Vert \nabla u^{\ast}\left(\bx-t\bv\right)\right\Vert\leq 1,\ \forall t.
 \]
 Suppose $f_\bv\left(T\right)=0$ for some $\bv\in S^{n-1}$ and $T\in\bR^{+}$. Then, we have
 \[
 \left\vert \frac{f_\bv\left(T\right) - f_\bv\left(0\right)}{T-0}\right\vert \leq 1,
 \]
which can be reformulated as
\begin{equation}\label{eq:lem}
\left\vert f_\bv\left(0\right)\right\vert = \left\vert u^{\ast}\left(\bx\right) \right\vert \leq \mid T \mid.
\end{equation}
Recall that for the distance function $d\left(\cdot,\Gamma\right)$, there exists a unit vector $\bv \in S^{n-1}$ such that 
\[
\bx- d\left(\bx,\Gamma\right)\cdot \bv \in \Gamma,
\]
which is identical to
\[
f_\bv\left(d\left(\bx,\Gamma\right)\right)=0.
\]
The inequality we deduced in \eqref{eq:lem} leads to 
\[
\left\vert u^{\ast}\left(\bx\right)\right\vert \leq d\left(x,\Gamma\right).
\]
However, because $u$ is the optimal solution, it satisfies $\bx-u^{\ast}\left(\bx\right)\nabla u^{\ast}\left(\bx\right)\in\Gamma$ for points $\bx\in\Omega$ at which $u^\ast$ is differentiable. Therefore, it can be written as $f_{\nabla u^{\ast}\left(\bx\right)}\left(u^{\ast}\left(\bx\right)\right)=0$, and hence
\[
d\left(\bx,\Gamma\right)\leq \left\vert u^{\ast}\left(\bx\right)\right\vert.
\]
Combining both inequalities, we see that 
\[
\left\vert u^{\ast}\left(\bx\right) \right\vert = d\left(\bx,\Gamma\right),
\]
concluding the proof.
\end{proof}

Note that according to the proof we can obtain the SDF without the last term in \eqref{eq:total} where the SDF is differentiable. In practice, we employ $\cL_{RS}$ to facilitate training at singular points discussed in Section \ref{sec:objective}.
 
\section{Proof of Proposition \ref{prop:HJB}}\label{appen:prop2}
\begin{proof}
For a given point $\bx\in\Omega$, let us define a function $g:\bR^+\rightarrow\bR$ as
\[
g_\bx\left(t\right)=u\left(\bx-t\nabla u\left(\bx\right)\right).
\]
It satisfies 
\begin{equation} \label{eq:tmp1}
    \begin{cases}
        g_\bx\left(0\right)=u\left(\bx\right), \\
        g_\bx\left(u\left(\bx\right)\right) =u\left(\bx-u\left(\bx\right)\nabla u\left(\bx\right)\right)=0,
    \end{cases}
\end{equation}
where the last equation is deduced from the property \eqref{eq:condition1}.
Also, since the signed distance function has a gradient with unit norm, we get
\begin{equation} \label{eq:tmp2}
\left\vert g_\bx'\left(t\right)\right\vert = \left\vert \nabla u\left(\bx\right)\cdot \nabla u\left(\bx-t\nabla u\left(\bx\right)\right) \right\vert \leq 1.
\end{equation}
From \eqref{eq:tmp1} and \eqref{eq:tmp2}, it follows that 
\[
\left\vert g_\bx '\left(t\right)\right\vert=1,
\]
for $t$ between $0$ to $u\left(\bx\right)$.
This implies that $\nabla u\left(\bx\right)$ and $\nabla u\left(\bx- t\nabla\left(\bx\right)\right)$ are parallel.
This concludes 
    \begin{equation}\label{eq:condition2}
    \nabla u\left(\bx\right)=\nabla u\left(\bx - t\sgn\left(u\left(\bx\right)\right)\nabla u\left(\bx\right)\right), \forall\ t\in\left[0,\left\vert u\left(\bx\right)\right\vert\right).
\end{equation}
\end{proof}

\bibliographystyle{elsarticle-num-names}
\bibliography{mybib}







\end{document}